\documentclass[11pt]{amsart}
\pdfoutput=1

\usepackage[utf8]{inputenc}
\usepackage[english]{babel}
\usepackage{amssymb}
\usepackage{stmaryrd}
\usepackage{xcolor}
\usepackage{microtype}
\usepackage{mathtools}
\usepackage{enumitem}

\usepackage[margin=1in, marginparwidth=.75in]{geometry}

\usepackage[numbers,compress]{natbib}

\usepackage[bookmarksdepth=3]{hyperref}
\hypersetup{
  pdftitle={Multisymplecticity in finite element exterior calculus},
  pdfauthor={Ari Stern and Enrico Zampa},
  pdfsubject={MSC 2020: 65N30 (Primary), 70S05 (Secondary)},
  bookmarksopen=false
}

\usepackage[capitalize,nameinlink]{cleveref}

\theoremstyle{plain} 
\newtheorem{theorem}             {Theorem}  [section]
\newtheorem{lemma}      [theorem]{Lemma}
\newtheorem{corollary}  [theorem]{Corollary}
\newtheorem{proposition}[theorem]{Proposition}

\theoremstyle{definition}
\newtheorem{definition} [theorem]{Definition}
\newtheorem{example}    [theorem]{Example}

\theoremstyle{remark}
\newtheorem{remark}     [theorem]{Remark}

\AddToHook{env/lemma/begin}{\crefalias{theorem}{lemma}}
\AddToHook{env/corollary/begin}{\crefalias{theorem}{corollary}}
\AddToHook{env/proposition/begin}{\crefalias{theorem}{proposition}}
\AddToHook{env/definition/begin}{\crefalias{theorem}{definition}}
\AddToHook{env/example/begin}{\crefalias{theorem}{example}}
\AddToHook{env/remark/begin}{\crefalias{theorem}{remark}}

\newcommand{\ringhat}[1]{\mathchoice
  {\smash{\mathring{\widehat{#1}}}}%
  {\vphantom{{\mathring{\widehat{#1}}}}\smash{\mathring{\widehat{#1}}}}%
  {\vphantom{{\mathring{\widehat{#1}}}}\smash{\mathring{\widehat{#1}}}}%
  {\vphantom{{\mathring{\widehat{#1}}}}\smash{\mathring{\widehat{#1}}}}%
}

\newcommand{\av}[1]{\{\!\!\{#1\}\!\!\}}
\newcommand{\bigav}[1]{\bigl\{\!\!\bigl\{#1\bigr\}\!\!\bigr\}}
\newcommand{\Bigav}[1]{\Bigl\{\!\!\Bigl\{#1\Bigr\}\!\!\Bigr\}}
\newcommand{\jump}[1]{\llbracket #1 \rrbracket}

\DeclareMathSymbol{\star}{\mathord}{letters}{"3F}

\title{Multisymplecticity in finite element exterior calculus}
\author{Ari Stern}
\address{Department of Mathematics, Washington University in St.~Louis}
\email{stern@wustl.edu}

\author{Enrico Zampa}
\address{Department of Mathematics, University of Trento}
\curraddr{Department of Mathematics, University of Vienna}
\email{enrico.zampa@univie.ac.at}

\begin{document}

\begin{abstract}
  We consider the application of finite element exterior calculus
  (FEEC) methods to a class of canonical Hamiltonian PDE systems
  involving differential forms. Solutions to these systems satisfy a
  local \emph{multisymplectic conservation law}, which generalizes the
  more familiar symplectic conservation law for Hamiltonian systems of
  ODEs, and which is connected with physically-important reciprocity
  phenomena, such as Lorentz reciprocity in electromagnetics. We
  characterize hybrid FEEC methods whose numerical traces satisfy a
  version of the multisymplectic conservation law, and we apply this
  characterization to several specific classes of FEEC methods,
  including conforming Arnold--Falk--Winther-type methods and various
  hybridizable discontinuous Galerkin (HDG) methods. Interestingly,
  the HDG-type and other nonconforming methods are shown, in general,
  to be multisymplectic in a stronger sense than the conforming FEEC
  methods. This substantially generalizes previous work of McLachlan
  and Stern [Found.\ Comput.\ Math., 20 (2020), pp.~35--69] on the
  more restricted class of canonical Hamiltonian PDEs in the
  de~Donder--Weyl ``grad-div'' form.
\end{abstract}

\maketitle

\section{Introduction}

\subsection{Motivation and background}

In classical mechanics, a canonical Hamiltonian system of ODEs has the
form
\begin{equation}
  \label{eq:hamiltonian_ODE}
  \begin{bmatrix}
    & - \mathrm{d} / \mathrm{d} t \\
    \mathrm{d} / \mathrm{d} t  &
  \end{bmatrix}
  \begin{bmatrix}
    q \\
    p
  \end{bmatrix} =
  \begin{bmatrix}
    \partial H / \partial q \\
    \partial H / \partial p
  \end{bmatrix},
\end{equation}
where the function $ H = H ( t, q, p ) $ is called the Hamiltonian of
the system. In 1935, \citet{deDonder1935} and \citet{Weyl1935}
extended the canonical Hamiltonian framework to systems of PDEs having
the form
\begin{equation}
  \label{eq:ddw}
  \begin{bmatrix}
    & - \operatorname{div} \\
    \operatorname{grad} &
  \end{bmatrix}
  \begin{bmatrix}
    u \\
    \sigma
  \end{bmatrix} =
  \begin{bmatrix}
    \partial H / \partial u \\
    \partial H / \partial \sigma
  \end{bmatrix},
\end{equation}
where $ H = H ( x, u, \sigma ) $. The de~Donder--Weyl theory includes
many types of PDEs arising in classical field theory, including both
Laplace-type operators and, if $ \operatorname{grad} $ and
$ \operatorname{div} $ are taken in a spacetime sense, wave-type
operators. In 2006, \citet{Bridges2006} extended this theory further,
considering systems of PDEs having the form
\begin{equation}
  \label{eq:bridges}
  \begin{bmatrix}
    & \delta ^1 \\
    \mathrm{d} ^0 & & \ddots  & \\
    & \ddots   & & \delta ^n  \\
    & & \mathrm{d} ^{ n -1 }  
  \end{bmatrix}
  \begin{bmatrix}
    z ^0 \\
    z ^1 \\
    \vdots\\
    z ^n 
  \end{bmatrix} =
  \begin{bmatrix}
    \partial H / \partial z ^0 \\
    \partial H / \partial z ^1 \\
    \vdots \\
    \partial H / \partial z ^n
  \end{bmatrix},
\end{equation}
for $ H = H ( x, z ) $, where each $ z ^k $ is a differential
$k$-form, $ \mathrm{d} ^k $ and $ \delta ^k $ are the exterior
differential and codifferential on $k$-forms, and the domain is an
$n$-dimensional pseudo-Riemannian manifold. For example, on a domain
in $\mathbb{R}^3$ equipped with the Euclidean metric, we can identify
$0$- and $3$-forms with scalar fields and $1$- and $2$-forms with
vector fields, allowing \eqref{eq:bridges} to be written as
\begin{equation*}
  \begin{bmatrix}
     & - \operatorname{div} & \\
    \operatorname{grad} & & \operatorname{curl} \\
    & \operatorname{curl} & & - \operatorname{grad} \\
    & & \operatorname{div}
  \end{bmatrix}
  \begin{bmatrix}
    z ^0 \\
    z ^1 \\
    z ^2 \\
    z ^3 
  \end{bmatrix}
  =
  \begin{bmatrix}
    \partial H / \partial z ^0 \\
    \partial H / \partial z ^1 \\
    \partial H / \partial z ^2 \\
    \partial H / \partial z ^3 
  \end{bmatrix}.
\end{equation*}
The inclusion of the curl operator makes this a natural framework for
expressing the Hamiltonian structure of, for example, Maxwell's
equations.

Each of the systems \eqref{eq:hamiltonian_ODE}, \eqref{eq:ddw},
\eqref{eq:bridges} has a physically-important local conservation law
that one may wish for a numerical method to preserve exactly. For
Hamiltonian ODEs in the form \eqref{eq:hamiltonian_ODE}, the
\emph{symplectic conservation law} is preserved by \emph{symplectic
  integrators}, whose numerical advantages for such problems have been
well studied \citep{SaCa1994,LeRe2004,HaLuWa2006}. For Hamiltonian
PDEs in the canonical de~Donder--Weyl form \eqref{eq:ddw} or Bridges
form \eqref{eq:bridges}, there is a \emph{multisymplectic conservation
  law} whose preservation by numerical methods has been the subject of
considerable investigation, inspired by the success of symplectic
integrators. Much of the initial progress on multisymplectic methods
focused on tensor products of symplectic integrators on rectangular
grids or, alternatively, on relatively low-order finite difference and
finite volume methods on unstructured meshes. In particular, we
mention the seminal work of Marsden and collaborators
\citep{MaPaSh1998,MaSh1999,MaPeShWe2001,LeMaOrWe2003} and Reich and
collaborators
\citep{Reich2000a,Reich2000b,BrRe2001,MoRe2003,FrMoRe2006} in the late
1990s and 2000s. There had been some work around this time on finite
element methods, but primarily as an approach for constructing
multisymplectic finite difference stencils for special classes of
problems and meshes \citep{GuJiLiWu2001,ZhBaLiWu2003,Chen2008}.

In 2020, \citet{McSt2020} developed a general theory of
multisymplectic hybrid finite element methods for canonical
Hamiltonian PDEs in the de~Donder--Weyl form \eqref{eq:ddw}, based on
the unified hybridization framework of \citet*{CoGoLa2009}. The
results of \citep{McSt2020} show that the hybridized versions of
several conforming, mixed, nonconforming, and hybridizable
discontinuous Galerkin (HDG) methods are multisymplectic---that is,
their numerical traces/fluxes satisfy a multisymplectic conservation
law---giving several classes of arbitrarily high-order multisymplectic
methods on unstructured meshes. Interestingly, the classic
``continuous Galerkin'' Lagrange finite element method is shown to
satisfy a weaker version of the multisymplectic conservation law than
the other methods considered, due to the fact that its numerical flux
is only weakly rather than strongly conservative. The use of these
methods was more recently extended to semidiscretization of
time-dependent de~Donder--Weyl systems in \citep{McSt2024}.

In this paper, we examine multisymplecticity of hybrid finite element
methods for the more general class of systems \eqref{eq:bridges},
using techniques from \emph{finite element exterior calculus} (FEEC)
\citep{ArFaWi2006,ArFaWi2010,Arnold2018}. The approach is based on the
recent hybridization framework for FEEC developed by
\citet*{AwFaGuSt2023}. We show that several hybridized FEEC
methods---including the conforming methods of
\citet*{ArFaWi2006,ArFaWi2010}, as well as various HDG-type
methods---have numerical traces satisfying a multisymplectic
conservation law. The main results of \citep{McSt2020} are recovered
as a special case of those presented here---and, as in
\citep{McSt2020}, the conforming methods are multisymplectic in a
weaker sense than the HDG-type and other nonconforming methods.

In \citep{McSt2020}, the multisymplectic conservation law for
de~Donder--Weyl systems was linked to physically-important
``reciprocity'' phenomena, such as Green's reciprocity in
electrostatics, that multisymplectic methods preserve numerically. In
the more general setting considered here, we link multisymplecticity
to other reciprocity phenomena, including \emph{Lorentz reciprocity}
in electromagnetics. As \citet*{RyInBo2013} observe in their text
\emph{Computational Electromagnetics}:
\begin{quotation}
  \emph{The symmetric, or reciprocal, property of the FEM appears not to
  hold for finite volume discretizations. In fact, lack of symmetry is
  a likely explanation of the late-time instability observed for many
  schemes.}
\end{quotation}
Numerical reciprocity (i.e., multisymplecticity) is therefore not
merely an interesting theoretical property; it is also a feature of
finite element methods believed to be important to practitioners.

\subsection{Organization of the paper}

The paper is organized as follows:
\begin{itemize}
\item \Cref{sec:multisymplectic} introduces a class of canonical
  systems of PDEs involving differential forms. Among these, we show
  that canonical Hamiltonian systems \eqref{eq:bridges} satisfy a
  local multisymplectic conservation law, which we connect with
  reciprocity phenomena. The theory is illustrated with several
  examples involving Hodge--Dirac and Hodge--Laplace operators, Stokes
  flow in vorticity-velocity-pressure form, and Maxwell's equations.

\item \Cref{sec:multisymplectic_hybrid} develops a unified
  hybridization framework for FEEC methods approximating the canonical
  systems of the previous section. Within this framework, we develop
  criteria characterizing multisymplectic methods, i.e., those whose
  numerical traces satisfy a multisymplectic conservation law when
  applied to Hamiltonian systems. Both ``weak'' and ``strong'' notions
  of multisymplecticity are considered, and these are linked to
  weak/strong conservativity of numerical traces.

\item \Cref{sec:particular_methods} applies the general results of
  \cref{sec:multisymplectic_hybrid} to study the multisymplecticity of
  several particular classes of methods, including the conforming
  methods of \citet*{ArFaWi2006,ArFaWi2010} (using the hybridization
  of \citep{AwFaGuSt2023}), various HDG methods (including LDG-H and
  IP-H-type methods and those making use of reduced stabilization
  \citep{Lehrenfeld2010,LeSc2016,Oikawa2015,Oikawa2016}), and
  primal-hybrid nonconforming methods. Several previously-proposed HDG
  methods for vector problems in $ \mathbb{R}^2 $ and $ \mathbb{R}^3 $
  are shown to be among the multisymplectic methods.

\end{itemize}
  
\subsection{Acknowledgments} Ari Stern was supported by the National
Science Foundation (Grant No.\ DMS-2208551). Enrico Zampa acknowledges PhD grant funding from the University of Trento (UNITN).

\section{Multisymplectic PDEs and the Hodge--Dirac operator}

\label{sec:multisymplectic}

\subsection{Canonical systems of PDEs: smooth theory}

We begin by describing a class of canonical systems of PDEs that
includes the Hamiltonian systems \eqref{eq:bridges} considered by
\citet{Bridges2006}. For simplicity of exposition, and because it
contains the cases of numerical interest to us, we restrict attention
to domains $ \Omega \subset \mathbb{R}^n $ with the Euclidean metric.
(See \cref{rmk:manifold} for a brief discussion of the case where
$\Omega$ is a manifold.)

We first quickly recall some preliminaries from exterior algebra and
exterior calculus on $\Omega$. For an in-depth treatment of this
subject motivated by numerical PDEs on
$ \Omega \subset \mathbb{R}^n $, see
\citet*{ArFaWi2006,ArFaWi2010,Arnold2018}; for a more general approach
on manifolds, see \citet*{AbMaRa1988,Lee2013}.

\begin{definition}
  Let $ \operatorname{Alt} ^k \mathbb{R}^n $ denote the space of
  alternating (i.e., totally antisymmetric) $k$-linear forms
  $ \mathbb{R}^n \times \cdots \times \mathbb{R}^n \rightarrow
  \mathbb{R} $, and let
  $ \operatorname{Alt} \mathbb{R}^n \coloneqq \bigoplus _{ k = 0 } ^n
  \operatorname{Alt} ^k \mathbb{R}^n $.
\end{definition}

We recall some standard algebraic operations on
$ \operatorname{Alt} \mathbb{R}^n $. The \emph{wedge} (or
\emph{exterior}) \emph{product}
\begin{equation*}
  \wedge \colon \operatorname{Alt} ^k \mathbb{R}^n \times \operatorname{Alt} ^\ell \mathbb{R}^n \rightarrow \operatorname{Alt} ^{ k + \ell } \mathbb{R}^n
\end{equation*}
gives $ ( \operatorname{Alt} \mathbb{R}^n , \wedge ) $ the structure
of an associative algebra, called the \emph{exterior algebra} on
$\mathbb{R}^n$. Next, the Euclidean inner product on $\mathbb{R}^n$
naturally induces an inner product $ ( \cdot , \cdot ) $ on
$ \operatorname{Alt} \mathbb{R}^n $. This defines the \emph{Hodge
  star} operator
$\star \colon \operatorname{Alt} ^k \mathbb{R}^n \rightarrow
\operatorname{Alt} ^{ n - k } \mathbb{R}^n $ by the condition
\begin{equation*}
  v \wedge \star w = ( v , w ) \operatorname{vol}, \qquad v, w \in \operatorname{Alt} ^k \mathbb{R}^n ,
\end{equation*}
where $ \operatorname{vol} \in \operatorname{Alt} ^n \mathbb{R}^n $ is
the Euclidean volume form (i.e., the determinant). This is an
isometric isomorphism for each $k$ and thus extends to an isometric
automorphism
$ \star \colon \operatorname{Alt} \mathbb{R}^n \rightarrow
\operatorname{Alt} \mathbb{R}^n $.

\begin{definition}
  Let $\Lambda ^k (\Omega) $ be the space of \emph{smooth differential
    $k$-forms} on $\Omega \subset \mathbb{R}^n $, i.e., smooth maps
  $ \Omega \rightarrow \operatorname{Alt} ^k \mathbb{R} ^n $. Denote
  $ \Lambda (\Omega) \coloneqq \bigoplus _{ k = 0 } ^n \Lambda ^k
  (\Omega) $, whose elements are smooth maps
  $ \Omega \rightarrow \operatorname{Alt} \mathbb{R}^n $.
\end{definition}

The wedge product and Hodge star extend to $ \Lambda (\Omega) $ by
applying them pointwise at each $ x \in \Omega $. We also recall the
\emph{exterior differential}
$ \mathrm{d} ^k \colon \Lambda ^k (\Omega) \rightarrow \Lambda ^{ k +
  1 } (\Omega) $ and \emph{codifferential}
$ \delta ^k \colon \Lambda ^k (\Omega) \rightarrow \Lambda ^{ k -1 }
(\Omega) $. The codifferential is defined by
$ \delta ^k \coloneqq ( - 1 ) ^k \star ^{-1} \mathrm{d} ^{ n - k }
\star $, which implies the useful identity
\begin{equation}
  \label{eq:d_delta_identity}
  \mathrm{d} ( \tau \wedge \star v ) = \mathrm{d} \tau \wedge \star v - \tau \wedge \star \delta v , \qquad \tau \in \Lambda ^{ k -1 } (\Omega) ,\ v \in \Lambda ^k (\Omega) .
\end{equation}
We take
\begin{equation*}
  \mathrm{d} \coloneqq \bigoplus _{ k = 0 } ^n \mathrm{d} ^k , \qquad \delta \coloneqq \bigoplus _{ k = 0 } ^n \delta ^k  ,
\end{equation*}
to be operators on $\Lambda (\Omega) $, where $ \mathrm{d} $ is
$ ( + 1 ) $-graded and $ \delta $ is $ ( -1 ) $-graded.

\begin{definition}
  The \emph{Hodge--Dirac} (sometimes \emph{Hodge--de~Rham}) operator
  on $ \Lambda (\Omega) $ is
  $ \mathrm{D} \coloneqq \mathrm{d} + \delta $.
\end{definition}

\begin{remark}
  \citet{Bridges2006} calls $ \mathrm{D} $ the ``multi-symplectic
  Dirac operator'' and denotes it by
  $ \boldsymbol{ J _{ \partial } } $.  Unlike $ \mathrm{d} $ and
  $ \delta $, the operator $ \mathrm{D} $ does not have an integer
  grading. However, it is sometimes treated as
  $ \mathbb{Z} _2 $-graded, corresponding to even and odd
  $k$. Furthermore, since $ \mathrm{d} \mathrm{d} = 0 $ and
  $ \delta \delta = 0 $, we have
  $ \mathrm{D} \mathrm{D} = \mathrm{d} \delta + \delta \mathrm{d} $,
  which is the ($0$-graded) \emph{Hodge--Laplace operator}.
\end{remark}

We will consider systems of PDEs in the canonical form
\begin{equation}
  \label{eq:canonical}
  \mathrm{D} z (x) = f \bigl( x , z (x) \bigr),  \quad x \in  \Omega ,
\end{equation}
where $ z \in \Lambda (\Omega) $ and
$ f \colon \Omega \times \operatorname{Alt} \mathbb{R}^n \rightarrow
\operatorname{Alt} \mathbb{R}^n $ is a given (possibly nonlinear)
function. In particular, we say that \eqref{eq:canonical} is a
\emph{canonical Hamiltonian system} if $ f = \partial H / \partial z $
for some
$ H \colon \Omega \times \operatorname{Alt} \mathbb{R}^n \rightarrow
\mathbb{R} $. Observe that this is precisely the form of
\eqref{eq:bridges}, where $ \mathrm{D} $ corresponds to the
left-hand-side matrix.

\begin{remark}
  \label{rmk:manifold}
  More generally, if $\Omega$ is an $n$-dimensional oriented
  pseudo-Riemannian manifold, the space $ \Lambda (\Omega) $ consists
  of smooth sections of the \emph{total exterior algebra bundle}
  $ \operatorname{Alt} T \Omega $, whose fiber at $ x \in \Omega $ is
  $ \operatorname{Alt} T _x \Omega $.  (In \citet{Bridges2006}, the
  total exterior algebra bundle is written using the equivalent
  notation $ \bigwedge T ^\ast \Omega $, whose fiber at
  $ x \in \Omega $ is $ \bigwedge T ^\ast _x \Omega $.)  The canonical
  system \eqref{eq:canonical} can be defined for any (possibly
  nonlinear) bundle map
  $ f \colon \operatorname{Alt} T \Omega \rightarrow
  \operatorname{Alt} T \Omega $, where
  $ f ( x, \cdot ) \colon \operatorname{Alt} T _x \Omega \rightarrow
  \operatorname{Alt} T _x \Omega $. In particular, a canonical
  Hamiltonian system can be defined for any
  $ H \colon \operatorname{Alt} T \Omega \rightarrow \mathbb{R} $.
\end{remark}

\begin{example}
  \label{ex:1D}
  In the case $ n = 1 $, we can write
  $ \mathrm{D} = J \frac{\mathrm{d}}{\mathrm{d}t} $, where
  $ J = \bigl[ \begin{smallmatrix}
    & - 1 \\
    1 &
  \end{smallmatrix} \bigr] $ is the canonical symplectic matrix; this
  $ \mathrm{D} $ is precisely the matrix on the left-hand side of
  \eqref{eq:hamiltonian_ODE}. Thus, \eqref{eq:canonical} becomes
  \begin{equation*}
    J \dot{ z } (t) = f \bigl(  t, z (t) \bigr) .
  \end{equation*}
  When $ f = \partial H / \partial z $, this is the canonical
  Hamiltonian system of ODEs \eqref{eq:hamiltonian_ODE} with
  $ z = q \oplus p \,\mathrm{d}t $.
\end{example}
  
\begin{example}
  If $ f = f (x) $, then \eqref{eq:canonical} corresponds to the
  linear Hodge--Dirac problem
  \begin{equation*}
    \mathrm{D} z (x) = f (x) ,
  \end{equation*}
  which is a canonical Hamiltonian system with
  $ H ( x, z ) = \bigl( f (x) , z \bigr) $.
\end{example}

\begin{example}[semilinear Hodge--Laplace problem]
  \label{ex:laplace}
  Given a potential function
  $ F \colon \Omega \times \operatorname{Alt} ^k \mathbb{R}^n
  \rightarrow \mathbb{R} $, suppose that $ u \in \Lambda ^k (\Omega) $
  is a solution to the semilinear Hodge--Laplace problem
  \begin{equation*}
    \mathrm{d} \delta u + \delta \mathrm{d} u = \frac{ \partial F }{ \partial u }  .
  \end{equation*}
  The linear Hodge--Laplace problem is the special case
  $ F ( x, u ) = \bigl( f ^k (x) , u \bigr) $ for some
  $ f ^k \in \Lambda ^k (\Omega) $.  Introducing variables
  $ \sigma = \delta u \in \Lambda ^{ k -1 } (\Omega) $ and
  $ \rho = \mathrm{d} u \in \Lambda ^{ k + 1 } (\Omega) $ lets us
  write this in first-order form,
  \begin{align*}
    \delta u &= \sigma ,\\
    \mathrm{d} \sigma + \delta \rho &= \frac{ \partial F }{ \partial u } ,\\
    \mathrm{d} u &= \rho .
  \end{align*}
  Thus, $ z = \sigma \oplus u \oplus \rho $ is a solution to the
  canonical Hamiltonian system with
  \begin{equation*}
    H ( x, z ) = \frac{1}{2} \lvert \sigma \rvert ^2 + F ( x, u ) + \frac{1}{2} \lvert \rho \rvert ^2 .
  \end{equation*}
\end{example}

\begin{example}[vorticity-velocity-pressure problem]
  \label{ex:vvp}
  As a slightly modified version of the previous example, suppose
  instead that $ z = \sigma \oplus u \oplus \rho $ satisfies
  \begin{align*}
    \delta u &= \sigma ,\\
    \mathrm{d} \sigma + \delta \rho &= \frac{ \partial F }{ \partial u } ,\\
    \mathrm{d} u &= 0 ,
  \end{align*}
  which differs only in the right-hand side of the third
  equation. This corresponds to the Hamiltonian
  \begin{equation*}
    H ( x, z ) = \frac{1}{2} \lvert \sigma \rvert ^2 + F ( x, u ) .
  \end{equation*}
  When $ n = 3 $ and $ k = 2 $, the linear case coincides with the
  \emph{vorticity-velocity-pressure} formulation of the Stokes
  problem considered by \citet{Nedelec1982},
  \begin{align*}
    \operatorname{curl} \mathbf{u} &= \mathbf{w} ,\\
    \operatorname{curl} \mathbf{w} + \operatorname{grad} p &= \mathbf{f} ,\\
    \operatorname{div} \mathbf{u} &= 0 .
  \end{align*}
  This same formulation has also been applied (e.g., in
  \citep{ChQiShSo2017,ChQiSh2018}) to the Maxwell-type problem
  \begin{align*}
    \operatorname{curl} \operatorname{curl} \mathbf{u} &= \mathbf{f} ,\\
    \operatorname{div} \mathbf{u} &= 0 ,
  \end{align*}
  where the ``pseudo-pressure'' $p$ acts as a Lagrange multiplier
  enforcing the divergence-free constraint.
\end{example}

\subsection{The multisymplectic conservation law}
\label{sec:mscl_differential}

The multisymplectic conservation law is expressed in terms of
variations of solutions to \eqref{eq:canonical}, which we now define.

\begin{definition}
  \label{def:variation}
  Let $ z \in \Lambda (\Omega) $ be a solution to
  \eqref{eq:canonical}. A \emph{(first) variation} of $z$ is a
  solution $ w \in \Lambda (\Omega) $ to the linearized equation
  \begin{equation}
    \label{eq:variational}
    \mathrm{D} w (x) = \frac{ \partial f }{ \partial z } \bigl( x, z (x) \bigr) w (x) , \quad x \in \Omega ,
  \end{equation}
  called the \emph{variational equation} of \eqref{eq:canonical} at
  $z$.
\end{definition}

\begin{definition}
  The canonical system \eqref{eq:canonical} is said to be
  \emph{multisymplectic} if, for all solutions
  $ z \in \Lambda (\Omega) $ of \eqref{eq:canonical} and first
  variations $ w _1 , w _2 \in \Lambda (\Omega) $, we have
  \begin{equation}
    \label{eq:mscl_differential}
    ( \mathrm{D} w _1 , w _2 ) - ( w _1 , \mathrm{D} w _2 ) = 0 ,
  \end{equation}
  holding at all $ x \in \Omega $. We call
  \eqref{eq:mscl_differential} the \emph{multisymplectic conservation
    law}.
\end{definition}

\begin{proposition}
  The system \eqref{eq:canonical} is multisymplectic if
  $ \partial f / \partial z $ is symmetric at all $ x \in \Omega $,
  $ z \in \operatorname{Alt} \mathbb{R}^n $, which holds if and only
  if the system is Hamiltonian.
\end{proposition}

\begin{proof}
  For all first variations $ w _1 $ and $ w _2 $ of a solution $z$, we have
  \begin{equation*}
    ( \mathrm{D} w _1 , w _2 ) - ( w _1 , \mathrm{D} w _2 ) = \biggl( \frac{ \partial f }{ \partial z } w _1 , w _2 \biggr) - \biggl(  w _1 , \frac{ \partial f }{ \partial z } w _2 \biggr) .
  \end{equation*}
  If $ \partial f / \partial z $ is symmetric, then the right-hand
  side vanishes, and thus \eqref{eq:mscl_differential} holds. In
  particular, if $ f = \partial H / \partial z $, then
  $ \partial f / \partial z = \partial ^2 H / \partial z ^2 $ is a
  Hessian and thus symmetric. Conversely, symmetry of
  $ \partial f / \partial z $ implies that
  $ f = \partial H / \partial z $ for some $H$ by the Poincar\'e
  lemma.
\end{proof}

We now explain how \eqref{eq:mscl_differential} is a conservation law,
which may not be immediately obvious. A direct calculation using
\eqref{eq:d_delta_identity} shows that, for all
$ w _1 , w _2 \in \Lambda (\Omega) $, not necessarily first variations, we
have
\begin{align}
  \bigl( ( \mathrm{D} w _1 , w _2 ) - ( w _1 , \mathrm{D} w _2 ) \bigr) \mathrm{vol}
  &= 
    \sum _{ k = 1 } ^n \bigl( ( \mathrm{d} w _1 ^{ k -1 } \wedge \star w _2 ^k - w _1 ^{ k -1 } \wedge \star \delta w _2 ^k ) - ( \mathrm{d} w _2 ^{ k -1 } \wedge \star w _1 ^k - w _2 ^{ k -1 } \wedge \star \delta w _1 ^k ) \bigr) \notag\\
  &= \mathrm{d} \sum _{ k = 1 } ^n ( w _1 ^{ k -1 } \wedge \star w _2 ^k - w _2 ^{ k -1 } \wedge \star w _1 ^k ) \label{eq:dirac_domega},
\end{align}
cf.~\citet[Proposition 2.5]{Bridges2006}.  If we define
$ \omega \colon \operatorname{Alt} \mathbb{R}^n \times
\operatorname{Alt} \mathbb{R}^n \rightarrow \operatorname{Alt} ^{ n -1
} \mathbb{R}^n $ by
\begin{equation}
  \label{eq:multisymplectic_form}
  \omega ( w _1 , w _2 ) \coloneqq \sum _{ k = 1 } ^n ( w _1 ^{ k -1 } \wedge \star w _2 ^k - w _2 ^{ k -1 } \wedge \star w _1 ^k ) ,
\end{equation}
then given $ w _1 , w _2 \in \Lambda (\Omega) $, we have
$ \omega ( w _1 , w _2 ) \in \Lambda ^{ n -1 } (\Omega) $. The
multisymplectic conservation law \eqref{eq:mscl_differential} is the
statement that
$ \mathrm{d} \omega ( w _1 , w _2 ) \in \Lambda ^n (\Omega) $ vanishes
whenever $w _1$ and $w _2$ are first variations. Equivalently, since
$ \operatorname{Alt} ^{ n -1 } \mathbb{R}^n \cong \mathbb{R}^n $, we
can interpret $ \omega ( w _1 , w _2 ) $ as a vector field and
\eqref{eq:mscl_differential} as the statement that
\begin{equation*}
  \operatorname{div} \omega ( w _1 , w _2 ) = 0 ,
\end{equation*}
whenever $w _1$ and $w _2$ are first variations.

This also illuminates the reasoning behind the term
\emph{multisymplectic}. Since \eqref{eq:multisymplectic_form} is
antisymmetric in $w_1$ and $w_2$, we can view $\omega$ as a $2$-form on
$ \operatorname{Alt} \mathbb{R}^n $ taking values in
$ \operatorname{Alt} ^{ n -1 } \mathbb{R}^n \cong \mathbb{R}^n $. That
is, we have
$ \omega \in \operatorname{Alt} ^2 ( \operatorname{Alt} \mathbb{R}^n )
\otimes \operatorname{Alt} ^{ n -1 } \mathbb{R}^n \cong
\operatorname{Alt} ^2 ( \operatorname{Alt} \mathbb{R}^n ) \otimes
\mathbb{R}^n $. This can be seen as a collection of $n$ antisymmetric
$2$-forms (``multiple symplectic forms''), one for each coordinate
direction. We refer to $\omega$ as the \emph{canonical multisymplectic
  $2$-form} on $ \operatorname{Alt} \mathbb{R}^n $.

\begin{remark}
  We always interpret $ \mathrm{d} \omega $ as an element of
  $ \operatorname{Alt} ^2 \Lambda (\Omega) \otimes \Lambda ^n (\Omega)
  $, not
  $ \operatorname{Alt} ^3 \Lambda (\Omega) \otimes \Lambda ^{ n -1 }
  (\Omega) $. In the language of \citet*{BrHyLa2010}, $ \mathrm{d} $
  is the \emph{horizontal} exterior derivative, not the vertical
  exterior derivative.
\end{remark}

\begin{example}
  When $ n = 1 $, as in \cref{ex:1D}, the multisymplectic conservation
  law \eqref{eq:mscl_differential} becomes
  \begin{equation*}
    ( J \dot{ w } _1 , w _2 ) - ( w _1 , J \dot{ w } _2 ) = 0 .
  \end{equation*}
  This is just the usual symplectic conservation law
  $ \frac{\mathrm{d}}{\mathrm{d}t} \omega ( w_1, w_2 ) = 0 $, where
  $ \omega ( w_1, w_2 ) \coloneqq ( J w_1, w_2 ) $ is the canonical symplectic
  $2$-form.
\end{example}

\begin{example}
  \label{ex:laplace_mscl_diff}
  Let $ z = \sigma \oplus u \oplus \rho $ be a solution to the
  semilinear $k$-form Hodge--Laplace problem, as in
  \cref{ex:laplace}. A first variation
  $ w = \tau \oplus v \oplus \eta $ satisfies the linearized problem
  \begin{align*}
    \delta v &= \tau ,\\
    \mathrm{d} \tau + \delta \eta &= \frac{ \partial ^2 F }{ \partial u ^2 } v ,\\
    \mathrm{d} v &= \eta .
  \end{align*}
  Hence, if $ w _1 = \tau _1 \oplus v _1 \oplus \eta _1 $ and
  $ w _2 = \tau _2 \oplus v _2 \oplus \eta _2 $ are a pair of first
  variations, we have
  \begin{align*}
    ( \mathrm{D} w _1 , w _2 )
    &= ( \delta v _1 , \tau _2 ) + ( \mathrm{d} \tau _1 + \delta \eta _1 , v _2 ) + ( \mathrm{d} v _1 , \eta _2 ) \\
    &= ( \tau _1 , \tau _2 ) + \biggl( \frac{ \partial ^2 F }{ \partial u ^2 } v _1, v _2 \biggr) + ( \eta _1, \eta _2 ) \\
    &= ( \tau _1 , \tau _2 ) + \biggl( v _1, \frac{ \partial ^2 F }{ \partial u ^2 } v _2 \biggr) + ( \eta _1, \eta _2 ) \\
    &= ( \tau _1  , \delta v _2 ) + ( v _1, \mathrm{d} \tau _2 + \delta \eta _2 ) + ( \eta _1, \mathrm{d} v _2 ) \\
    &= ( w _1 , \mathrm{D} w _2 ) ,
  \end{align*}
  which confirms the multisymplectic conservation law
  \eqref{eq:mscl_differential}. The multisymplectic $2$-form
  \eqref{eq:multisymplectic_form} is
  \begin{equation*}
    \omega ( w _1, w _2 ) = ( \tau _1 \wedge \star v _2 - \tau _2 \wedge \star v _1 ) + ( v _1 \wedge \star \eta _2 - v _2 \wedge \star \eta _1 ) .
  \end{equation*}
\end{example}

\begin{example}
  \label{ex:mscl_vvp}
  If $ z = \sigma \oplus u \oplus \rho $ solves the $k$-form
  vorticity-velocity-pressure problem of \cref{ex:vvp}, then the
  multisymplectic conservation law has precisely the same form as in
  the previous example. The calculation is nearly identical, except
  that
  $ ( \mathrm{d} v _1 , \eta _2 ) = ( \eta _1 , \mathrm{d} v _2 ) $
  holds since both expressions equal $0$ rather than both equaling
  $ ( \eta _1 , \eta _2 ) $.

  In terms of the vector and scalar proxy fields for the Stokes
  problem with $ n = 3 $ and $ k = 2 $, a first variation is
  identified with a solution
  $ w _i = \boldsymbol{ \tau } _i \oplus \mathbf{v} _i \oplus \eta _i
  $ to
  \begin{align*}
    \operatorname{curl} \mathbf{v} _i &= \boldsymbol{\tau} _i ,\\
    \operatorname{curl} \boldsymbol{\tau} _i - \operatorname{grad} \eta _i &= 0 ,\\
    \operatorname{div} \mathbf{v} _i &= 0 .
  \end{align*}
  (Note that $ p = - \rho $, hence the sign change in the second
  equation.) The multisymplectic conservation law is the statement
  that
  $ \operatorname{div} \boldsymbol{ \omega } ( w _1 , w _2 ) = 0 $,
  where $ \boldsymbol{ \omega } $ is the vector-valued $2$-form
  \begin{equation*}
    \boldsymbol{ \omega } ( w _1 , w _2 ) = ( \boldsymbol{ \tau } _1 \times \mathbf{v} _2 - \boldsymbol{ \tau } _2 \times \mathbf{v} _1 ) + ( \mathbf{v} _1 \eta _2 - \mathbf{v} _2 \eta _1 ) .
  \end{equation*}
  
\end{example}

\subsection{Integral form of the multisymplectic conservation law}
If $ K \Subset \Omega $ is a subdomain with smooth boundary
$ \partial K $, Stokes's theorem implies that, for all
$ w_1 , w_2 \in \Lambda (\Omega) $, we have
\begin{equation*}
  \int _K \mathrm{d} \omega ( w_1, w_2 ) = \int _{ \partial K } \operatorname{tr} \omega ( w_1, w_2 ) ,
\end{equation*}
where $ \operatorname{tr} $ denotes pullback by the inclusion
$ \partial K \hookrightarrow \Omega $. Thus, the multisymplectic
conservation law \eqref{eq:mscl_differential} is equivalent to the
statement that
\begin{equation}
  \label{eq:mscl_integral}
  \int _{ \partial K } \operatorname{tr} \omega ( w_1, w_2 ) = 0 ,
\end{equation}
for all such $K$, whenever $w_1$ and $w_2$ are first variations of a
solution to \eqref{eq:canonical}. We call \eqref{eq:mscl_integral} the
\emph{integral form of the multisymplectic conservation law}.

We next show how \eqref{eq:mscl_integral} may be expressed in terms of
the \emph{tangential} and \emph{normal traces} of $w _1$ and $w _2$ on
$ \partial K $. In addition to illuminating what multisymplecticity
says about boundary conditions, this formulation will later generalize
to the weak form of the problem \eqref{eq:canonical} with merely
Lipschitz boundaries. For the moment, though, we continue to assume
that everything is smooth.

\begin{definition}
  \label{def:smooth_traces}
  Let $ \widehat{ \star } $ be the Hodge star on $ \partial K $, with
  respect to the orientation induced by $K$ and the metric induced by
  $\mathbb{R}^n$. The \emph{tangential} and \emph{normal traces} of
  $ w \in \Lambda (\Omega) $ are
  \begin{equation*}
    w ^{\mathrm{tan}} \coloneqq \operatorname{tr} w \in \Lambda (\partial K) , \qquad w ^{\mathrm{nor}} \coloneqq \widehat{ \star } ^{-1} \operatorname{tr} \star w \in \Lambda ( \partial K ) .
  \end{equation*}
  Note that
  $ (w ^k )^{\mathrm{tan}} \in \Lambda ^k ( \partial K ) $, while
  $ (w ^k )^{\mathrm{nor}} \in \Lambda ^{ k -1 } ( \partial K ) $.
\end{definition}

Now, let $ ( \cdot , \cdot ) _K $ denote the $ L ^2 $ inner product on
$ \Lambda (K) $, i.e.,
$ ( w_1, w_2 ) _K = \sum _{ k = 0 } ^n \int _K w_1 ^k \wedge \star w_2 ^k
$. Similarly, let $ \langle \cdot , \cdot \rangle _{ \partial K } $
denote the $ L ^2 $ inner product on $ \Lambda ( \partial K ) $, using
the boundary Hodge star $ \widehat{ \star } $. From Stokes's theorem
and the identity \eqref{eq:d_delta_identity}, we obtain the
integration-by-parts formula
\begin{equation}
  \label{eq:ibp}
  \langle w_1 ^{\mathrm{tan}} , w_2 ^{\mathrm{nor}} \rangle _{ \partial K } = ( \mathrm{d} w_1 , w_2 ) _K - ( w_1 , \delta w_2 ) _K ,
\end{equation}
which holds for all $ w_1 , w_2 \in \Lambda (\Omega) $; see
\citep[Proposition 2.2]{AwFaGuSt2023}.

\begin{definition}
  Define the antisymmetric bilinear form
  \begin{equation*}
    [ w_1, w_2 ] _{ \partial K } \coloneqq \langle w_1 ^{\mathrm{tan}} , w_2 ^{\mathrm{nor}} \rangle _{ \partial K } - \langle w_2 ^{\mathrm{tan}} , w_1 ^{\mathrm{nor}} \rangle _{ \partial K } .
  \end{equation*}
\end{definition}

\begin{proposition}
  \label{prop:mscl_bracket}
  For all $ w_1, w_2 \in \Lambda (\Omega) $, we have
  \begin{align}
    [ w_1 , w_2 ] _{ \partial K }
    &= ( \mathrm{D} w_1 , w_2 ) _K - ( w_1, \mathrm{D} w_2 ) _K \label{eq:ibp_bracket}\\
    &= \int _{ \partial K } \operatorname{tr} \omega ( w_1, w_2 ) \notag.
  \end{align}
  Consequently, the multisymplectic conservation law is equivalent to
  the condition
  \begin{equation}
    \label{eq:mscl_bracket}
    [ w_1 , w_2 ] _{ \partial K } = 0 ,
  \end{equation}
  for all $K \Subset \Omega $ with smooth boundary, when
  $w_1, w_2 \in \Lambda (\Omega) $ are first variations of a solution
  to \eqref{eq:canonical}.
  \label{prop:integral_form}
\end{proposition}

\begin{proof}
  The first two equalities follow immediately from \eqref{eq:ibp} and
  \eqref{eq:dirac_domega}. Therefore, \eqref{eq:mscl_bracket} is just
  a rewriting of the integral form of the multisymplectic conservation
  law \eqref{eq:mscl_integral}.
\end{proof}

\begin{example}
  When $ n = 1 $, the integral form of the multisymplectic
  conservation law over a time interval $ K = [t,t + \Delta t] $ says
  that
  $ \omega \bigl( w _1 (t + \Delta t) , w _2 (t + \Delta t ) \bigr) -
  \omega \bigl( w _1 (t), w _2 (t) \bigr) = 0 $. That is, the
  time-$ \Delta t $ flow map preserves the symplectic form.
\end{example}

\begin{example}
  As in \cref{ex:laplace_mscl_diff}, let
  $ z = \sigma \oplus u \oplus \rho $ be a solution to the semilinear
  $k$-form Hodge--Laplace problem, and let
  $ w _i = \tau _i \oplus v _i \oplus \eta _i $ be first variations
  for $ i = 1 , 2 $. Then
  \begin{equation*}
    [ w _1 , w _2 ] _{ \partial K } = \langle \tau _1 ^{\mathrm{tan}} , v _2 ^{\mathrm{nor}} \rangle _{ \partial K } + \langle v _1 ^{\mathrm{tan}} , \eta _2 ^{\mathrm{nor}} \rangle _{ \partial K } - \langle \tau _2 ^{\mathrm{tan}} , v _1 ^{\mathrm{nor}} \rangle _{ \partial K } - \langle v _2 ^{\mathrm{tan}} , \eta _1 ^{\mathrm{nor}} \rangle _{ \partial K } .
  \end{equation*}
  Hence, the multisymplectic conservation law in the integral form
  \eqref{eq:mscl_bracket} states that
  \begin{equation}
    \label{eq:mscl_HL}
    \langle \tau _1 ^{\mathrm{tan}} , v _2 ^{\mathrm{nor}} \rangle _{ \partial K } + \langle v _1 ^{\mathrm{tan}} , \eta _2 ^{\mathrm{nor}} \rangle _{ \partial K } = \langle \tau _2 ^{\mathrm{tan}} , v _1 ^{\mathrm{nor}} \rangle _{ \partial K } + \langle v _2 ^{\mathrm{tan}} , \eta _1 ^{\mathrm{nor}} \rangle _{ \partial K }.
  \end{equation}
  This expresses a symmetry between tangential and normal boundary
  conditions for the variational problem. In the special case of the
  linear Hodge--Laplace problem, where
  $ \partial ^2 F / \partial u ^2 = 0 $, we recover the symmetry of
  the Dirichlet-to-Neumann operator for harmonic $k$-forms due to
  \citet[Equation 3.6]{BeSh2008}.
\end{example}

\begin{example}
  For the vorticity-velocity-pressure formulation of the Stokes
  problem discussed in \cref{ex:vvp,ex:mscl_vvp}, the integral form of
  the multisymplectic conservation law is
  \begin{equation*}
    \int _{ \partial K } \boldsymbol{ \omega } ( w _1 , w _2 ) \cdot \mathrm{d} \mathbf{S} = 0 .
  \end{equation*}
  In terms of the vector and scalar proxy fields for variations
  $ w _i = \boldsymbol{ \tau } _i \oplus \mathbf{v} _i \oplus \eta _i
  $, this says that
  \begin{equation*}
    \int _{ \partial K } ( \boldsymbol{ \tau } _1 \times \mathbf{v} _2 + \mathbf{v} _1 \eta _2 ) \cdot \mathrm{d} \mathbf{S} = \int _{ \partial K } ( \boldsymbol{ \tau } _2 \times \mathbf{v} _1 + \mathbf{v} _2 \eta _1 ) \cdot \mathrm{d} \mathbf{S} ,
  \end{equation*}
  which is equivalent to \eqref{eq:mscl_HL} using the proxies for
  tangential and normal traces in $ \mathbb{R}^3 $, cf.~\citep[Table
  1]{AwFaGuSt2023}.
\end{example}

\subsection{Remarks on multisymplecticity and reciprocity}

As noted in \citet[Section 2.4]{McSt2020} multisymplecticity of
de~Donder--Weyl-type canonical systems is related to
\emph{reciprocity} phenomena in physics, such as Green's reciprocity
in electrostatics and Betti reciprocity in elasticity;
cf.~\citet[Section 5.3]{AbMa1978}, \citet[Section 5.6]{MaHu1994}, and
\citet{LeMaOrWe2003}. We briefly remark on how the relationship
between multisymplecticity and reciprocity generalizes to the setting
of this paper, with Lorentz reciprocity in electromagnetics arising as
a special case.

\begin{definition}
  If $ z \in \Lambda (\Omega) $ is a solution to \eqref{eq:canonical},
  we say that $ w \in \Lambda (\Omega) $ solves the \emph{linearized
    problem with incremental source}
  $ g \colon \Omega \times \operatorname{Alt} \mathbb{R}^n \rightarrow
  \operatorname{Alt} \mathbb{R}^n $ if
  \begin{equation}
    \mathrm{D} w (x) = \frac{ \partial f }{ \partial z } \bigl( x , z (x) \bigr) w (x) + g \bigl( x, w (x) \bigr) , \quad x \in \Omega .
  \end{equation}
  This generalizes \cref{def:variation}, where a first variation
  corresponds to the case $ g = 0 $.
\end{definition}

Suppose now that $ w _1 , w _2 $ are solutions to the linearized
problems with respective incremental sources $ g _1, g _2 $. If
$ \partial f / \partial z $ is symmetric, then
\begin{equation*}
  ( \mathrm{D} w _1 , w _2 ) - \bigl( g _1 ( \cdot , w _1 ) , w _2 \bigr) = ( w _1 , \mathrm{D} w _2 ) - \bigl( w _1, g _2 ( \cdot , w _2 ) \bigr) .
\end{equation*}
Substituting \eqref{eq:dirac_domega} and using the definition of the
Hodge star, we get the reciprocity equation
\begin{equation}
  \label{eq:reciprocity_differential}
  \mathrm{d} \sum _{ k = 1 } ^n ( w _1 ^{ k -1 } \wedge \star w _2 ^k ) - \sum _{ k = 0 } ^n g _1 ^k ( \cdot , w _1 ) \wedge \star  w _2 ^k = \mathrm{d} \sum _{ k = 1 } ^n ( w _2 ^{ k -1 } \wedge \star w _1 ^k ) - \sum _{ k = 0 } ^n g _2 ^k ( \cdot , w _2 ) \wedge \star  w _1 ^k .
\end{equation}
To get an integral form of reciprocity, we integrate over
$ K \Subset \Omega $ and apply Stokes's theorem to obtain
\begin{equation*}
  \sum _{ k = 1 } ^n \int _{ \partial K } w _1 ^{ k -1 } \wedge \star w _2 ^k - \sum _{ k = 0 } ^n \int _K g _1 ^k ( \cdot , w _1 ) \wedge \star w _2 ^k =
    \sum _{ k = 1 } ^n \int _{ \partial K } w _2 ^{ k -1 } \wedge \star w _1 ^k - \sum _{ k = 0 } ^n \int _K g _2 ^k ( \cdot , w _2 ) \wedge \star w _1 ^k ,
\end{equation*}
or equivalently, in the more convenient notation of the previous section,
\begin{equation*}
  \langle w _1 ^{\mathrm{tan}} , w _2 ^{\mathrm{nor}} \rangle _{ \partial K } - \bigl( g _1 ( \cdot , w _1 ) , w _2 \bigr) _K = \langle w _2 ^{\mathrm{tan}} , w _1 ^{\mathrm{nor}} \rangle _{ \partial K } - \bigl( g _2 ( \cdot , w _2 ) , w _1 \bigr) _K .
\end{equation*}
These reciprocity principles describe a symmetry between perturbations
of the system \eqref{eq:canonical} by incremental sources and the
incremental response of the system to these perturbations. In the
special case where $ g _1 = g _2 = 0 $, we recover the differential
and integral forms of the multisymplectic conservation law.

\begin{example}[Lorentz reciprocity for the time-harmonic Maxwell equations]
  Here, we sketch how an important form of electromagnetic reciprocity
  arises from multisymplectic structure. For consistency with
  conventional notation, this example uses $\omega$ to denote
  frequency rather than the multisymplectic two-form. The
  time-harmonic Maxwell equations with current density $\mathbf{J}$
  are
  \begin{align*}
    \operatorname{curl} \mathbf{E} &= - i \omega \mathbf{B} ,\\
    \operatorname{curl} \mathbf{H} &= i \omega \mathbf{D} + \mathbf{J} .
  \end{align*}
  In linear media, the electric field $\mathbf{E}$ is related to the
  electric flux density $\mathbf{D}$ by the constitutive relation
  $ \mathbf{D} = \epsilon \mathbf{E} $, where $\epsilon$ is the
  electric permittivity tensor. Similarly, the magnetic field
  $\mathbf{H}$ is related to the magnetic flux density $\mathbf{B}$ by
  $ \mathbf{B} = \mu \mathbf{H} $, where $\mu$ is the magnetic
  permeability tensor. Using the constitutive relations, we can
  rewrite the equations as
  \begin{equation*}
    \begin{bmatrix}
      & \epsilon ^{-1} \operatorname{curl} \mu ^{-1} \\
      \operatorname{curl} & 
    \end{bmatrix}
    \begin{bmatrix}
      \mathbf{E} \\
      \mathbf{B} 
    \end{bmatrix} =
    \begin{bmatrix}
      i \omega \mathbf{E} + \epsilon ^{-1} \mathbf{J} \\
      - i \omega \mathbf{B} 
    \end{bmatrix}.
  \end{equation*}
  Now, $\mathbf{E}$ and $\mathbf{B}$ may be viewed as vector proxy
  fields for an electric $1$-form and a magnetic $2$-form, with
  $\epsilon$ and $\mu$ corresponding to Hodge star operators (see
  \citet{Hiptmair2002} and references therein). With this
  identification, the time-harmonic Maxwell equations become a
  canonical Hamiltonian system.

  Furthermore, since this system is linear, we can view
  $ ( \mathbf{E} , \mathbf{B} ) $ as a solution to the linearized
  (homogeneous) problem with incremental source
  $ \mathbf{g} = ( \epsilon ^{-1} \mathbf{J} , \mathbf{0} ) $. Hence,
  given two solutions $ ( \mathbf{E} _1 , \mathbf{B} _1 ) $ and
  $ ( \mathbf{E} _2, \mathbf{B} _2 ) $ corresponding to currents
  $ \mathbf{J} _1 $ and $ \mathbf{J} _2 $, the reciprocity principle
  \eqref{eq:reciprocity_differential} becomes
  \begin{equation*}
    \operatorname{div} ( \mathbf{E} _1 \times \mathbf{H} _2 ) - \mathbf{J} _1 \cdot \mathbf{E} _2 = \operatorname{div} ( \mathbf{E} _2 \times \mathbf{H} _1 ) - \mathbf{J} _2 \cdot \mathbf{E} _1 ,
  \end{equation*}
  known as \emph{Lorentz reciprocity} \citep{Lorentz1895}. The case
  $ \mathbf{J} _1 = \mathbf{J} _2 = 0 $ is the multisymplectic
  conservation law.
\end{example}

\section{Multisymplectic hybrid methods for canonical PDEs}
\label{sec:multisymplectic_hybrid}

\subsection{Hilbert spaces of differential forms and weak traces on
  Lipschitz domains}

In \cref{sec:multisymplectic}, we focused on smooth differential forms
on domains with smooth boundary. To provide the functional analytical
setting for FEEC-type methods, we briefly recall several Hilbert
spaces of differential forms and their associated operators, along
with weak notions of tangential and normal traces on Lipschitz
domains. We refer the reader to
\citet*{ArFaWi2006,ArFaWi2010,Arnold2018} and references therein for a
more comprehensive background on the spaces and operators used in
FEEC, and to \citet{Weck2004,KuAu2012} for essential results on weak
traces. A review of these topics can also be found in
\citet{AwFaGuSt2023}, whose approach and notation we will largely
follow.

Given a bounded Lipschitz domain $ \Omega \subset \mathbb{R}^n $, let
$ L ^2 \Lambda ^k (\Omega) $ be the completion of
$ \Lambda ^k (\Omega) $ with respect to the $ L ^2 $ inner product
$ ( \cdot , \cdot ) _\Omega $. Taking $ \mathrm{d} $ and $\delta$ in
the sense of distributions, we may then define the Sobolev-like spaces
\begin{align*}
  H \Lambda ^k (\Omega) &\coloneqq \bigl\{ v \in L ^2 \Lambda ^k (\Omega) : \mathrm{d} v \in L ^2 \Lambda ^{ k + 1 } (\Omega) \bigr\} ,\\
  H ^\ast \Lambda ^k (\Omega) &\coloneqq \bigl\{ v \in L ^2 \Lambda ^k (\Omega) : \delta v \in L ^2 \Lambda ^{ k - 1 } (\Omega) \bigr\} .
\end{align*}
\citet{Weck2004} showed that it is possible to define the tangential
trace of $ \tau \in H \Lambda ^{k-1} (\Omega) $ and normal trace of
$ v \in H ^\ast \Lambda ^k (\Omega) $ on the Lipschitz boundary
$ \partial \Omega $, such that the integration-by-parts formula
\begin{equation*}
  \langle \tau ^{\mathrm{tan}} , v ^{\mathrm{nor}} \rangle _{ \partial \Omega } = ( \mathrm{d} \tau , v ) _\Omega - ( \tau , \delta v ) _\Omega 
\end{equation*}
still holds. Note that these weak traces generally live in subspaces
of $ H ^{ - 1/2 } \Lambda ^{ k -1 } ( \partial \Omega ) $ but not
necessarily in $ L ^2 \Lambda ^{ k -1 } ( \partial \Omega ) $, so
$ \langle \cdot , \cdot \rangle _{ \partial \Omega } $ should be
interpreted as a duality pairing that extends the $ L ^2 $ inner
product on $ \partial \Omega $ \citep[Theorem 8]{Weck2004}.

It is now straightforward to define the direct sums
\begin{alignat*}{2}
  L ^2 \Lambda (\Omega) &= \bigoplus _{ k = 0 } ^n L ^2 \Lambda ^k (\Omega) ,\\
  H \Lambda (\Omega) &= \bigoplus _{ k = 0 } ^n H \Lambda ^k (\Omega)
  &= \bigl\{ w \in L ^2 \Lambda (\Omega) : \mathrm{d} v \in L ^2
  \Lambda (\Omega) \bigr\} ,\\
  H ^\ast \Lambda (\Omega) &= \bigoplus _{ k = 0 } ^n H ^\ast \Lambda
  ^k (\Omega) &= \bigl\{ w \in L ^2 \Lambda (\Omega) : \delta v \in L
  ^2 \Lambda (\Omega) \bigr\} .
\end{alignat*}
so the Hodge--Dirac operator maps
\begin{equation*}
  \mathrm{D} \colon H \Lambda (\Omega) \cap H ^\ast \Lambda (\Omega) \subset L ^2 \Lambda (\Omega) \rightarrow L ^2 \Lambda (\Omega) .
\end{equation*}
Furthermore, since elements of
$ H \Lambda (\Omega) \cap H ^\ast \Lambda (\Omega) $ have both
tangential and normal traces, we have
\begin{equation*}
  [ w _1 , w _2 ] _{ \partial \Omega } = \langle w _1 ^{\mathrm{tan}} , w _2 ^{\mathrm{nor}} \rangle _{ \partial \Omega } - \langle w _2 ^{\mathrm{tan}} , w _1 ^{\mathrm{nor}} \rangle _{ \partial \Omega } = ( \mathrm{D} w _1, w _2 ) _\Omega - ( w _1 , \mathrm{D} w _2 ) _\Omega ,
\end{equation*}
for all
$ w _1 , w _2 \in H \Lambda (\Omega) \cap H ^\ast \Lambda (\Omega)
$.

\subsection{Hybrid variational form with approximate traces}
\label{sec:hybrid}

In this section, we introduce a weak formulation of the canonical
system of PDEs \eqref{eq:canonical}. This generalizes the \emph{flux
  formulation} of \citet{McSt2020} for canonical PDEs in the
de~Donder--Weyl form \eqref{eq:ddw}, which was based on the unified
hybridization framework of \citet{CoGoLa2009} and related to the
earlier framework for DG methods introduced by \citet{ArBrCoMa2001}
for elliptic problems. For the $k$-form Hodge--Laplace problem, our
weak formulation contains that presented in Section~8 of
\citet{AwFaGuSt2023}, which includes conforming, nonconforming, and
HDG-type methods in FEEC.

\begin{remark}
  \label{rmk:non-uniqueness}
  Our goal is to obtain multisymplecticity criteria for weak and
  numerical solutions to \eqref{eq:canonical}, without assuming
  anything about their existence/uniqueness. Indeed,
  multisymplecticity is a statement about variations within a
  \emph{family} of solutions, so uniqueness would make all variations
  trivial and the multisymplectic conservation law a vacuous
  property. For this reason, our variational formulation deliberately
  omits the additional conditions that one imposes to specify a unique
  solution, such as boundary conditions on $ \partial \Omega $ or
  orthogonality to harmonic forms.

  Some readers may be surprised by this statement, since well-posed
  problems often do satisfy variational principles involving
  (nontrivial) variations about a unique solution. However, those
  variations are \emph{arbitrary} test functions, whereas the
  multisymplectic conservation law involves variations that are
  \emph{restricted} to be tangent to the space of solutions---which is
  what the variational equation \eqref{eq:variational} expresses---and
  if there is only a single unique solution, then the tangent space is
  trivial. See \citet*{MaPaSh1998} for a geometric perspective on how
  multisymplecticity arises from a restricted variational principle,
  as well as \citet[Remark 4]{McSt2020}.
\end{remark}

To motivate the approach, observe that if $ z \in \Lambda (\Omega) $
is a solution to \eqref{eq:canonical} and $ w \in \Lambda (\Omega) $
is an arbitrary test function, then integrating by parts with
\eqref{eq:ibp_bracket} over $ K \Subset \Omega $ gives
\begin{equation*}
  ( z , \mathrm{D} w ) _K + [z , w ] _{ \partial K } = ( \mathrm{D} z , w ) _K = \bigl( f ( \cdot , z ) , w \bigr) _K .
\end{equation*}
If $\Omega$ is partitioned into non-overlapping Lipschitz subdomains
$ K \in \mathcal{T} _h $ (e.g., a simplicial triangulation), let us
now define
\begin{equation*}
  ( \cdot , \cdot ) _{ \mathcal{T} _h } \coloneqq \sum _{ K \in \mathcal{T} _h } ( \cdot , \cdot ) _K , \qquad \langle \cdot , \cdot \rangle _{ \partial \mathcal{T} _h } \coloneqq \sum _{ K \in \mathcal{T} _h } \langle \cdot , \cdot \rangle _{ \partial K } , \qquad [ \cdot , \cdot ] _{ \partial \mathcal{T} _h  } \coloneqq \sum _{ K \in \mathcal{T} _h } [ \cdot , \cdot ] _{ \partial K } .
\end{equation*}
By summing the preceding equation over $K \in \mathcal{T} _h $, we see
that $z$ satisfies
\begin{equation*}
  ( z, \mathrm{D} w ) _{ \mathcal{T} _h } + [ z , w ] _{ \partial \mathcal{T} _h } = \bigl( f ( \cdot , z ) , w \bigr) _{ \mathcal{T} _h } ,
\end{equation*}
for all test functions $w$.

We now present a numerical formulation that weakens the regularity and
continuity assumptions of this problem in several ways. First, our
trial and test functions live in a ``broken'' space
\begin{equation*}
  W _h \coloneqq \prod _{ K \in \mathcal{T} _h } W _h (K) , \qquad W _h (K) \subset H \Lambda (K) \cap H ^\ast \Lambda (K) ,
\end{equation*}
with the additional assumption that
$ w _h ^{\mathrm{nor}}, w _h ^{\mathrm{tan}} \in L ^2 \Lambda (
\partial \mathcal{T}_h ) \coloneqq \prod _{ K \in \mathcal{T} _h } L
^2 \Lambda ( \partial K ) $ for all $ w _h \in W _h $. Next, the
normal and tangential traces of $z$ are replaced by approximate traces
living in
\begin{alignat*}{2}
  \widehat{ W } _h ^{\mathrm{nor}} &\coloneqq \prod _{ K \in
    \mathcal{T} _h } \widehat{ W } _h ^{\mathrm{nor}} ( \partial K ) ,
  \qquad & \widehat{ W } _h ^{\mathrm{nor}} ( \partial K ) \subset L
  ^2 \Lambda ( \partial K ) ,\\
  \widehat{ W } _h ^{\mathrm{tan}} &\coloneqq \prod _{ K \in
    \mathcal{T} _h } \widehat{ W } _h ^{\mathrm{tan}} ( \partial K ) ,
  \qquad & \widehat{ W } _h ^{\mathrm{tan}} ( \partial K ) \subset L
  ^2 \Lambda ( \partial K ) .
\end{alignat*}
To impose a relation between the normal and tangential traces, we
choose a \emph{local flux function} (in the terminology of
\citet{McSt2020}), which is a bounded linear map
\begin{equation*}
  \Phi \coloneqq \prod _{ K \in \mathcal{T} _h } \Phi _K , \qquad \Phi _K \colon W _h (K) \times \widehat{ W } _h ^{\mathrm{nor}} ( \partial K ) \times \widehat{ W } _h ^{\mathrm{tan}} ( \partial K ) \rightarrow L ^2 \Lambda ( \partial K ) .
\end{equation*}
We also replace the pointwise definition of $f (x, z) $ by a weaker
local source term,
\begin{equation*}
  f \coloneqq \prod _{ K \in \mathcal{T} _h } f _K , \qquad f _K \colon W _h (K) \rightarrow L ^2 \Lambda (K) .
\end{equation*}
Assume $f$ is at least $ C ^1 $, so that the partial derivative
$ \partial f / \partial z $ may be replaced by the variational
derivative $ f ^\prime $ in discussing variations of weak solutions.

At an interface $ e = \partial K ^+ \cap \partial K ^- $ (e.g., a
facet of the triangulation), traces in
$ L ^2 ( \partial \mathcal{T} _h ) $ are generally double-valued,
since there is no continuity imposed between the $ e ^\pm $
values---so we also need a space of single-valued traces to ``glue
together'' the broken spaces defined above. However, the notion of
what it means to be single-valued is different for normal and
tangential traces. For example, if $ w \in \Lambda (\Omega) $ is
smooth, then
$ w ^{\mathrm{nor}} \rvert _{ e ^+ } = - w ^{\mathrm{nor}} \rvert _{ e
  ^- } $ but
$ w ^{\mathrm{tan}} \rvert _{ e ^+ } = w ^{\mathrm{tan}} \rvert _{ e
  ^- } $, where the sign flip in the normal trace is due to the
orientation-dependence of $ \widehat{ \star } $ in
\cref{def:smooth_traces}. Based on this observation, we now define
jump and average operators for normal and tangential traces.

\begin{definition}
  \label{def:jump_avg}
  Given
  $ \widehat{ w } ^{\mathrm{nor}} \in L ^2 \Lambda ( \partial
  \mathcal{T} _h ) $, define the \emph{normal jump}
  $ \jump{ \widehat{ w } ^{\mathrm{nor}} } \in L ^2 \Lambda ( \partial
  \mathcal{T} _h ) $ by
  \begin{alignat*}{2}
    \jump{ \widehat{ w } ^{\mathrm{nor}} } _{ e ^\pm } &= \frac{1}{2} \bigl( \widehat{ w } ^{\mathrm{nor}} \rvert _{ e ^+ } + \widehat{ w } ^{\mathrm{nor}} \rvert _{ e ^- } \bigr) , \quad & e & \not\subset \partial \Omega  ,\\
    \jump{ \widehat{ w } ^{\mathrm{nor}} } _{ e \phantom{^\pm} } &= 0 , \quad & e & \subset \partial \Omega  ,
  \end{alignat*}
  and the \emph{normal average}
  $ \av{ \widehat{ w } ^{\mathrm{nor}} } \in L ^2 \Lambda ( \partial
  \mathcal{T} _h ) $ by
  \begin{alignat*}{2}
    \av{ \widehat{ w } ^{\mathrm{nor}} } _{ e ^\pm } &= \frac{1}{2} \bigl( \widehat{ w } ^{\mathrm{nor}} \rvert _{ e ^\pm } - \widehat{ w } ^{\mathrm{nor}} \rvert _{ e ^\mp } \bigr) , \quad & e & \not\subset \partial \Omega  ,\\
    \av{ \widehat{ w } ^{\mathrm{nor}} } _{ e \phantom{^\pm} } &= \widehat{ w } ^{\mathrm{nor}} \rvert _e  , \quad & e & \subset \partial \Omega  .
  \end{alignat*}
  On the other hand, given
  $ \widehat{ w } ^{\mathrm{tan}} \in L ^2 \Lambda ( \partial
  \mathcal{T} _h ) $, define the \emph{tangential jump}
  $ \jump{ \widehat{ w } ^{\mathrm{tan}} } \in L ^2 \Lambda ( \partial
  \mathcal{T} _h ) $ by
  \begin{alignat*}{2}
    \llbracket \widehat{ w } ^{\mathrm{tan}} \rrbracket _{ e ^\pm } &= \frac{1}{2} \bigl( \widehat{ w } ^{\mathrm{tan}} \rvert _{ e ^\pm } - \widehat{ w } ^{\mathrm{tan}} \rvert _{ e ^\mp } \bigr) , \quad & e & \not\subset \partial \Omega  ,\\
    \llbracket \widehat{ w } ^{\mathrm{tan}} \rrbracket _{ e \phantom{^\pm} } &= \widehat{ w } ^{\mathrm{tan}} \rvert _e , \quad & e & \subset \partial \Omega  ,
  \end{alignat*}
  and the \emph{tangential average}
  $ \av{ \widehat{ w } ^{\mathrm{tan}} } \in L ^2 \Lambda ( \partial
  \mathcal{T} _h ) $ by
  \begin{alignat*}{2}
    \av{ \widehat{ w } ^{\mathrm{tan}} } _{ e ^\pm } &= \frac{1}{2} \bigl( \widehat{ w } ^{\mathrm{tan}} \rvert _{ e ^+ } + \widehat{ w } ^{\mathrm{tan}} \rvert _{ e ^- } \bigr) , \quad & e & \not\subset \partial \Omega  ,\\
    \av{ \widehat{ w } ^{\mathrm{tan}} } _{ e \phantom{^\pm}} &= 0  , \quad & e & \subset \partial \Omega .
  \end{alignat*}
  A normal or tangential trace is \emph{single-valued} if its jump is
  zero on $ \partial \mathcal{T} _h \setminus \partial \Omega $.
\end{definition}

\begin{remark}
  Normal jump corresponds precisely to tangential average, and
  tangential jump to normal average.  The reader may be surprised that
  the definition of jump contains a factor of $ \frac{1}{2} $. This is
  because it is still an element of
  $ L ^2 \Lambda ( \partial \mathcal{T} _h ) $, so the jump across $e$
  is divided equally between $ e ^\pm $. Standard definitions choose
  only one of $ e ^\pm $ to represent $e$, which makes it harder to
  distinguish between orientation-independent and
  orientation-dependent single-valued traces. The definitions above
  are unchanged if we flip $ e ^\pm $ and thus do not require a global
  orientation of the facets.
\end{remark}

The identity
$ \widehat{ w } = \av { \widehat{ w } } + \jump { \widehat{ w } } $
holds for both normal and tangential traces. Consequently, we obtain
the following version of the familiar average-jump formulas for the
inner product on $ L ^2 \Lambda ( \partial \mathcal{T} _h ) $.

\begin{proposition}
  \label{prop:average-jump}
  For all
  $ \widehat{ w } _1^{\mathrm{nor}} , \widehat{ w } _1
  ^{\mathrm{tan}}, \widehat{ w } _2 ^{\mathrm{nor}} , \widehat{ w }
  _2^{\mathrm{tan}} \in L ^2 \Lambda ( \partial \mathcal{T} _h ) $,
  \begin{align*}
    \langle \widehat{ w } _1^{\mathrm{nor}} , \widehat{ w } _2^{\mathrm{tan}} \rangle _{ \partial \mathcal{T} _h } &= \bigl\langle \av{ \widehat{ w } _1^{\mathrm{nor}} } , \jump{ \widehat{ w } _2^{\mathrm{tan}} } \bigr\rangle _{ \partial \mathcal{T} _h } + \bigl\langle \jump{ \widehat{ w } _1^{\mathrm{nor}} } , \av{ \widehat{ w } _2^{\mathrm{tan}} } \bigr\rangle _{ \partial \mathcal{T} _h } ,\\
    \langle \widehat{ w } _1 ^{\mathrm{nor}} , \widehat{ w } _2 ^{\mathrm{nor}} \rangle _{ \partial \mathcal{T} _h } &= \bigl\langle \av { \widehat{ w } _1 ^{\mathrm{nor}} } , \av { \widehat{ w } _2 ^{\mathrm{nor}} } \bigr\rangle _{ \partial \mathcal{T} _h } + \bigl\langle \jump { \widehat{ w } _1 ^{\mathrm{nor}} } , \jump { \widehat{ w } _2 ^{\mathrm{nor}} } \bigr\rangle _{ \partial \mathcal{T} _h }, \\
    \langle \widehat{ w } _1 ^{\mathrm{tan}} , \widehat{ w } _2 ^{\mathrm{tan}} \rangle _{ \partial \mathcal{T} _h } &= \bigl\langle \av { \widehat{ w } _1 ^{\mathrm{tan}} } , \av { \widehat{ w } _2 ^{\mathrm{tan}} } \bigr\rangle _{ \partial \mathcal{T} _h } + \bigl\langle \jump { \widehat{ w } _1 ^{\mathrm{tan}} } , \jump { \widehat{ w } _2 ^{\mathrm{tan}} } \bigr\rangle _{ \partial \mathcal{T} _h } .
  \end{align*}
\end{proposition}

\begin{proof}
  Writing
  $ \widehat{ w } _1^{\mathrm{nor}} = \av{ \widehat{ w } _1^{\mathrm{nor}}
  } + \jump{ \widehat{ w } _1^{\mathrm{nor}} } $ and
  $ \widehat{ w } _2^{\mathrm{tan}} = \av{ \widehat{ w } _2^{\mathrm{tan}}
  } + \jump{ \widehat{ w } _2^{\mathrm{tan}} } $, we expand
  \begin{align*}
    \langle \widehat{ w } _1^{\mathrm{nor}} , \widehat{ w } _2^{\mathrm{tan}} \rangle _{ \partial \mathcal{T} _h } =
    \bigl\langle \av{ \widehat{ w } _1^{\mathrm{nor}} } , \av{ \widehat{
        w } _2^{\mathrm{tan}} } \bigr\rangle _{ \partial \mathcal{T} _h
    } &+ \bigl\langle \av{ \widehat{ w } _1^{\mathrm{nor}} } , \jump{
      \widehat{ w } _2^{\mathrm{tan}} } \bigr\rangle _{ \partial
        \mathcal{T} _h }\\
      &+ \bigl\langle \jump{ \widehat{ w }
      _1^{\mathrm{nor}} } , \av{ \widehat{ w } _2^{\mathrm{tan}} }
    \bigr\rangle _{ \partial \mathcal{T} _h } + \bigl\langle \jump{
      \widehat{ w } _1^{\mathrm{nor}} } , \jump{ \widehat{ w }
      _2^{\mathrm{tan}} } \bigr\rangle _{ \partial \mathcal{T} _h }.
  \end{align*}
  If $e$ is an interior facet,
  $ \av{ \widehat{ w } _1^{\mathrm{nor}} } $ flips sign across
  $ e ^\pm $ but $ \av{ \widehat{ w } _2^{\mathrm{tan}}} $ does
  not. Therefore,
  \begin{equation*}
    \bigl\langle \av{ \widehat{ w } _1^{\mathrm{nor}} } , \av{ \widehat{
        w } _2^{\mathrm{tan}} } \bigr\rangle _{ \partial \mathcal{T} _h } = 0 ,
  \end{equation*}
  since the $ e ^\pm $ terms cancel. Likewise,
  $ \jump{ \widehat{ w } _2^{\mathrm{tan}}} $ flips sign across
  $ e ^\pm $ but $ \jump{ \widehat{ w } _1^{\mathrm{nor}} } $ does
  not, so
  \begin{equation*}
    \bigl\langle \jump{ \widehat{ w } _1^{\mathrm{nor}} } , \jump{
      \widehat{ w } _2^{\mathrm{tan}} } \bigr\rangle _{ \partial
      \mathcal{T} _h } = 0 ,
  \end{equation*}
  which proves the first formula. The other two formulas are proved by
  the same argument.
\end{proof}

Finally, define the subspaces of single-valued approximate tangential
traces
$ \ringhat{V} _h ^{\mathrm{tan}} \subset
  \widehat{ V } _h ^{\mathrm{tan}} $ by
\begin{equation*}
  \ringhat{V} _h ^{\mathrm{tan}} \coloneqq \bigl\{ \widehat{ w } _h ^{\mathrm{tan}} \in \widehat{ W } _h ^{\mathrm{tan}} : \jump{\widehat{ w } _h ^{\mathrm{tan}}} = 0 \bigr\} , \qquad \widehat{ V } _h ^{\mathrm{tan}} \coloneqq \bigl\{ \widehat{ w } _h ^{\mathrm{tan}} \in \widehat{ W } _h ^{\mathrm{tan}} : \jump{ \widehat{ w } _h ^{\mathrm{tan}}} = 0 \text{ on } \partial \mathcal{T} _h \setminus \partial \Omega \bigr\} .
\end{equation*}
We now seek solutions
$ ( z _h , \widehat{ z } _h ^{\mathrm{nor}} , \widehat{ z } _h
^{\mathrm{tan}} ) \in W _h \times \widehat{ W } _h ^{\mathrm{nor}}
\times \widehat{ V } _h ^{\mathrm{tan}} $ satisfying
\begin{alignat*}{2}
  ( z _h , \mathrm{D} w _h ) _{ \mathcal{T} _h } + \langle \widehat{ z } _h ^{\mathrm{tan}} , w _h ^{\mathrm{nor}} \rangle  _{ \partial \mathcal{T} _h } - \langle \widehat{ z } _h ^{\mathrm{nor}} , w _h ^{\mathrm{tan}} \rangle  _{ \partial \mathcal{T} _h } &= \bigl( f (z _h) , w _h \bigr) _{ \mathcal{T} _h } , \quad & \forall w _h &\in W _h , \\
  \bigl\langle \Phi ( z _h , \widehat{ z } _h ^{\mathrm{nor}}, \widehat{ z } _h ^{\mathrm{tan}} ) , \widehat{ w } _h ^{\mathrm{nor}} \bigr\rangle _{ \partial \mathcal{T} _h } &= 0, \quad & \forall \widehat{ w } _h ^{\mathrm{nor}} &\in \widehat{ W } _h ^{\mathrm{nor}} , \\
  \langle \widehat{ z } _h ^{\mathrm{nor}} , \widehat{ w } _h ^{\mathrm{tan}} \rangle _{ \partial \mathcal{T} _h } &= 0 , \quad & \forall \widehat{ w } _h ^{\mathrm{tan}} &\in \ringhat{V} _h ^{\mathrm{tan}}.
\end{alignat*}
By abbreviating
$ ( z _h , \widehat{ z } _h ) \coloneqq ( z _h , \widehat{ z } _h
^{\mathrm{nor}} , \widehat{ z } _h ^{\mathrm{tan}} ) $, we can also
write this in the equivalent form
\begin{subequations}
  \label{eq:weakform}
  \begin{alignat}{2}
    ( z _h , \mathrm{D} w _h ) _{ \mathcal{T} _h } + [ \widehat{ z } _h , w _h ] _{ \partial \mathcal{T} _h } &= \bigl( f (z _h) , w _h \bigr) _{ \mathcal{T} _h } , \quad & \forall w _h &\in W _h , \label{eq:weakform_w} \\
    \bigl\langle \Phi ( z _h , \widehat{ z } _h ) , \widehat{ w } _h ^{\mathrm{nor}} \bigr\rangle _{ \partial \mathcal{T} _h } &= 0, \quad & \forall \widehat{ w } _h ^{\mathrm{nor}} &\in \widehat{ W } _h ^{\mathrm{nor}} , \label{eq:weakform_wnor} \\
    \langle \widehat{ z } _h ^{\mathrm{nor}} , \widehat{ w } _h ^{\mathrm{tan}} \rangle _{ \partial \mathcal{T} _h } &= 0 , \quad & \forall \widehat{ w } _h ^{\mathrm{tan}} &\in \ringhat{V} _h ^{\mathrm{tan}}  \label{eq:weakform_wtan} .
  \end{alignat}
  In light of \cref{prop:average-jump}, observe that
  \eqref{eq:weakform_wtan} may be rewritten equivalently as
  \begin{equation}
    \label{eq:jump.conservativity}\tag{\ref*{eq:weakform_wtan}$^\prime$}
    \bigl\langle \jump{ \widehat{ z } _h ^{\mathrm{nor}} } , \widehat{ w } _h ^{\mathrm{tan}} \bigr\rangle _{ \partial \mathcal{T} _h } = 0, \quad \forall \widehat{ w } _h ^{\mathrm{tan}} \in \ringhat{V} _h ^{\mathrm{tan}} .
  \end{equation}
\end{subequations}
Using the terminology of \citet{CoGoLa2009},
\eqref{eq:weakform_w}--\eqref{eq:weakform_wnor} are \emph{local
  solvers} for \eqref{eq:canonical} on each $ K \in \mathcal{T} _h $
with tangential boundary conditions
$ \widehat{ z } _h ^{\mathrm{tan}} $, and \eqref{eq:weakform_wtan} is
a \emph{conservativity condition} that globally couples the local
problems by requiring that $ \widehat{ z } _h ^{\mathrm{nor}} $ be
weakly single-valued.

Note that, since all traces are assumed to be in
$ L ^2 \Lambda ( \partial \mathcal{T} _h ) $, we may henceforth
interpret
$ \langle \cdot , \cdot \rangle _{ \partial \mathcal{T} _h } $ as the
$ L ^2 $ inner product rather than as a duality pairing between weaker
trace spaces.

\begin{example}
  For the semilinear $k$-form Hodge--Laplace problem, we have
  \begin{equation*}
    z _h = \sigma _h \oplus u _h \oplus \rho _h , \qquad 
    \widehat{ z } _h ^{\mathrm{nor}} = \widehat{ u } _h
    ^{\mathrm{nor}} \oplus \widehat{ \rho } _h ^{\mathrm{nor}} , \qquad 
    \widehat{ z } _h ^{\mathrm{tan}} = \widehat{ \sigma } _h
    ^{\mathrm{tan}} \oplus \widehat{ u } _h ^{\mathrm{tan}} .
  \end{equation*}
  In terms of the individual form degrees, \eqref{eq:weakform_w} says
  that
  \begin{subequations}
  \label{eq:weakform_HL}
    \begin{alignat}{2}
      ( u _h , \mathrm{d} \tau _h ) _{ \mathcal{T} _h } - \langle \widehat{ u } _h ^{\mathrm{nor}} , \tau _h ^{\mathrm{tan}} \rangle _{ \partial \mathcal{T} _h } &= ( \sigma _h , \tau _h ) _{ \mathcal{T} _h } , \quad & \forall \tau _h &\in W _h ^{ k -1 } , \label{eq:weakform_tau}\\
      ( \sigma _h , \delta v _h ) _{ \mathcal{T} _h } + ( \rho _h ,
      \mathrm{d} v _h ) _{ \mathcal{T} _h } + \langle \widehat{ \sigma }
      _h ^{\mathrm{tan}} , v _h ^{\mathrm{nor}} \rangle _{ \partial
        \mathcal{T} _h } - \langle \widehat{ \rho } _h ^{\mathrm{nor}} ,
      v _h ^{\mathrm{tan}} \rangle _{ \partial \mathcal{T} _h } &=
      \biggl( \frac{ \partial F }{ \partial u _h } , v _h \biggr) _{ \mathcal{T} _h } , \quad &\forall v _h &\in W _h ^k , \label{eq:weakform_v}\\
      ( u _h , \delta \eta _h ) _{ \mathcal{T} _h } + \langle \widehat{ u } _h ^{\mathrm{tan}} , \eta _h ^{\mathrm{nor}} \rangle _{ \partial \mathcal{T} _h } &= ( \rho _h , \eta _h ) _{ \mathcal{T} _h } , \quad &\forall \eta _h &\in W _h ^{ k + 1 } \label{eq:weakform_eta}.
    \end{alignat}
  \end{subequations}
  For the linear $k$-form Hodge--Laplace problem, this agrees with
  Equations (20a)--(20c) in \citet{AwFaGuSt2023}. The conservativity
  condition \eqref{eq:weakform_wtan} becomes the pair of conditions
  \begin{alignat*}{2}
    \langle \widehat{ u } _h ^{\mathrm{nor}} , \widehat{ \tau } _h ^{\mathrm{tan}} \rangle _{ \partial \mathcal{T} _h } &= 0 , \quad &\forall \widehat{ \tau } _h ^{\mathrm{tan}} &\in \ringhat{V} _h ^{k -1 , \mathrm{tan}} ,\\
    \langle \widehat{ \rho } _h ^{\mathrm{nor}}, \widehat{ v } _h ^{\mathrm{tan}} \rangle _{ \partial \mathcal{T} _h } &= 0 , \quad &\forall \widehat{ v } _h ^{\mathrm{tan}} &\in \ringhat{V} _h ^{k, \mathrm{tan}} ,
  \end{alignat*}
  which agrees with Equations (20g)--(20h) in
  \citep{AwFaGuSt2023}. Several choices are possible for the spaces
  and flux function $\Phi$, leading to various families of conforming,
  nonconforming, and HDG-type methods
  \citep[Section~8.2]{AwFaGuSt2023}. Specific choices and the
  corresponding methods will be discussed further in
  \cref{sec:particular_methods}.
\end{example}

\subsection{Local multisymplecticity criteria}

Elements of $ H \Lambda (K) \cap H ^\ast \Lambda (K) $ generally do
not have well-defined pointwise values, so it does not make sense to
talk about pointwise multisymplecticity as in
\cref{sec:mscl_differential}. Our approach is motivated instead by the
integral form \eqref{eq:mscl_bracket} of the multisymplectic
conservation law, applied to the \emph{approximate} tangential and
normal traces of variations. We first define what it means for
$ ( w _1 , \widehat{ w } _1 ) $ and $ ( w _2 , \widehat{ w } _2 ) $ to
be variations of \eqref{eq:weakform}.

\begin{definition}
  We say that
  $ ( w _i , \widehat{ w } _i ) \coloneqq ( w _i , \widehat{ w } _i
  ^{\mathrm{nor}} , \widehat{ w } _i ^{\mathrm{tan}} ) \in W _h \times
  \widehat{ W } _h ^{\mathrm{nor}} \times \widehat{ V } _h
  ^{\mathrm{tan}} $ is a \emph{(first) variation} of a solution
  $ ( z _h , \widehat{ z } _h ) $ to \eqref{eq:weakform} if it
  satisfies the linearized system of equations
  \begin{subequations}
    \label{eq:weakvar}
    \begin{alignat}{2}
      ( w _i , \mathrm{D} w _h ) _{ \mathcal{T} _h } + [ \widehat{ w } _i , w _h ] _{ \partial \mathcal{T} _h } &= \bigl( f ^\prime (z _h) w _i , w _h \bigr) _{ \mathcal{T} _h } , \quad & \forall w _h &\in W _h , \label{eq:weakvar_w} \\
      \bigl\langle \Phi ( w _i , \widehat{ w } _i ) , \widehat{ w } _h ^{\mathrm{nor}} \bigr\rangle _{ \partial \mathcal{T} _h } &= 0, \quad & \forall \widehat{ w } _h ^{\mathrm{nor}} &\in \widehat{ W } _h ^{\mathrm{nor}} , \label{eq:weakvar_wnor} \\
      \langle \widehat{ w } _i ^{\mathrm{nor}} , \widehat{ w } _h ^{\mathrm{tan}} \rangle _{ \partial \mathcal{T} _h } &= 0 , \quad & \forall \widehat{ w } _h ^{\mathrm{tan}} &\in \ringhat{V} _h ^{\mathrm{tan}} \label{eq:weakvar_wtan} .
    \end{alignat}
  \end{subequations}
\end{definition}

\begin{remark}
  Note that \eqref{eq:weakvar} approximates the variational equation
  \eqref{eq:variational} using the same weak formulation by which
  \eqref{eq:weakform} approximates \eqref{eq:canonical}. In other
  words, variations of a numerical solution are numerical solutions of
  the variational equation. The reason is that the left-hand side of
  \eqref{eq:weakform} corresponds to a bounded linear operator, which
  is unchanged when linearizing. This can be seen as a FEEC version of
  a property that, in the context of linear methods for numerical
  ODEs, \citet{BoSc1994} call \emph{closure under differentiation},
  which they use to characterize symplectic integrators. See also
  \citet*[Lemma~4.1]{HaLuWa2006}.
\end{remark}

The following definition of multisymplecticity depends only on the
properties of the local problems
\eqref{eq:weakform_w}--\eqref{eq:weakform_wnor} and the corresponding
variational equations \eqref{eq:weakvar_w}--\eqref{eq:weakvar_wnor} on
each $ K \in \mathcal{T} _h $. Later, in \cref{sec:global_mscl}, we
will explore conditions under which \eqref{eq:weakform_wtan} and
\eqref{eq:weakvar_wtan} allow the local multisymplectic conservation
laws to be ``glued together'' to obtain a stronger notion of
multisymplecticity.

\begin{definition}
  \label{def:local_multisymplecticity}
  The weak problem \eqref{eq:weakform} is \emph{multisymplectic} if,
  whenever $ ( z _h , \widehat{ z } _h ) $ satisfies
  \eqref{eq:weakform_w}--\eqref{eq:weakform_wnor} with
  $ f ^\prime ( z _h ) $ being symmetric, and
  $ ( w _1 , \widehat{ w } _1 ) $ and $ ( w _2 , \widehat{ w } _2 ) $
  satisfy \eqref{eq:weakvar_w}--\eqref{eq:weakvar_wnor}, we have
  \begin{equation}
    \label{eq:mscl_flux}
    [ \widehat{ w } _1 , \widehat{ w } _2 ] _{ \partial K } = 0 ,
  \end{equation}
  for all $ K \in \mathcal{T} _h $.
\end{definition}

The following result generalizes Lemma~2 of \citet{McSt2020},
characterizing multisymplecticity entirely in terms of the jumps
$ (\widehat{ w } _i - w _i ) \bigr\rvert _{ \partial K } $ between
approximate and actual traces.

\begin{lemma}
  \label{lem:jump}
  If $ ( z _h , \widehat{ z } _h ) $ satisfies \eqref{eq:weakform_w}
  with $ f ^\prime ( z _h ) $ being symmetric, and if
  $ ( w _1 , \widehat{ w } _1 ) $ and $ ( w _2 , \widehat{ w } _2 ) $
  satisfy \eqref{eq:weakvar_w}, then
  \begin{equation*}
    [ \widehat{ w } _1 - w _1 , \widehat{ w } _2 - w _2 ] _{ \partial K } = [ \widehat{ w } _1 , \widehat{ w } _2 ] _{ \partial K } ,
  \end{equation*}
  for all $ K \in \mathcal{T} _h $. Consequently, the local
  multisymplecticity condition \eqref{eq:mscl_flux} holds if and
  only if
  \begin{equation}
    \label{eq:mscl_jump}
    [ \widehat{ w } _1 - w _1 , \widehat{ w } _2 - w _2 ] _{ \partial K } = 0 .
  \end{equation}
\end{lemma}

\begin{proof}
  Since $ ( w _i , \widehat{ w } _i ) $ satisfies \eqref{eq:weakvar_w}
  for $ i = 1, 2 $, we have
  \begin{equation*}
    [ \widehat{ w } _i , w _h ] _{ \partial K } = \bigl( f ^\prime ( z _h ) w _i , w _h \bigr) _K - ( w _i , \mathrm{D} w _h ) _K ,
  \end{equation*}
  for all $ w _h \in W _h $ and $ K \in \mathcal{T} _h $. In
  particular, taking $ i = 1 $ with $ w _h = w _2 $ and $ i = 2 $ with
  $ w _h = w _1 $ gives
  \begin{align*}
    [ \widehat{ w } _1 , w _2 ] _{ \partial K } &= \bigl( f ^\prime ( z _h ) w _1 , w _2 \bigr) _K - ( w _1 , \mathrm{D} w _2 ) _K ,\\
    [ \widehat{ w } _2 , w _1 ] _{ \partial K } &= \bigl( f ^\prime ( z _h ) w _2 , w _1 \bigr) _K - ( w _2 , \mathrm{D} w _1 ) _K .
  \end{align*}
  Since $ f ^\prime ( z _h ) $ is assumed to be symmetric, we have
  $ \bigl( f ^\prime ( z _h ) w _1 , w _2 \bigr) _K = \bigl( f ^\prime
  ( z _h ) w _2 , w _1 \bigr) _K $. Therefore, subtracting the two
  equations above and using the antisymmetry of
  $ [ \cdot , \cdot ] _{ \partial K } $, we obtain
  \begin{equation*}
    [ \widehat{ w } _1, w _2 ] _{ \partial K } + [w _1 ,  \widehat{ w } _2 ] _{ \partial K } = ( \mathrm{D} w _1, w _2 ) _K - ( w _1 , \mathrm{D} w _2 ) _K = [ w _1, w _2 ] _{ \partial K } ,
  \end{equation*}
  where the last equality is \eqref{eq:ibp_bracket}. Finally,
  \begin{equation*}
    [ \widehat{ w } _1 - w _1 , \widehat{ w } _2 - w _2 ] _{ \partial K } = [ \widehat{ w } _1 , \widehat{ w } _2 ] _{ \partial K } - [ \widehat{ w } _1 , w _2 ] _{ \partial K } - [ w _1 , \widehat{ w } _2 ] _{ \partial K } + [ w _1 , w _2 ] _{ \partial K } ,
  \end{equation*}
  and since we have just shown that the last three terms cancel, the
  proof is complete.
\end{proof}

\begin{example}
  \label{ex:mscl_jump_HL}
  For the semilinear $k$-form Hodge--Laplace problem, variations have the form
  \begin{equation*}
    w _i = \tau _i \oplus v _i \oplus \eta _i , \qquad \widehat{ w } _i ^{\mathrm{nor}} = \widehat{ v } _i ^{\mathrm{nor}} \oplus \widehat{ \eta } _i ^{\mathrm{nor}} , \qquad \widehat{ w } _i ^{\mathrm{tan}} = \widehat{ \tau } _i ^{\mathrm{tan}} \oplus \widehat{ v } _i ^{\mathrm{tan}} ,
  \end{equation*}
  and \eqref{eq:weakvar_w} says that these satisfy
  \begin{subequations}
    \label{eq:weakvar_HL}
    \begin{alignat}{2}
      ( v _i , \mathrm{d} \tau _h ) _{ \mathcal{T} _h } - \langle \widehat{ v } _i ^{\mathrm{nor}} , \tau _h ^{\mathrm{tan}} \rangle _{ \partial \mathcal{T} _h } &= ( \tau _i , \tau _h ) _{ \mathcal{T} _h } , \quad & \forall \tau _h &\in W _h ^{ k -1 } ,\\
      ( \tau _i , \delta v _h ) _{ \mathcal{T} _h } + ( \eta _i ,
      \mathrm{d} v _h ) _{ \mathcal{T} _h } + \langle \widehat{ \tau }
      _i ^{\mathrm{tan}} , v _h ^{\mathrm{nor}} \rangle _{ \partial
        \mathcal{T} _h } - \langle \widehat{ \eta } _i ^{\mathrm{nor}} ,
      v _h ^{\mathrm{tan}} \rangle _{ \partial \mathcal{T} _h } &=
      \biggl( \frac{ \partial ^2 F }{ \partial u _h ^2 } v _i , v _h \biggr) _{ \mathcal{T} _h } , \quad &\forall v _h &\in W _h ^k ,\\
      ( v _i , \delta \eta _h ) _{ \mathcal{T} _h } + \langle \widehat{ v } _i ^{\mathrm{tan}} , \eta _h ^{\mathrm{nor}} \rangle _{ \partial \mathcal{T} _h } &= ( \eta _i , \eta _h ) _{ \mathcal{T} _h } , \quad &\forall \eta _h &\in W _h ^{ k + 1 } .
    \end{alignat}
  \end{subequations}
  The local multisymplecticity condition \eqref{eq:mscl_flux} becomes
  \begin{equation*}
    \langle \widehat{ \tau } _1 ^{\mathrm{tan}} , \widehat{ v } _2 ^{\mathrm{nor}} \rangle _{ \partial K } + \langle \widehat{ v } _1 ^{\mathrm{tan}} , \widehat{ \eta } _2 ^{\mathrm{nor}} \rangle _{ \partial K } = \langle \widehat{ \tau } _2 ^{\mathrm{tan}} , \widehat{ v } _1 ^{\mathrm{nor}} \rangle _{ \partial K } + \langle \widehat{ v } _2 ^{\mathrm{tan}} , \widehat{ \eta } _1 ^{\mathrm{nor}} \rangle _{ \partial K },
  \end{equation*}
  which is a version of the Dirichlet-to-Neumann symmetry
  \eqref{eq:mscl_HL} for approximate traces. Finally, since
  $ f ^\prime ( z _h ) = \partial ^2 H / \partial z _h ^2 =
  \mathrm{id} \oplus \partial ^2 F / \partial u _h ^2 \oplus
  \mathrm{id} $ is symmetric, \cref{lem:jump} implies that a necessary
  and sufficient condition for multisymplecticity is
  \begin{align*}
    &\hphantom{{}={}} \langle \widehat{ \tau } _1 ^{\mathrm{tan}} - \tau _1 ^{\mathrm{tan}} , \widehat{ v } _2 ^{\mathrm{nor}} - v _2 ^{\mathrm{nor}} \rangle _{ \partial K } + \langle \widehat{ v } _1 ^{\mathrm{tan}} - v _1 ^{\mathrm{tan}} , \widehat{ \eta } _2 ^{\mathrm{nor}} - \eta _2 ^{\mathrm{nor}} \rangle _{ \partial K } \\
    &= \langle \widehat{ \tau } _2 ^{\mathrm{tan}} - \tau _2 ^{\mathrm{tan}} , \widehat{ v } _1 ^{\mathrm{nor}} - v _1 ^{\mathrm{nor}} \rangle _{ \partial K } + \langle \widehat{ v } _2 ^{\mathrm{tan}} - v _2 ^{\mathrm{tan}} , \widehat{ \eta } _1 ^{\mathrm{nor}} - \eta _1 ^{\mathrm{nor}} \rangle _{ \partial K },
  \end{align*}
  which corresponds to \eqref{eq:mscl_jump}.
\end{example}

\Cref{lem:jump} shows that multisymplecticity of \eqref{eq:weakform}
is determined by the local flux function $\Phi$, which imposes the
relationship between $ w _i $ and $ \widehat{ w } _i $ through
\eqref{eq:weakvar_wnor}.

\begin{definition}
  The local flux function $\Phi$ is \emph{multisymplectic} if
  \eqref{eq:mscl_jump} holds for all $ ( w _1 , \widehat{ w } _1 ) $
  and $ ( w _2 , \widehat{ w } _2 ) $ satisfying
  \eqref{eq:weakvar_wnor}.
\end{definition}

We now prove multisymplecticity for two particular choices of
$\Phi$. The first is used for the hybridization of conforming FEEC
methods, as in \citet[Section 8.2.1]{AwFaGuSt2023}, while the second
is used for the mixed/nonconforming and HDG-type methods in
\citep[Sections 8.2.2--8.2.3]{AwFaGuSt2023}. To state the hypotheses
that the spaces must satisfy, it will be useful to denote
\begin{equation*}
  W _h ^{\mathrm{nor}} \coloneqq \{ w _h ^{\mathrm{nor}} : w _h \in W _h \}, \qquad   W _h ^{\mathrm{tan}} \coloneqq \{ w _h ^{\mathrm{tan}} : w _h \in W _h \} .
\end{equation*}
The following results generalize Theorems~1--2 in \citet{McSt2020}.

\begin{lemma}
  \label{lem:afw-h_flux}
  If
  $ W _h ^{\mathrm{nor}} \subset \widehat{ W } _h ^{\mathrm{nor}} + (
  W _h ^{\mathrm{tan}} + \widehat{ W } _h ^{\mathrm{tan}} ) ^\perp $,
  then
  $ \Phi ( z _h , \widehat{ z } _h ) = \widehat{ z } _h
  ^{\mathrm{tan}} - z _h ^{\mathrm{tan}} $ is multisymplectic.
\end{lemma}

\begin{proof}
  Since $ ( w _i, \widehat{ w } _i ) $
  satisfies \eqref{eq:weakvar_wnor} for $ i = 1 , 2 $, we have
  \begin{equation*}
    \langle \widehat{ w } _i ^{\mathrm{tan}} - w _i ^{\mathrm{tan}} , \widehat{ w } _h ^{\mathrm{nor}} \rangle _{ \partial K } = 0 ,
  \end{equation*}
  for all
  $ \widehat{ w } _h ^{\mathrm{nor}} \in \widehat{ W } _h
  ^{\mathrm{nor}} $ and $ K \in \mathcal{T} _h $. By hypothesis,
  $ \widehat{ w } _2 ^{\mathrm{nor}} - w _2 ^{\mathrm{nor}} $ can be
  written as the sum of
  $ \widehat{ w } _h ^{\mathrm{nor}} \in \widehat{ W } _h
  ^{\mathrm{nor}} $ and an element of
  $ ( W _h ^{\mathrm{tan}} + \widehat{ W } _h ^{\mathrm{tan}} ) ^\perp
  $, both of which are orthogonal to
  $ \widehat{ w } _1 ^{\mathrm{tan}} - w _1 ^{\mathrm{tan}} $ on
  $ \partial K$. Thus,
  \begin{equation*}
    \langle \widehat{ w } _1 ^{\mathrm{tan}} - w _1 ^{\mathrm{tan}} , \widehat{ w } _2 ^{\mathrm{nor}} - w _2 ^{\mathrm{nor}} \rangle _{ \partial K } = 0 ,
  \end{equation*}
  and similarly,
  \begin{equation*}
    \langle \widehat{ w } _2 ^{\mathrm{tan}} - w _2 ^{\mathrm{tan}} , \widehat{ w } _1 ^{\mathrm{nor}} - w _1 ^{\mathrm{nor}} \rangle _{ \partial K } = 0 .
  \end{equation*}
  Subtracting these implies that \eqref{eq:mscl_jump} holds, and we
  conclude that $\Phi$ is multisymplectic.
\end{proof}

\begin{lemma}
  \label{lem:ldg-h_flux}
  If
  $ W _h ^{\mathrm{tan}} + \widehat{ W } _h ^{\mathrm{tan}} \subset
  \widehat{ W } _h ^{\mathrm{nor}} $, then for any bounded symmetric
  operator
  $ \alpha = \prod _{ K \in \mathcal{T} _h } \alpha _{\partial K} $ on
  $ L ^2 \Lambda ( \partial \mathcal{T} _h ) $, the local flux
  function
  $ \Phi ( z _h , \widehat{ z } _h ) = ( \widehat{ z } _h
  ^{\mathrm{nor}} - z _h ^{\mathrm{nor}} ) + \alpha ( \widehat{ z } _h
  ^{\mathrm{tan}} - z _h ^{\mathrm{tan}} ) $ is multisymplectic.
\end{lemma}

\begin{proof}
  Since $ ( w _i, \widehat{ w } _i ) $
  satisfies \eqref{eq:weakvar_wnor} for $ i = 1 , 2 $, we have
  \begin{equation*}
    \bigl\langle ( \widehat{ w } _i ^{\mathrm{nor}} - w _i ^{\mathrm{nor}} ) + \alpha ( \widehat{ w } _i ^{\mathrm{tan}} - w _i ^{\mathrm{tan}} ) , \widehat{ w } _h ^{\mathrm{nor}} \bigr\rangle _{ \partial K } = 0 ,
  \end{equation*}
  for all
  $ \widehat{ w } _h ^{\mathrm{nor}} \in \widehat{ W } _h
  ^{\mathrm{nor}} $ and $ K \in \mathcal{T} _h $. By hypothesis, when $ i = 1 $ we
  may take
  $ \widehat{ w } _h ^{\mathrm{nor}} = \widehat{ w } _2
  ^{\mathrm{tan}} - w _2 ^{\mathrm{tan}} $ to get
  \begin{equation*}
    \bigl\langle ( \widehat{ w } _1 ^{\mathrm{nor}} - w _1 ^{\mathrm{nor}} ) + \alpha ( \widehat{ w } _1 ^{\mathrm{tan}} - w _1 ^{\mathrm{tan}} ) ,  \widehat{ w } _2 ^{\mathrm{tan}} - w _2 ^{\mathrm{tan}} \bigr\rangle _{ \partial K } = 0 ,
  \end{equation*}
  and similarly,
  \begin{equation*}
    \bigl\langle ( \widehat{ w } _2 ^{\mathrm{nor}} - w _2 ^{\mathrm{nor}} ) + \alpha ( \widehat{ w } _2 ^{\mathrm{tan}} - w _2 ^{\mathrm{tan}} ) ,  \widehat{ w } _1 ^{\mathrm{tan}} - w _1 ^{\mathrm{tan}} \bigr\rangle _{ \partial K } = 0 .
  \end{equation*}
  Finally, subtracting and using the symmetry of $\alpha$ implies
  \eqref{eq:mscl_jump}.
\end{proof}

\subsection{Global multisymplecticity criteria}
\label{sec:global_mscl}

The multisymplectic conservation law in the previous section holds on
individual elements $ K \in \mathcal{T} _h $, rather than arbitrary
subsets of $\Omega$ as in \cref{prop:mscl_bracket}. A stronger notion
of multisymplecticity is obtained by considering \emph{collections} of
elements $ \mathcal{K} \subset \mathcal{T} _h $, which cover a region
with boundary $ \partial ( \overline{ \bigcup \mathcal{K} } ) $.

\begin{definition}
  We say that \eqref{eq:weakform} is \emph{strongly multisymplectic}
  if, whenever $ ( z _h , \widehat{ z } _h ) $ is a solution with
  $ f ^\prime ( z _h ) $ being symmetric, and
  $ ( w _1 , \widehat{ w } _1 ) $ and $ ( w _2 , \widehat{ w } _2 ) $
  are first variations satisfying \eqref{eq:weakvar}, we have
  \begin{equation}
    \label{eq:mscl_strong}
    [ \widehat{ w } _1 , \widehat{ w } _2 ] _{ \partial (\overline{ \bigcup \mathcal{K} } ) } = 0 ,
  \end{equation}
  for any collection of elements
  $ \mathcal{K} \subset \mathcal{T} _h $.
\end{definition}

Clearly, strong multisymplecticity implies multisymplecticity in the
sense of \cref{def:local_multisymplecticity}, since we may take
$ \mathcal{K} = \{ K \} $ consisting of a single element. Conversely,
if \eqref{eq:weakform} is multisymplectic in the sense of
\cref{def:local_multisymplecticity}, summing over
$ K \in \mathcal{K} $ gives
\begin{equation}
  \label{eq:mscl_weak}
  [ \widehat{ w } _1 , \widehat{ w } _2 ] _{ \partial \mathcal{K} } \coloneqq \sum _{ K \in \mathcal{K} } [ \widehat{ w } _1 , \widehat{ w } _2 ] _{ \partial K } = 0 .
\end{equation}
However, this is weaker than \eqref{eq:mscl_strong}, since it
generally includes contributions from facets in the interior of the
region covered by $\mathcal{K}$, not just those on the boundary. To
deduce strong multisymplecticity, we will need an additional
assumption under which the interior contributions cancel.

While all solutions to \eqref{eq:weakform} satisfy the weak
conservativity condition \eqref{eq:jump.conservativity}, in certain
cases a stronger form of conservativity holds. The terminology of weak
and strong conservativity is again adapted from \citet{CoGoLa2009}; see
also \citet{ArBrCoMa2001}.

\begin{definition}
  The local flux function $\Phi$ is \emph{strongly conservative} if
  \eqref{eq:weakform_wnor}--\eqref{eq:weakform_wtan} imply
  $ \llbracket \widehat{ z } _h ^{\mathrm{nor}} \rrbracket = 0 $.
\end{definition}

The following simple result gives a sufficient condition for strong
conservativity.

\begin{lemma}
  \label{lemma:strong.conservativity}
  If \eqref{eq:weakform_wnor} implies
  $ \llbracket \widehat{ z } _h ^{\mathrm{nor}} \rrbracket \in
  \ringhat{V} _h ^{\mathrm{tan}} $, then $\Phi$ is
  strongly conservative.
\end{lemma}

\begin{proof}
  If
  $ \llbracket \widehat{ z } _h ^{\mathrm{nor}} \rrbracket \in
  \ringhat{V} _h ^{\mathrm{tan}} $, then taking
  $ \widehat{ w } _h ^{\mathrm{tan}} = \llbracket \widehat{ z } _h
  ^{\mathrm{nor}} \rrbracket $ in \eqref{eq:jump.conservativity}
  immediately implies
  $ \llbracket \widehat{ z } _h ^{\mathrm{nor}} \rrbracket = 0 $.
\end{proof}

The following result generalizes Theorem 3 in \citet{McSt2020}.

\begin{theorem}
  \label{thm:strong.MS}
  If the local flux function $\Phi$ is strongly conservative and
  multisymplectic, then \eqref{eq:weakform} is strongly
  multisymplectic.
\end{theorem}

\begin{proof}
  Suppose $ ( z _h , \widehat{ z } _h ) $ is a solution to
  \eqref{eq:weakform} with $ f ^\prime ( z _h ) $ being symmetric, and
  let $ ( w _1 , \widehat{ w } _1 ) $ and
  $ ( w _2 , \widehat{ w } _2 ) $ be first variations satisfying
  \eqref{eq:weakvar}. \cref{lem:jump} implies that the weak
  multisymplecticity condition \eqref{eq:mscl_weak} holds on
  $ \mathcal{K} \subset \mathcal{T} _h $. Therefore, to establish
  \eqref{eq:mscl_strong}, it suffices to show that the contributions
  to \eqref{eq:mscl_weak} from interior facets
  $ e = \partial K ^+ \cap \partial K ^- $ with
  $ K ^\pm \in \mathcal{K} $ cancel.

  Since $\Phi$ is strongly conservative,
  \eqref{eq:weakvar_wnor}--\eqref{eq:weakvar_wtan} imply
  $ \llbracket \widehat{ w } _i ^{\mathrm{nor}} \rrbracket _{ e ^\pm }
  = 0 $, i.e.,
  $ \widehat{ w } _i ^{\mathrm{nor}} \rvert _{ e ^+ } = - \widehat{ w
  } _i ^{\mathrm{nor}} \rvert _{ e ^- } $. On the other hand, since
  $ \widehat{ w } _i ^{\mathrm{tan}} \in \widehat{ V } _h
  ^{\mathrm{tan}} $, we have
  $ \llbracket \widehat{ w } _i ^{\mathrm{tan}} \rrbracket _{ e^\pm }
  = 0 $, i.e.,
  $ \widehat{ w } _i ^{\mathrm{tan}} \rvert _{ e ^+ } = \widehat{ w }
  _i ^{\mathrm{tan}} \rvert _{ e ^- } $. In other words, the
  single-valued normal trace $ \widehat{ w } _i ^{\mathrm{nor}} $
  flips sign across $ e ^\pm $, while the single-valued tangential
  trace $ \widehat{ w } _i ^{\mathrm{tan}} $ does not. Therefore,
  \begin{align*}
    \langle \widehat{ w } _1 ^{\mathrm{tan}} , \widehat{ w } _2 ^{\mathrm{nor}} \rangle _{ e ^+ } + \langle \widehat{ w } _1 ^{\mathrm{tan}} , \widehat{ w } _2 ^{\mathrm{nor}} \rangle _{ e ^- } &= 0 ,\\
    \langle \widehat{ w } _2 ^{\mathrm{tan}} , \widehat{ w } _1 ^{\mathrm{nor}} \rangle _{ e ^+ } + \langle \widehat{ w } _2 ^{\mathrm{tan}} , \widehat{ w } _1 ^{\mathrm{nor}} \rangle _{ e ^- } &= 0 ,
  \end{align*}
  and subtracting these equations gives
  \begin{equation*}
    [ \widehat{ w } _1 , \widehat{ w } _2 ] _{ e ^+ } +     [ \widehat{ w } _1 , \widehat{ w } _2 ] _{ e ^- } = 0 .
  \end{equation*} 
  Thus, the interior terms of \eqref{eq:mscl_weak} cancel, and we
  conclude that
  $ [ \widehat{ w } _1 , \widehat{ w } _2 ] _{ \partial ( \overline{
      \bigcup \mathcal{K} } ) } = [ \widehat{ w } _1 , \widehat{ w }
  _2 ] _{ \partial \mathcal{K} } = 0 $.
\end{proof}

\section{Multisymplecticity of particular FEEC methods}
\label{sec:particular_methods}
We now use the framework of \cref{sec:multisymplectic_hybrid} to
analyze the weak and strong multisymplecticity of certain methods,
including conforming, nonconforming, and HDG-type methods based on
FEEC. Particular attention will be devoted to the semilinear
Hodge--Laplace problem.

\subsection{The AFW-H method}

We first consider the hybridization of conforming FEEC, where the
solution $ z _h $ lives in a finite element subcomplex of
$ H \Lambda (\Omega) $; we refer to this as the \emph{AFW-H method},
after \citet*{ArFaWi2006,ArFaWi2010}. For the $k$-form Hodge--Laplace
problem, this includes the conforming methods of \citet{AwFaGuSt2023},
which hybridize the methods of \citep{ArFaWi2006,ArFaWi2010}. In
particular, when $ k = 0 $, this includes the hybridized continuous
Galerkin (CG-H) method, and when $ k = n $, it includes the hybridized
Raviart--Thomas (RT-H) and Brezzi--Douglas--Marini (BDM-H) methods, in
the terminology of \citet{CoGoLa2009}. Multisymplecticity of CG-H,
RT-H, and BDM-H was examined in \citet{McSt2020}; here we extend the
treatment to the general $k$-form case.

The AFW-H method is specified as follows.  As in
\citet{AwFaGuSt2023}, let $ W _h (K) $ be a finite element subcomplex
of $ H \Lambda (K) $ for each $ K \in \mathcal{T} _h $. For example,
on a simplicial triangulation, the segment
$ W _h ^{k-1} (K) \xrightarrow{ \mathrm{d} } W _h ^k (K) $ could be any of the
choices
\begin{equation}
  \label{eq:stable_pairs}
  W _h ^{ k -1 } (K) = \begin{Bmatrix}
    \renewcommand\arraystretch{2} \mathcal{P} _{r+1} \Lambda ^{k-1} ( K ) \\[0.5ex]
    \text{or}\\[0.5ex]
    \mathcal{P} _{ r + 1 } ^- \Lambda ^{k-1} ( K )
  \end{Bmatrix},
  \qquad W _h ^k =
  \begin{Bmatrix}
    \renewcommand\arraystretch{2} \mathcal{P} _r \Lambda ^k (
    K ) \text{ (if
      $ r \geq 1 $)} \\[0.5ex]
    \text{or}\\[0.5ex]
    \mathcal{P} _{ r + 1 } ^- \Lambda ^k ( K )
  \end{Bmatrix},
\end{equation}
which \citet{ArFaWi2006,ArFaWi2010} showed to be stable for the
$k$-form Hodge--Laplace problem. Again following \citep{AwFaGuSt2023},
we take the approximate trace spaces to be
\begin{equation*}
  \widehat{ W } _h ^{\mathrm{nor}} ( \partial K ) = \widehat{ W } _h ^{\mathrm{tan}} ( \partial K ) = W _h ^{\mathrm{tan}} ( \partial K ) ,
\end{equation*}
and the local flux function to be
\begin{equation*}
  \Phi ( z _h , \widehat{ z } _h ) = \widehat{ z } _h ^{\mathrm{tan}} - z _h ^{\mathrm{tan}} .
\end{equation*}
It follows that we have conforming subcomplexes $ \mathring{ V } _h \subset V _h \subset H \Lambda (\Omega) $ given by
\begin{equation*}
  \mathring{ V } _h \coloneqq \{ w _h \in W _h : \llbracket w _h ^{\mathrm{tan}} \rrbracket = 0 \} , \qquad V _h \coloneqq \bigl\{ w _h \in W _h : \llbracket w _h ^{\mathrm{tan}} \rrbracket = 0 \text{ on } \partial \mathcal{T} _h \setminus \partial \Omega \bigr\} ,
\end{equation*}
so by the definitions in \cref{sec:hybrid} we get
$ \ringhat{V} _h ^{\mathrm{tan}} = \mathring{ V }
_h ^{\mathrm{tan}} $ and
$ \widehat{ V } _h ^{\mathrm{tan}} = V _h ^{\mathrm{tan}} $.

We now prove that, with this choice of spaces and flux,
\eqref{eq:weakform} is equivalent to a conforming method for
\eqref{eq:canonical}. Compare Theorem~4.2 in \citet{AwFaGuSt2023} for
the Hodge--Laplace problem.

\begin{theorem}
  \label{thm:afw-h_conforming}
  With the AFW-H spaces and flux defined above,
  $ ( z _h , \widehat{ z } _h ) $ satisfies \eqref{eq:weakform} if and
  only if $ z _h \in V _h $ satisfies
  \begin{equation}
    \label{eq:conforming_v}
    ( \mathrm{d} z _h , w _h ) _\Omega + ( z _h , \mathrm{d} w _h ) _\Omega = \bigl( f ( z _h ) , w _h \bigr) _\Omega , \quad \forall w _h \in \mathring{ V } _h ,
  \end{equation}
  where $ \widehat{ z } _h ^{\mathrm{tan}} = z _h ^{\mathrm{tan}} $
  and
  $ \widehat{ z } _h ^{\mathrm{nor}} \in \widehat{ W } _h
  ^{\mathrm{nor}} $ is the unique solution of
  \begin{equation}
    \label{eq:conforming_w}
    \langle \widehat{ z } _h ^{\mathrm{nor}} , w _h ^{\mathrm{tan}} \rangle _{ \partial \mathcal{T} _h } = ( \mathrm{d} z _h , w _h ) _{ \mathcal{T} _h } + ( z _h , \mathrm{d} w _h ) _{ \mathcal{T} _h } - \bigl( f ( z _h ) , w _h \bigr) _{ \mathcal{T} _h } , \quad \forall w _h \in W _h .
  \end{equation}
\end{theorem}

\begin{proof}
  Suppose first that $ ( z _h , \widehat{ z } _h ) $ satisfies
  \eqref{eq:weakform}. Since
  $ \widehat{ W } _h ^{\mathrm{nor}} = \widehat{ W } _h
  ^{\mathrm{tan}} = W _h ^{\mathrm{tan}} $, we may take the test
  function
  $ \widehat{ w } _h ^{\mathrm{nor}} = \widehat{ z } _h
  ^{\mathrm{tan}} - z _h ^{\mathrm{tan}} $ in
  \eqref{eq:weakform_wnor}, which gives
  $ \widehat{ z } _h ^{\mathrm{tan}} = z _h ^{\mathrm{tan}}
  $. Therefore, $ z _h ^{\mathrm{tan}} $ is single-valued, which
  implies $ z _h \in V _h $. Next, recall the integration-by-parts
  identity
  \begin{equation*}
    ( \mathrm{d} z _h , w _h ) _{ \mathcal{T} _h } - ( z _h , \delta w _h ) _{ \mathcal{T} _h } = \langle z _h ^{\mathrm{tan}} , w _h ^{\mathrm{nor}} \rangle _{ \partial \mathcal{T} _h } .
  \end{equation*}
  Together with the fact that
  $ \widehat{ z } _h ^{\mathrm{tan}} = z _h ^{\mathrm{tan}} $, we may
  use this to rewrite \eqref{eq:weakform_w} as
  \begin{equation}
    \label{eq:conforming_w_rearranged}
    ( \mathrm{d} z _h , w _h ) _{ \mathcal{T} _h } +  ( z _h , \mathrm{d} w _h ) _{ \mathcal{T} _h } - \langle \widehat{ z } _h ^{\mathrm{nor}} , w _h ^{\mathrm{tan}} \rangle _{ \partial \mathcal{T} _h } = \bigl( f ( z _h ) , w _h \bigr) _{ \mathcal{T} _h } , \quad \forall w _h \in W _h .
  \end{equation}
  In particular, if $ w _h \in \mathring{ V } _h $, then
  \eqref{eq:weakform_wtan} implies that the trace terms cancel, giving
  \eqref{eq:conforming_v}. Rearranging
  \eqref{eq:conforming_w_rearranged} immediately gives
  \eqref{eq:conforming_w}; uniqueness of
  $ \widehat{ z } _h ^{\mathrm{nor}} $ follows from the
  existence-and-uniqueness argument given below for the converse
  direction.

  Conversely, suppose that $ z _h \in V _h $ satisfies
  \eqref{eq:conforming_v}. Letting
  $ \widehat{ z } _h ^{\mathrm{tan}} = z _h ^{\mathrm{tan}} $, we
  immediately get \eqref{eq:weakform_wnor}. To show existence and
  uniqueness of $ \widehat{ z } _h ^{\mathrm{nor}} $, we define the
  space of \emph{tangential bubbles},
  \begin{equation*}
    \mathring{ W } _h \coloneqq \{ w _h \in W _h : w _h ^{\mathrm{tan}} = 0 \} \subset \mathring{ V } _h .
  \end{equation*}
  If $ w _h \in \mathring{ W } _h \subset \mathring{ V } _h $, then
  the right-hand side of \eqref{eq:conforming_w} vanishes by
  \eqref{eq:conforming_v}. Hence, the right-hand side of
  \eqref{eq:conforming_w} is a well-defined functional in
  $ ( W _h ^{\mathrm{tan}} ) ^\ast $, because any two test functions
  with the same tangential trace differ by an element of
  $ \mathring{ W } _h $. Since
  $ \langle \cdot , \cdot \rangle _{ \partial \mathcal{T} _h } $ is an
  inner product on $ W _h ^{\mathrm{tan}} $, existence and uniqueness
  of $ \widehat{ z } _h ^{\mathrm{nor}} $ follows by the Riesz
  representation theorem. Rearranging \eqref{eq:conforming_w} to
  \eqref{eq:conforming_w_rearranged} and integrating by parts as
  above, but in the reverse order, gives
  \eqref{eq:weakform_w}. Finally, taking
  $ w _h \in \mathring{ V } _h $ and subtracting
  \eqref{eq:conforming_w_rearranged} from \eqref{eq:conforming_v}
  gives \eqref{eq:weakform_wtan}, which completes the proof.
\end{proof}

\begin{remark}
  For the linear Hodge--Dirac problem, \eqref{eq:conforming_v}
  coincides with the conforming method of \citet{LeSt2016} (modulo
  boundary conditions and harmonic forms; see
  \cref{rmk:non-uniqueness}).
\end{remark}

We show now that the AFW-H method is locally (i.e., weakly)
multisymplectic. The following result generalizes Theorem~8 from
\citep{McSt2020} on the multisymplecticity of the CG-H method.

\begin{theorem}
  \label{thm:MS.FEEC}
  The AFW-H method is multisymplectic. It is strongly multisymplectic
  if and only if $ n = 1 $ (in which case multisymplecticity is just
  symplecticity).
\end{theorem}

\begin{proof}
  There are two ways to see that the flux $ \Phi $ is multisymplectic.
  First, if $ ( w _i, \widehat{ w } _i ) $ are variations for
  $ i = 1 , 2 $, then we have seen above that \eqref{eq:weakvar_wnor}
  implies $ \widehat{ w } _i ^{\mathrm{tan}} = w _i ^{\mathrm{tan}}
  $. Thus, \eqref{eq:mscl_jump} holds, and $\Phi$ is multisymplectic
  by \cref{lem:jump}.  Alternatively, since
  $ \widehat{ W } _h ^{\mathrm{nor}} = \widehat{ W } _h
  ^{\mathrm{tan}} = W _h ^{\mathrm{tan}} $, the hypothesis of
  \cref{lem:afw-h_flux} becomes
  $ W _h ^{\mathrm{nor}} \subset W _h ^{\mathrm{tan}} + ( W _h
  ^{\mathrm{tan}} ) ^\perp = L ^2 \Lambda ( \partial \mathcal{T} _h )
  $. This holds, so $\Phi$ is multisymplectic.

  If $ n = 1 $, facets are simply vertices, so weak conservativity is
  the same as strong conservativity. Indeed, every single-valued
  function on vertices is in $ \widehat{ V } _h ^{\mathrm{tan}} $, so
  in particular
  $ \llbracket \widehat{ z } _h ^{\mathrm{nor}} \rrbracket \in
  \ringhat{V} _h ^{\mathrm{tan}} $, and strong conservativity holds by
  \cref{lemma:strong.conservativity}. Strong multisymplecticity then
  follows by \cref{thm:strong.MS}. Conversely, when $ n > 1 $, the
  failure of strong multisymplecticity for arbitrary Hamiltonian
  systems follows from the explicit counterexamples constructed in
  \cref{lem:cg-h_weak} below.
\end{proof}

We next turn our focus to the semilinear $k$-form Hodge--Laplace
problem, which will be a source of examples for which strong
multisymplecticity does or does not hold when $ n > 1 $. Using
$ \widehat{ \sigma } _h ^{\mathrm{tan}} = \sigma _h ^{\mathrm{tan}} $
and $ \widehat{ u } _h ^{\mathrm{tan}} = u _h ^{\mathrm{tan}} $ to
integrate \eqref{eq:weakform_v}--\eqref{eq:weakform_eta} by parts, as
in \eqref{eq:conforming_w_rearranged}, and eliminating
$ \rho _h = \mathrm{d} u _h $, gives the equivalent problem where
\begin{equation*}
  \sigma _h \in W _h ^{ k -1 } , \quad  u _h \in W _h ^k , \quad 
 \widehat{ u } _h ^{\mathrm{nor}} \in \widehat{ W } _h
^{k-1,\mathrm{nor}} , \quad 
\widehat{ \rho } _h ^{\mathrm{nor}} \in \widehat{ W } _h ^{k,
  \mathrm{nor}} , \quad 
\widehat{ \sigma } _h ^{\mathrm{tan}} \in \widehat{ V } _h
^{k-1,\mathrm{tan}} , \quad 
\widehat{ u } _h ^{\mathrm{tan}} \in \widehat{ V } _h ^{k,
  \mathrm{tan}}
\end{equation*}
satisfy
\begin{subequations}
  \label{eq:afw-h_HL}
  \begin{alignat}{2}
    ( \sigma _h , \tau _h ) _{ \mathcal{T} _h } - ( u _h , \mathrm{d} \tau _h ) _{ \mathcal{T} _h } + \langle \widehat{ u } _h ^{\mathrm{nor}} , \tau _h ^{\mathrm{tan}} \rangle _{ \partial \mathcal{T} _h } &= 0, \quad &\forall \tau _h &\in W _h ^{ k -1 } \\
    ( \mathrm{d} \sigma _h , v _h ) _{ \mathcal{T} _h } + ( \mathrm{d} u _h , \mathrm{d} v _h ) _{ \mathcal{T} _h } - \langle \widehat{ \rho } _h ^{\mathrm{nor}} , v _h ^{\mathrm{tan}} \rangle _{ \partial \mathcal{T} _h } &= \biggl( \frac{ \partial F }{ \partial u _h } , v _h \biggr) _{ \mathcal{T} _h } , \quad &\forall v _h &\in W _h ^k ,\\
    \langle \widehat{ \sigma } _h ^{\mathrm{tan}} - \sigma _h ^{\mathrm{tan}} , \widehat{ v } _h ^{\mathrm{nor}} \rangle _{ \partial \mathcal{T} _h } &= 0 , \quad &\forall \widehat{ v } _h ^{\mathrm{nor}} &\in \widehat{ W } _h ^{k-1, \mathrm{nor}}, \\
    \langle \widehat{ u } _h ^{\mathrm{tan}} - u _h ^{\mathrm{tan}} , \widehat{ \eta } _h ^{\mathrm{nor}} \rangle _{ \partial \mathcal{T} _h } &= 0 , \quad &\forall \widehat{ \eta } _h ^{\mathrm{nor}} &\in \widehat{ W } _h ^{k, \mathrm{nor}}, \\
    \langle \widehat{ u } _h ^{\mathrm{nor}} , \widehat{ \tau } _h ^{\mathrm{tan}} \rangle _{ \partial \mathcal{T} _h } &= 0 , \quad &\forall \widehat{ \tau } _h ^{\mathrm{tan}} &\in \ringhat{V} _h ^{k-1,\mathrm{tan}} \\
    \langle \widehat{ \rho } _h ^{\mathrm{nor}} , \widehat{ v } _h ^{\mathrm{tan}} \rangle _{ \partial \mathcal{T} _h } &= 0 , \quad &\forall \widehat{ v } _h ^{\mathrm{tan}} &\in \ringhat{V} _h ^{k,\mathrm{tan}} .
  \end{alignat}
\end{subequations}
By \cref{thm:afw-h_conforming}, this hybridizes the conforming FEEC
problem where $ \sigma _h \in V _h ^{ k -1 } $, $ u _h \in V _h ^k $
satisfy
\begin{subequations}
  \label{eq:afw}
  \begin{alignat}{2}
    ( \sigma _h , \tau _h ) _\Omega - ( u _h , \mathrm{d} \tau _h ) _\Omega &= 0 , \quad &\forall \tau _h &\in \mathring{ V } _h ^{ k -1 } ,\\
    ( \mathrm{d} \sigma _h , v _h ) _\Omega + ( \mathrm{d} u _h , \mathrm{d} v _h ) _\Omega &= \biggl( \frac{ \partial F }{ \partial u _h } , v _h \biggr) _\Omega , \quad &\forall v _h &\in \mathring{ V } _h ^k .
  \end{alignat}
\end{subequations}
For the linear Hodge--Laplace problem, this agrees with the
hybridization of \citet{AwFaGuSt2023} of the FEEC methods of
\citet{ArFaWi2006,ArFaWi2010} (again modulo boundary conditions and
harmonic forms).

The AFW-H method fails to be strongly multisymplectic for this problem
when $ n > 1 $ and $ k = 0 $, i.e., when it coincides with the CG-H
method. Failure of strong multisymplecticity for CG-H was shown in
\citep[Proposition~1]{McSt2020} for the case $ n = 2 $, where the
proof involved a direct computation. The following refinement of that
proof gives counterexamples for arbitrary $ n > 1 $, thereby also
completing the proof of \cref{thm:MS.FEEC}.

\begin{lemma}
  \label{lem:cg-h_weak}
  The AFW-H method is not strongly multisymplectic for the
  Hodge--Laplace problem when $ k = 0 $ and $ n > 1 $ (in which case
  it coincides with the CG-H method).
\end{lemma}

\begin{proof}
  Consider Laplace's equation ($ F = 0 $).  Suppose $ \mathcal{T} _h $
  consists of two equilateral $n$-simplices $ K ^\pm $ sharing a
  common facet $e$, and let
  $ W _h ^0 ( K ^\pm ) = \mathcal{P} _1 \Lambda ^0 ( K ^\pm ) $. The
  degrees of freedom are at vertices, all of which lie on the boundary
  of the domain when $ n > 1 $, so $ \mathring{ V } _h ^0 $ is trivial
  and thus \eqref{eq:afw} is trivially satisfied. Hence, every
  $ v _i \in V _h ^0 $ corresponds to a variation, and
  $ \widehat{ \eta } _i ^{\mathrm{nor}} $ may be recovered by solving
  \begin{equation}
    \label{eq:cg-h_etanor}
    \langle \widehat{ \eta } _i ^{\mathrm{nor}} , v _h ^{\mathrm{tan}} \rangle _{ \partial \mathcal{T} _h } = ( \mathrm{d} v _i , \mathrm{d} v _h ) _{ \mathcal{T} _h } , \quad \forall v _h \in W _h ^0 .
  \end{equation}
  For a counterexample to strong multisymplecticity, it therefore
  suffices to produce $ v _1 , v _2 \in V _h ^0 $ such that
  $ \bigl\langle \widehat{ v } _1 ^{\mathrm{tan}} , \llbracket
  \widehat{ \eta } _2 ^{\mathrm{nor}} \rrbracket \bigr\rangle _e \neq
  \bigl\langle \widehat{ v } _2 ^{\mathrm{tan}} , \llbracket \widehat{
    \eta } _1 ^{\mathrm{nor}} \rrbracket \bigr\rangle _e $.

  Take $ v _1 = 1 $ constant, so that \eqref{eq:cg-h_etanor} implies
  $ \widehat{ \eta } _1 ^{\mathrm{nor}} = 0 $. Next, take
  $ \widehat{ \eta } _2 ^{\mathrm{nor}} \in W _h ^{ 0 , \mathrm{tan} }
  $ equal to $1$ on $e$ and $ -n $ at the two vertices opposite $e$
  (one belonging to each of $ K ^\pm $). It follows that
  \begin{equation*}
    \langle \widehat{ \eta } _2 ^{\mathrm{nor}} , 1 \rangle _{ \partial K ^\pm } = \int _{ \partial K ^\pm } \widehat{ \eta } _2 ^{\mathrm{nor}} = 0 ,
  \end{equation*}
  because each vertex degree of freedom contributes equally (since
  $ K ^\pm $ are equilateral) and they sum to zero. Since
  $ \widehat{ \eta } _2 ^{\mathrm{nor}} $ is
  $ L ^2 \Lambda ( \partial \mathcal{T} _h ) $-orthogonal to piecewise
  constants, i.e., to the kernel of the right-hand side of
  \eqref{eq:cg-h_etanor}, there exists $ v _2 \in W _h ^0 $ satisfying
  \eqref{eq:cg-h_etanor}, and this is unique if we insist that
  $ v _2 $ is also orthogonal to the kernel. Furthermore, since
  $ \widehat{ \eta } _2 ^{\mathrm{nor}} $ has the same degrees of
  freedom on $ K ^\pm $ (simply reflected about $e$), by symmetry the
  same is true of $ v _2 $, and in particular $ v _2 \in V _h ^0 $.

  Finally, since
  $ \widehat{ v } _1 ^{\mathrm{tan}} = v _1 ^{\mathrm{tan}} = 1 $,
  $ \widehat{ v } _2 ^{\mathrm{tan}} = v _2 ^{\mathrm{tan}} $,
  $ \widehat{ \eta } _1 ^{\mathrm{nor}} = 0 $, and
  $ \llbracket \widehat{ \eta } _2 ^{\mathrm{nor}} \rrbracket _e =
  \frac{1}{2} ( \widehat{ \eta } _2 ^{\mathrm{nor}} \rvert _{ e ^+ } +
  \widehat{ \eta } _2 ^{\mathrm{nor}} \rvert _{ e ^- } ) = 1 $,
  \begin{equation*}
    \bigl\langle \widehat{ v } _1 ^{\mathrm{tan}}, \llbracket \widehat{ \eta } _2 ^{\mathrm{nor}} \rrbracket \bigr\rangle _e = \langle 1, 1 \rangle _e > 0 , \qquad \bigl\langle \widehat{ v } _2 ^{\mathrm{tan}} , \llbracket \widehat{ \eta } _1 ^{\mathrm{nor}} \rrbracket \bigr\rangle _e = 0 ,
  \end{equation*}
  and thus strong multisymplecticity is violated, as claimed.
\end{proof}

\begin{remark}
  Note that \emph{weak} multisymplecticity still holds for the
  counterexample constructed above: since
  $ \widehat{ \eta } _2 ^{\mathrm{nor}} $ is orthogonal to the
  constant $ v _1 $ on $ \partial K ^\pm $, we have
  $ \langle \widehat{ v } _1 ^{\mathrm{tan}} , \widehat{ \eta } _2
  ^{\mathrm{nor}} \rangle _{ \partial K^\pm } = 0 = \langle \widehat{
    v } _2 ^{\mathrm{tan}} , \widehat{ \eta } _1 ^{\mathrm{nor}}
  \rangle _{ \partial K ^\pm } $.
\end{remark}

Finally, we discuss a situation in which the AFW-H method \emph{is}
strongly multisymplectic: the case $ k = n $ where AFW-H coincides
with either the BDM-H or RT-H method. Strong multisymplecticity of
these methods was already proved in \citet[Theorems 5 and
6]{McSt2020}, but we now give a self-contained proof in the language
of FEEC.

\begin{proposition}
  \label{prop:afw-h_strong}
  Consider AFW-H for the $k$-form Hodge--Laplace problem, using the
  spaces \eqref{eq:stable_pairs} on a simplicial triangulation. If
  $ k = n $, then the method is strongly multisymplectic.
\end{proposition}

\begin{proof}
  When $ k = n $, AFW-H is equivalent to BDM-H if
  $ W _h ^{ n -1 } ( K ) = \mathcal{P} _{ r + 1 } \Lambda ^{ n -1 }
  (K) $ and to RT-H if
  $ W _h ^{ n -1 } (K) = \mathcal{P} _{ r + 1 } ^- \Lambda ^{ n -1 }
  (K) $. For either choice of $ W _h ^{ n -1 } (K) $, we have
  \begin{equation*}
    \widehat{ W } _h ^{n-1, \mathrm{nor}} ( \partial K ) = \widehat{ W } _h ^{n-1, \mathrm{tan}} ( \partial K ) = W _h ^{n -1 , \mathrm{tan}} ( \partial K ) = \prod _{ e \subset \partial K } \mathcal{P} _r \Lambda ^{ n -1 } (e) .
  \end{equation*}
  If $e$ is an interior facet,
  $ \widehat{ u } _h ^{\mathrm{nor}} \rvert _{ e ^\pm } \in
  \mathcal{P} _r \Lambda ^{ n -1 } ( e ^\pm ) $ implies
  $ \llbracket \widehat{ u } _h ^{\mathrm{nor}} \rrbracket _{ e ^\pm }
  \in \mathcal{P} _r \Lambda ^{ n -1 } ( e ^\pm ) $, so
  $ \llbracket \widehat{ u } _h ^{\mathrm{nor}} \rrbracket \in
  \widehat{ W } _h ^{ n -1 , \mathrm{tan}} $. Furthermore, normal
  jumps are \emph{tangentially} single-valued and vanish on
  $ \partial \Omega $, so
  $ \llbracket \widehat{ u } _h ^{\mathrm{nor}} \rrbracket \in
  \ringhat{V} _h
  ^{n -1 , \mathrm{tan}} $. Finally,
  $ \widehat{ \rho } _h ^{\mathrm{nor}} $ is trivial, since it is an
  $n$-form on the $ (n-1) $-dimensional boundary. Hence, we have
  strong conservativity by \cref{lemma:strong.conservativity} and thus
  strong multisymplecticity by \cref{thm:strong.MS}.
\end{proof}

\begin{remark}
  The key ingredient in this proof is that
  $ W _h ^{ n -1 , \mathrm{tan} } ( \partial K ) = \prod _{ e \subset
    \partial K } W _h ^{ n -1 , \mathrm{tan} } (e) $, i.e., we can
  independently specify tangential traces on each facet, which is not
  true for $k$-forms when $ k < n -1 $. This property of finite
  element $ ( n-1) $-forms remains true for a large variety of FEEC
  elements on non-simplicial meshes: abstractly, this is a property of
  \emph{finite element systems} satisfying the \emph{extension
    property}, as developed by \citeauthor{Christiansen2008} and
  coauthors
  \citep{Christiansen2008,Christiansen2009,ChMuOw2011,ChGi2016,ChRa2016}.
\end{remark}

\subsection{The LDG-H method}
\label{sec:ldg-h}

We next discuss the multisymplecticity of a class of hybridizable
local discontinuous Galerkin (LDG-H) methods for canonical PDEs. For
the $k$-form Hodge--Laplace problem, this includes the LDG-H methods
described in \citet[Section 8.2.3]{AwFaGuSt2023}, as well as certain
cases of the extended Galerkin (XG) method of \citet*{HoLiXu2022}. For
$ k = 0 $ and $ k = n $, this also includes the original LDG-H method
described in \citet{CoGoLa2009} and its alternative implementation in
\citet[Section 5]{Cockburn2016}, respectively. We further show that
this class includes several HDG methods proposed for specific
problems, including Maxwell's equations
\citep{NgPeCo2011,ChQiShSo2017,ChQiSh2018} and the
vorticity-velocity-pressure form of Stokes flow \citep{CoGo2009}.

For LDG-H and other HDG methods, we assume that the trace spaces have
the form
\begin{equation*}
  \widehat{ W } _h ^{\mathrm{tan}} ( \partial K ) = \prod _{ e \subset \partial K } \widehat{ W } _h ^{\mathrm{tan}} (e) , \qquad \widehat{ W } _h ^{\mathrm{nor}} ( \partial K ) = L ^2 \Lambda ( \partial K ) ,
\end{equation*}
so that all three spaces $ W _h $,
$ \widehat{ W } _h ^{\mathrm{nor}} $, and
$ \widehat{ W } _h ^{\mathrm{tan}} $ are broken. We also assume that
$ \widehat{ W } _h ^{\mathrm{tan}} ( e ^+ ) = \widehat{ W } _h
^{\mathrm{tan}} ( e ^- ) $ at interior facets $e$, so that
$ \av{ \widehat{ W } _h ^{\mathrm{tan}} } = \ringhat{V} _h
^{\mathrm{tan}} $. There are many possibilities for $ W _h (K) $ and
$ \widehat{ W } _h ^{\mathrm{tan}} ( e ) $. On a simplicial
triangulation, one of the simplest choices is to take
\begin{equation}
  \label{eq:equal_order}
  W _h (K) = \mathcal{P} _r \Lambda (K) , \qquad \widehat{ W } _h ^{\mathrm{tan}} ( e ) = \mathcal{P} _r \Lambda (e) ,
\end{equation}
corresponding to an ``equal-order'' {HDG} method, i.e., the same $r$
is used for every form degree $k$. Other choices are possible, as we
will explore when we discuss the Hodge--Laplace problem.

The LDG-H flux function is given by
\begin{equation*}
  \Phi ( z _h , \widehat{ z } _h ) = ( \widehat{ z } _h
  ^{\mathrm{nor}} - z _h ^{\mathrm{nor}} ) + \alpha ( \widehat{ z } _h
  ^{\mathrm{tan}} - z _h ^{\mathrm{tan}} ),
\end{equation*}
and since
$ \widehat{ W } _h ^{\mathrm{nor}} = L ^2 \Lambda ( \partial
\mathcal{T} _h ) $, equation \eqref{eq:weakform_wnor} strongly
enforces the flux relation
\begin{equation}
  \label{eq:ldg-h_znor}
  \widehat{ z } _h ^{\mathrm{nor}} = z _h ^{\mathrm{nor}} - \alpha ( \widehat{ z } _h ^{\mathrm{tan}} - z _h ^{\mathrm{tan}} ) .
\end{equation}
As in \cref{lem:ldg-h_flux},
$ \alpha = \prod _{ K \in \mathcal{T} _h } \alpha _{\partial K} $ is a
bounded symmetric penalty operator on
$ L ^2 \Lambda ( \partial \mathcal{T} _h ) $; we assume further that
$ \alpha _{ \partial K } = \prod _{ e \subset \partial K } \alpha _e
$, so that the penalty may be evaluated facet-by-facet. A typical
choice is the \emph{piecewise-constant penalty}
\begin{equation}
  \label{eq:piecewise_constant_penalty}
  \alpha _e \widehat{ w } ^{\mathrm{tan}}  = \bigoplus _{ k = 0 } ^{ n -1 } \alpha ^k _e \widehat{ w } ^{k ,\mathrm{tan}} ,
\end{equation}
where $ \alpha ^k _e \in \mathbb{R} $. Note that
$ \alpha ^k _{ e ^\pm } $ need not be equal at interior facets, i.e.,
the penalty may be double-valued. Another interesting choice is the
\emph{reduced-stabilization penalty}
\begin{equation}
  \label{eq:reduced_penalty}
  \alpha _e \widehat{ w } ^{\mathrm{tan}}  = \mathbb{P} _e \bigoplus _{ k = 0 } ^{ n -1 } \alpha ^k _e \widehat{ w } ^{k ,\mathrm{tan}} ,
\end{equation}
where $ \mathbb{P} _e $ is $ L ^2 $-orthogonal projection onto
$ \widehat{ W } _h ^{\mathrm{tan}} (e) $. This latter form is inspired
by the work of \citet{Lehrenfeld2010,LeSc2016} and
\citet{Oikawa2015,Oikawa2016}.

\begin{remark}
  \label{rmk:reduced_wnor}
  Reduced stabilization can also be implemented by choosing
  $ \widehat{ W } _h ^{\mathrm{nor}} = \widehat{ W } _h
  ^{\mathrm{tan}} $ rather than
  $ \widehat{ W } _h ^{\mathrm{nor}} = L ^2 \Lambda ( \partial
  \mathcal{T} _h )$. This approach will be discussed in
  \cref{sec:reduced}.
\end{remark}

In practice, the fact that $ \widehat{ W } _h ^{\mathrm{nor}} $ is
infinite-dimensional is not a problem, since we can substitute
\eqref{eq:ldg-h_znor} wherever $ \widehat{ z } _h ^{\mathrm{nor}} $
appears and thereby eliminate \eqref{eq:weakform_wnor}. This yields a
reduced variational problem in the remaining unknowns
$ ( z _h , \widehat{ z } _h ^{\mathrm{tan}} ) \in W _h \times
\widehat{ V } _h ^{\mathrm{tan}} $, which we can write in the
symmetric form
\begin{subequations}
  \label{eq:ldg-h_weakform}
  \begin{align}
    ( z _h , \delta w _h ) _{ \mathcal{T} _h } + ( \delta z _h , w _h ) _{ \mathcal{T} _h } + \langle \widehat{ z } _h ^{\mathrm{tan}} , w _h ^{\mathrm{nor}} \rangle _{ \partial \mathcal{T} _h } + \bigl\langle \alpha ( \widehat{ z } _h ^{\mathrm{tan}} - z _h ^{\mathrm{tan}} ) , w _h ^{\mathrm{tan}} \bigr\rangle _{ \partial \mathcal{T} _h }  &= \bigl( f ( z  _h) , w _h \bigr) _{ \mathcal{T} _h } , \label{eq:ldg-h_weakform_w}\\
    \bigl\langle z _h ^{\mathrm{nor}} - \alpha ( \widehat{ z } _h ^{\mathrm{tan}} - z _h ^{\mathrm{tan}} ) , \widehat{ w } _h ^{\mathrm{tan}} \bigr\rangle _{ \partial \mathcal{T} _h } &= 0 ,\label{eq:ldg-h_weakform_wtan}
  \end{align}
\end{subequations}
for all
$ ( w _h , \widehat{ w } _h ^{\mathrm{tan}} ) \in W _h \times
\ringhat{V} _h ^{\mathrm{tan}} $. Here,
\eqref{eq:ldg-h_weakform_w} results from integrating
\eqref{eq:weakform_w} by parts.

\begin{theorem}
  \label{thm:ldg-h_ms}
  The LDG-H method is multisymplectic. Furthermore, if
  \begin{equation}
    \label{eq:ldg-h_strong}
    \llbracket W _h ^{\mathrm{nor}} \rrbracket + \bigav{ \alpha ( W _h ^{\mathrm{tan}} + \widehat{ V } _h ^{\mathrm{tan}} ) } \subset \ringhat{V} _h ^{\mathrm{tan}} ,
  \end{equation}
  then it is strongly multisymplectic. For the piecewise-constant
  penalty \eqref{eq:piecewise_constant_penalty}, this condition
  reduces to
  \begin{equation}
    \label{eq:ldg-h_strong_piecewise}
    \llbracket W _h ^{\mathrm{nor}} \rrbracket + \av{ W _h ^{\mathrm{tan}} } \subset \ringhat{V} _h ^{\mathrm{tan}} ,
  \end{equation}
  and for the reduced-stabilization penalty \eqref{eq:reduced_penalty} it reduces even further to
  \begin{equation}
    \label{eq:ldg-h_strong_reduced}
    \llbracket W _h ^{\mathrm{nor}} \rrbracket \subset \ringhat{V} _h ^{\mathrm{tan}}.
  \end{equation}
\end{theorem}

\begin{proof}
  Since
  $ \widehat{ W } _h ^{\mathrm{nor}} = L ^2 \Lambda ( \partial
  \mathcal{T} _h ) $ satisfies the hypothesis of
  \cref{lem:ldg-h_flux}, regardless of how we choose the penalty and
  the finite element spaces, we immediately see that LDG-H is
  multisymplectic. Furthermore, since normal jump coincides with
  tangential average, \eqref{eq:ldg-h_strong} implies
  \begin{equation*}
    \llbracket \widehat{ z } _h ^{\mathrm{nor}} \rrbracket = \llbracket z _h ^{\mathrm{nor}} \rrbracket - \bigav{ \alpha ( \widehat{ z } _h ^{\mathrm{tan}} - z _h ^{\mathrm{tan}} ) } \in \ringhat{V} _h ^{\mathrm{tan}} .
  \end{equation*}
  Strong conservativity follows by \cref{lemma:strong.conservativity}
  and strong multisymplecticity by \cref{thm:strong.MS}.

  If $\alpha$ is piecewise constant, then
  $ \alpha _e \widehat{ W } _h ^{\mathrm{tan}} (e) \subset \widehat{ W
  } _h ^{\mathrm{tan}} (e) $ for all
  $ e \subset \partial \mathcal{T} _h $, so
  $ \alpha \widehat{ W } _h ^{\mathrm{tan}} \subset \widehat{ W } _h
  ^{\mathrm{tan}} $. Hence, the condition
  $ \av{ \alpha \widehat{ V } _h ^{\mathrm{tan}} } \subset \av{
    \widehat{ W } _h ^{\mathrm{tan}} } \subset \ringhat{V} _h
  ^{\mathrm{tan}} $ is always satisfied. Furthermore,
  $ \av{ \alpha W _h ^{\mathrm{tan}} } _e \subset \av{ W _h
    ^{\mathrm{tan}} } _e $ holds for all
  $ e \subset \partial \mathcal{T} _h $; we see this by writing
  $ \av{ \alpha z _h } _e = \av{ w _h } _e $, where
  $ w _h \rvert _K = \alpha _e z _h \rvert _K \in W _h (K) $ for
  $ e \subset \partial K $. Therefore,
  $ \av{ W _h ^{\mathrm{tan}} } \subset \ringhat{V} _h ^{\mathrm{tan}}
  $ implies
  $ \av{ \alpha W _h ^{\mathrm{tan}} } \subset \ringhat{V} _h
  ^{\mathrm{tan}} $. Altogether, this shows that
  \eqref{eq:ldg-h_strong_piecewise} suffices to ensure that
  \eqref{eq:ldg-h_strong} holds for all piecewise-constant $\alpha$.

  Finally, if $\alpha$ is the reduced-stabilization penalty, then
  $ \bigav{ \alpha ( W _h ^{\mathrm{tan}} + \widehat{ V } _h
    ^{\mathrm{tan}} ) } \subset \av{ \widehat{ W } _h ^{\mathrm{tan}}
  } \subset \ringhat{V} _h ^{\mathrm{tan}} $, so
  \eqref{eq:ldg-h_strong_reduced} suffices to ensure that
  \eqref{eq:ldg-h_strong} holds.
\end{proof}

\begin{corollary}
  \label{cor:equal-order}
  On a simplicial mesh with piecewise-constant penalty coefficients
  \eqref{eq:piecewise_constant_penalty} and equal-order finite element
  spaces \eqref{eq:equal_order}, the LDG-H method is strongly
  multisymplectic.
\end{corollary}

For the semilinear $k$-form Hodge--Laplace problem,
\eqref{eq:ldg-h_weakform_w} becomes
\begin{alignat*}{2}
  ( \delta u _h , \tau _h ) _{ \mathcal{T} _h } + \bigl\langle \alpha ^{ k -1 } ( \widehat{ \sigma } _h ^{\mathrm{tan}} - \sigma _h ^{\mathrm{tan}} ) , \tau _h ^{\mathrm{tan}} \bigr\rangle _{ \partial \mathcal{T} _h } &= ( \sigma _h , \tau _h ) _{ \mathcal{T} _h } , &\forall \tau _h &\in W _h ^{ k -1 } , \\
  ( \sigma _h , \delta v _h ) _{ \mathcal{T} _h } +   ( \delta \rho _h , v _h ) _{ \mathcal{T} _h } + \langle \widehat{ \sigma } _h ^{\mathrm{tan}} , v _h ^{\mathrm{nor}} \rangle _{ \partial \mathcal{T} _h } + \bigl\langle \alpha ^k ( \widehat{ u } _h ^{\mathrm{tan}} - u _h ^{\mathrm{tan}} ) , v _h ^{\mathrm{tan}} \bigr\rangle _{ \partial \mathcal{T} _h } &= \biggl( \frac{ \partial F }{ \partial u _h } , v _h \biggr) _{ \mathcal{T} _h } ,\  &\forall v _h &\in W _h ^k ,\\
  ( u _h , \delta \eta _h ) _{ \mathcal{T} _h } + \langle \widehat{ u } _h ^{\mathrm{tan}} , \eta _h ^{\mathrm{nor}} \rangle _{ \partial \mathcal{T} _h } &= ( \rho _h , \eta _h ) _{ \mathcal{T} _h } , & \forall \eta _h &\in W _h ^{ k + 1 } ,
\end{alignat*}
while \eqref{eq:ldg-h_weakform_wtan} becomes
\begin{alignat*}{3}
  \bigl\langle u _h ^{\mathrm{nor}} &{}-{}& \alpha ^{ k -1 } ( \widehat{ \sigma } _h ^{\mathrm{tan}} - \sigma _h ^{\mathrm{tan}} ) , \widehat{ \tau } _h ^{\mathrm{tan}} \bigr\rangle _{ \partial \mathcal{T} _h } &= 0, \quad &\forall \widehat{ \tau } _h ^{\mathrm{tan}} &\in \ringhat{V} _h ^{k-1, \mathrm{tan}} ,\\
  \bigl\langle \rho _h ^{\mathrm{nor}} &{}-{}& \alpha ^k ( \widehat{ u } _h ^{\mathrm{tan}} - u _h ^{\mathrm{tan}} ) , \widehat{ v } _h ^{\mathrm{tan}} \bigr\rangle _{ \partial \mathcal{T} _h } &= 0, \quad &\forall \widehat{ v } _h ^{\mathrm{tan}} &\in \ringhat{V} _h ^{k, \mathrm{tan}} .
\end{alignat*}
In the linear case, this agrees (modulo boundary conditions and
harmonic forms) with the LDG-H methods proposed in \citet[Section
8.2.3]{AwFaGuSt2023}---and, when $ \alpha = 0 $, with certain mixed
and nonconforming hybrid methods in \citep[Section
8.2.2]{AwFaGuSt2023}.

\subsubsection{LDG-H methods for scalar problems}
\label{sec:ldg-h_scalar}

When $ k = 0 $ and $ k = n $, this coincides with several known
methods for the scalar Poisson equation, whose strong
multisymplecticity for de~Donder--Weyl systems more generally was
established in \citet{McSt2020}.

For $ k = 0 $, we recover the original LDG-H method in
\citet{CoGoLa2009}. The hypotheses of \cref{thm:ldg-h_ms} are
satisfied by taking
$ \widehat{ W } _h ^{0,\mathrm{tan}} (e) = \mathcal{P} _r \Lambda ^0
(e) $ and any of
\begin{subequations}
  \label{eq:ldg-h_spaces_k=0}
  \begin{alignat}{2}
    W _h ^0 (K) &= \mathcal{P} _r \Lambda ^0 (K) , \qquad & W _h ^1 (K) &= \mathcal{P} _r \Lambda ^1 (K) ,\\
    W _h ^0 (K) &= \mathcal{P} _r \Lambda ^0 (K) , \qquad & W _h ^1 (K) &= \mathcal{P} _{r-1} \Lambda ^1 (K) ,\label{eq:k=0_r_r-1}\\
    W _h ^0 (K) &= \mathcal{P} _{r-1} \Lambda ^0 (K) , \qquad & W _h ^1 (K) &= \mathcal{P} _r \Lambda ^1 (K) , \label{eq:k=0_r-1_r}
  \end{alignat}
\end{subequations}
with piecewise-constant penalty coefficients $ \alpha ^0 $. In the
special case $ \alpha ^0 = 0 $, we obtain the BDM-H method by taking
the function spaces \eqref{eq:k=0_r-1_r} and the RT-H method by
taking
\begin{equation*}
  W _h ^0 (K) = \mathcal{P} _r \Lambda ^0 (K) , \qquad W _h ^1 (K) = \star \mathcal{P} _{r+1} ^- \Lambda ^{ n -1 } (K) ,
\end{equation*}
both of which can be seen as Hodge duals of the corresponding
$ k = n $ AFW-H methods.

\begin{remark}
  \label{rmk:reduced_k=0}
  For the above spaces, there is no difference between the
  piecewise-constant and reduced-stabilization penalties: the relevant
  traces are already in
  $ \widehat{ W } _h ^{0, \mathrm{tan}} (e) = \mathcal{P} _r \Lambda ^0
  (e) $, so projection has no effect. However, the
  reduced-stabilization penalty lets us take \eqref{eq:k=0_r_r-1}
  with
  $ \widehat{ W } _h ^{0, \mathrm{tan}} (e) = \mathcal{P} _{ r -1 }
  \Lambda ^0 (e) $, precisely as in \citet{Lehrenfeld2010,LeSc2016}
  and \citet{Oikawa2015,Oikawa2016}, thereby reducing the number of
  global degrees of freedom.
\end{remark}

For $ k = n $, we get the alternative implementation of LDG-H in
\citet{Cockburn2016}, using local Neumann solvers rather than
Dirichlet solvers. (In the linear case, well-posedness of these local
solvers is achieved using a space of piecewise-constant ``locally
harmonic'' $n$-forms, which is generalized to $k$-forms in
\citep{AwFaGuSt2023}.) We take
$ \widehat{ W } _h ^{ n -1 , \mathrm{tan} } ( e ) = \mathcal{P} _r
\Lambda ^{ n -1 } (e) $ and any of
\begin{subequations}
  \label{eq:ldg-h_spaces_k=n}
  \begin{alignat}{2}
    W _h ^{n-1} (K) &= \mathcal{P} _r \Lambda ^{n-1} (K) , \qquad & W _h ^n (K) &= \mathcal{P} _r \Lambda ^n (K) ,\\
    W _h ^{n-1} (K) &= \mathcal{P} _{r-1} \Lambda ^{n-1} (K) , \qquad & W _h ^n (K) &= \mathcal{P} _r \Lambda ^n (K) , \label{eq:k=n_r-1_r} \\
    W _h ^{n-1} (K) &= \mathcal{P} _r \Lambda ^{n-1} (K) , \qquad & W _h ^n (K) &= \mathcal{P} _{r-1} \Lambda ^n (K) , \label{eq:k=n_r_r-1}
  \end{alignat}
\end{subequations}
which are the respective Hodge duals of \eqref{eq:ldg-h_spaces_k=0},
with piecewise-constant $ \alpha ^{n-1} $.  As in
\cref{rmk:reduced_k=0}, the reduced-stabilization penalty allows us to
take \eqref{eq:k=n_r_r-1} with
$ \widehat{ W } _h ^{n -1 , \mathrm{tan}} (e) = \mathcal{P} _{r - 1}
\Lambda ^{ n -1 } (e) $.

\begin{remark}
  In the special case $ \alpha ^{n-1} = 0 $, choosing
  \eqref{eq:k=n_r-1_r} with
  $ \widehat{ W } _h ^{n -1 , \mathrm{tan}} (e) = \mathcal{P} _{r - 1}
  \Lambda ^{ n -1 } (e) $ recovers the nonconforming primal-hybrid
  method of \citet{RaTh1977_hybrid}. However, this hybridization of
  the method is not strongly conservative, since
  $ \bigl\llbracket (W _h ^n) ^{\mathrm{nor}} \bigr\rrbracket
  \not\subset \ringhat{V} _h ^{n-1, \mathrm{tan}} $. It is equivalent
  to a strongly conservative hybrid method, called NC-H, if we take an
  alternative hybridization with
  $ \widehat{ W } _h ^{\mathrm{nor}} = \widehat{ W } _h
  ^{\mathrm{tan}} $, as in \cref{rmk:reduced_wnor}. We will examine
  such methods further in \cref{sec:reduced}.
\end{remark}

\subsubsection{LDG-H methods for vector problems}

The $ k = 1 $ case of LDG-H includes several proposed HDG methods for
Maxwell's equations \citep{NgPeCo2011,ChQiShSo2017,ChQiSh2018} and for
the vorticity-velocity-pressure formulation of Stokes flow
\citep{CoGo2009}. (Recall the canonical structure of these problems
from \cref{sec:multisymplectic}.) It follows from \cref{thm:ldg-h_ms}
that these methods are multisymplectic.

In particular, the method of \citet*{NgPeCo2011} is obtained with the
equal-order spaces \eqref{eq:equal_order}, whereas the method of
\citet*{ChQiSh2018} is recovered with the choices
\begin{alignat*}{2}
  W _h ^0 (K) &= \mathcal{P} _{r+1} \Lambda ^0 (K) , \qquad
  &W _h ^1 (K) &= \mathcal{P} _r \Lambda ^1 (K), \qquad
  W _h ^2 (K) = \mathcal{P} _r \Lambda ^2 (K) ,\\
  \widehat{ W } _h ^{0, \mathrm{tan}} (e) &= \mathcal{P} _{r+1} \Lambda ^0 (e) ,
  \qquad &\widehat{ W } _h ^{1, \mathrm{tan}} (e) &= \mathcal{P} _r \Lambda ^1 (e).
\end{alignat*}
These clearly satisfy the hypotheses of \cref{thm:ldg-h_ms} when
$\alpha$ is piecewise constant.

Finally, the method of \citet*{ChQiShSo2017} takes the function spaces
\begin{alignat*}{2}
  W _h ^0 (K) &= \mathcal{P} _r \Lambda ^0 (K) , \qquad
  & W _h ^1 (K) &= \mathcal{P} _{ r + 1 } \Lambda ^1 (K) , \qquad
  W _h ^2 (K) = \mathcal{P} _r \Lambda ^2 (K) ,\\
  \widehat{ W } _h ^{0,\mathrm{tan}} (e) &= \mathcal{P} _{ r + 1 } \Lambda ^0 ( e ), \qquad
  & \widehat{ W } _h ^{1, \mathrm{tan}} (e) &= \mathcal{P} _r \Lambda ^1 (e) \oplus \mathrm{d} \mathcal{H} _{ r + 2 } \Lambda ^0 (e) ,
\end{alignat*}
where $ \mathcal{H} _{ r + 2 } \Lambda ^0 (e) $ is the space of
homogeneous degree-$(r+2)$ polynomial $0$-forms on $e$, and takes
$\alpha$ to be the reduced-stabilization penalty. This method is also
strongly multisymplectic by \cref{thm:ldg-h_ms}. Note that the
projection is essential for this to hold, since it ensures that
$ \alpha W _h ^{\mathrm{tan}} \subset \widehat{ W } _h ^{\mathrm{tan}}
$ even though
$ W _h ^{1, \mathrm{tan}} \not\subset \widehat{ W } _h ^{1,
  \mathrm{tan}} $.

\subsubsection{The extended Galerkin (XG) method}
\label{sec:xg}
\citet*{HoLiXu2022} recently introduced an \emph{extended Galerkin}
(XG) method for the Hodge--Laplace problem. In its most general form,
the XG method is not an HDG method; however, it is hybridizable for
certain choices of parameters \citep[Section 6]{HoLiXu2022}, and in
these cases, we will see that it is an LDG-H method.

\begin{remark}
  The hybridization procedure in \citep[Section 6]{HoLiXu2022}
  eliminates tangential traces to get a system only involving normal
  traces---whereas in our approach, as in \citep{AwFaGuSt2023}, the
  roles of tangential and normal traces are switched. One can easily
  pass back and forth between these two approaches by taking the Hodge
  star, and they are ultimately equivalent.
\end{remark}

The XG method may be extended from the Hodge--Laplace problem to
canonical PDEs more generally as follows. In addition to the broken
space $ W _h $, we specify two broken ``check'' spaces
\begin{equation*}
  \check{W} _h ^{\mathrm{nor}} \coloneqq \prod _{ e \subset \partial \mathcal{T} _h } \check{W} _h ^{\mathrm{nor}} (e) , \qquad \check{W} _h ^{\mathrm{tan}} \coloneqq \prod _{ e \subset \partial \mathcal{T} _h } \check{W} _h ^{\mathrm{tan}} (e) ,
\end{equation*}
where
$ \check{W} _h ^{\mathrm{nor}} (e) , \check{W} _h ^{\mathrm{tan}} (e)
\subset L ^2 \Lambda (e) $. Define the normally and tangentially
single-valued subspaces
\begin{align*}
  \check{ V } _h ^{\mathrm{nor}} &\coloneqq \bigl\{ \check{w} _h ^{\mathrm{nor}} \in \check{ W } _h ^{\mathrm{nor}} : \llbracket \check{w} _h ^{\mathrm{nor}} \rrbracket = 0 \bigr\} ,\\
  \mathring{\check{V}} _h ^{\mathrm{tan}} &\coloneqq \bigl\{ \check{w} _h ^{\mathrm{tan}} \in \check{ W } _h ^{\mathrm{tan}} : \llbracket \check{w} _h ^{\mathrm{tan}} \rrbracket = 0 \bigr\},\\
  \check{V} _h ^{\mathrm{tan}} &\coloneqq \bigl\{ \check{w} _h ^{\mathrm{tan}} \in \check{ W } _h ^{\mathrm{tan}} : \llbracket \check{w} _h ^{\mathrm{tan}} \rrbracket = 0 \text{ on } \partial \mathcal{T} _h \setminus \partial \Omega \bigr\}.
\end{align*}
The XG method seeks
$ ( z _h , \check{z} _h ^{\mathrm{nor}} , \check{z} _h ^{\mathrm{tan}}
) \in W _h \times \check{V} _h ^{\mathrm{nor}} \times \check{V} _h
^{\mathrm{tan}} $ such that
\begin{subequations}
  \label{eq:xg}
  \begin{alignat}{2}
    ( z _h , \mathrm{D} w _h ) _{ \mathcal{T} _h } + [ \widehat{ z } _h , w _h ] _{ \partial \mathcal{T} _h } &= \bigl( f ( z _h ) , w _h \bigr) _{ \mathcal{T} _h } , \quad &\forall w _h &\in W _h , \label{eq:xg_w} \\
    \bigl\langle \check{z} _h ^{\mathrm{nor}} - \alpha \llbracket z _h ^{\mathrm{tan}} \rrbracket , \check{w} _h ^{\mathrm{nor}} \bigr\rangle _{ \partial \mathcal{T} _h } + \langle \alpha \check{z} _h ^{\mathrm{tan}} , \check{w} _h ^{\mathrm{nor}} \rangle _{ \partial \Omega } &= 0, \quad &\forall \check{w} _h ^{\mathrm{nor}} &\in \check{V} _h ^{\mathrm{nor}} , \label{eq:xg_wnor} \\
    \bigl\langle \check{z} _h ^{\mathrm{tan}} - \beta \llbracket z _h ^{\mathrm{nor}} \rrbracket , \check{w} _h ^{\mathrm{tan}} \bigr\rangle _{ \partial \mathcal{T} _h }  &= 0, \quad &\forall \check{w} _h ^{\mathrm{tan}} &\in \mathring{\check{V}} _h ^{\mathrm{tan}} \label{eq:xg_wtan},
  \end{alignat}
\end{subequations}
where $ \alpha ^k _e , \beta ^k _e \in \mathbb{R} $ are single-valued
penalty parameters for each $k = 0 , \ldots, n -1 $ and
$ e \subset \partial \mathcal{T} _h $, i.e.,
$ \alpha ^k _{ e ^+ } = \alpha ^k _{ e ^- } $ and
$ \beta ^k _{e ^+} = \beta ^k _{ e^- } $ at interior facets, and where
the numerical traces are
\begin{equation*}
  \widehat{ z } _h ^{\mathrm{nor}} \coloneqq \av{ z _h ^{\mathrm{nor}} } + \check{z} _h ^{\mathrm{nor}} , \qquad \widehat{ z } _h ^{\mathrm{tan}} \coloneqq \av{ z _h ^{\mathrm{tan}} } + \check{z} _h ^{\mathrm{tan}} .
\end{equation*}
Following \citep[Equations 3.13 and 6.7]{HoLiXu2022}, we assume that
\begin{subequations}
  \label{eq:xg_inclusions}
  \begin{align}
    \llbracket W _h ^{\mathrm{tan}} + \check{V} _h ^{\mathrm{tan}} \rrbracket &\subset \check{V} _h ^{\mathrm{nor}} , \label{eq:xg_inclusion_nor}\\
  \llbracket W _h ^{\mathrm{nor}} \rrbracket + \av{ W _h ^{\mathrm{tan}} } &\subset \mathring{\check{V}} _h ^{\mathrm{tan}} , \label{eq:xg_inclusion_tan}
  \end{align}
\end{subequations}
which implies in particular that
\eqref{eq:xg_wnor}--\eqref{eq:xg_wtan} hold strongly, i.e.,
\begin{equation*}
  \check{z} _h ^{\mathrm{nor}} = - \alpha \llbracket \check{z} _h ^{\mathrm{tan}} - z _h ^{\mathrm{tan}} \rrbracket ,
  \qquad \av{\check{z} _h ^{\mathrm{tan}}} = \beta \llbracket z _h ^{\mathrm{nor}} \rrbracket .
\end{equation*}
Since $ \check{z} _h ^{\mathrm{nor}} $ and
$ \check{z} _h ^{\mathrm{tan}} $ are single-valued, this means that
\begin{equation*}
  \check{z} _h ^{\mathrm{nor}} =
  \begin{cases}
    \alpha \llbracket z _h ^{\mathrm{tan}} \rrbracket & \text{on } \partial \mathcal{T} _h \setminus \partial \Omega ,\\
    - \alpha \check{z} _h ^{\mathrm{tan}} & \text{on } \partial \Omega ,
  \end{cases} \qquad \check{z} _h ^{\mathrm{tan}} = \beta \llbracket z _h ^{\mathrm{nor}} \rrbracket \quad \text{on } \partial \mathcal{T} _h \setminus \partial \Omega ,
\end{equation*}
where $ \check{z} _h ^{\mathrm{tan}} $ is free on $ \partial \Omega $
to allow for arbitrary boundary conditions. Note that
\eqref{eq:xg_inclusion_tan} corresponds to the strong conservativity
condition \eqref{eq:ldg-h_strong_piecewise} for LDG-H.

\begin{remark}
  Since $ \check{V} _h ^{\mathrm{tan}} $ is single-valued, the
  inclusion
  $ \llbracket \check{V} _h ^{\mathrm{tan}} \rrbracket \subset
  \check{V} _h ^{\mathrm{nor}} $ in \eqref{eq:xg_inclusion_nor} only
  imposes conditions at the boundary. Specifically,
  $ \llbracket \check{z} _h ^{\mathrm{tan}} \rrbracket $ is the
  extension by zero of the boundary values
  $ \check{z} _h ^{\mathrm{tan}} \rvert _{ \partial \Omega } $, and
  this needs to be in $ \check{W} _h ^{\mathrm{nor}} $ to enforce
  \eqref{eq:xg_wnor} strongly on $ \partial \Omega $. Therefore, a
  necessary and sufficient condition for this inclusion is
  $ \check{W} _h ^{\mathrm{tan}} (e) \subset \check{W} _h
  ^{\mathrm{nor}} (e) $ for all $ e \subset \partial \Omega $.
\end{remark}

The following generalizes the discussion of hybridizable XG methods
from \citep[Section 6]{HoLiXu2022} for the Hodge--Laplace problem,
extending it to the more general class of problems considered here. We
note that factors of $ \frac{ 1 }{ 4 } $ appear in \citep{HoLiXu2022}
that do not appear in our calculations, since these are absorbed by
the $ \frac{1}{2} $ factor in \cref{def:jump_avg} of normal and
tangential jump.

\begin{theorem}
  \label{thm:xg_ldg-h}
  If $ \alpha \beta = 1 $ at interior facets (i.e.,
  $ \alpha ^k _e \beta _e ^k = 1 $) and \eqref{eq:xg_inclusions}
  holds, then the XG method is equivalent to the corresponding LDG-H
  method with penalty $\alpha$ and
  $ \widehat{ V } _h ^{\mathrm{tan}} = \check{V} _h ^{\mathrm{tan}} $.
\end{theorem}

\begin{proof}
  If
  $ ( z _h , \check{z} _h ^{\mathrm{nor}}, \check{z} _h
  ^{\mathrm{tan}} ) $ is an XG solution, then
  $ \widehat{ z } _h ^{\mathrm{nor}} $ and
  $ \widehat{ z } _h ^{\mathrm{tan}} $ are single-valued and satisfy
  \begin{align*}
    \widehat{ z } _h ^{\mathrm{nor}} - z _h ^{\mathrm{nor}}
    &= \av{ \widehat{ z } _h ^{\mathrm{nor}} - z _h ^{\mathrm{nor}} } + \llbracket \widehat{ z } _h ^{\mathrm{nor}} - z _h ^{\mathrm{nor}} \rrbracket \\
    &= \check{z} _h ^{\mathrm{nor}} - \llbracket z _h ^{\mathrm{nor}} \rrbracket \\
    &= - \alpha \llbracket \check{z} _h ^{\mathrm{tan}} - z _h ^{\mathrm{tan}} \rrbracket - \alpha \av{ \check{z} _h ^{\mathrm{tan}} } \\
    &= - \alpha \bigl( \llbracket \widehat{ z } _h ^{\mathrm{tan}} - z _h ^{\mathrm{tan}} \rrbracket + \av { \widehat{ z } _h ^{\mathrm{tan}} - z _h ^{\mathrm{tan}} } \bigr) \\
    &= - \alpha \bigl( \widehat{ z } _h ^{\mathrm{tan}} - z _h ^{\mathrm{tan}} \bigr) .
  \end{align*}
  Since \eqref{eq:xg_inclusion_tan} implies
  $ \widehat{ z } _h ^{\mathrm{tan}} \in \check{V} _h ^{\mathrm{tan}}
  = \widehat{ V } _h ^{\mathrm{tan}} $, it follows that
  $ ( z _h , \widehat{ z } _h ^{\mathrm{nor}} , \widehat{ z } _h
  ^{\mathrm{tan}} ) $ is an LDG-H solution.

  Conversely, suppose
  $ ( z _h , \widehat{ z } _h ^{\mathrm{nor}} , \widehat{ z } _h
  ^{\mathrm{tan}} ) $ is an LDG-H solution, so in particular
  \begin{align*}
    \widehat{ z } _h ^{\mathrm{nor}} - z _h ^{\mathrm{nor}} &= - \alpha ( \widehat{ z } _h ^{\mathrm{tan}} - z _h ^{\mathrm{tan}} ) \quad \text{on } \partial \mathcal{T} _h ,\\
    \widehat{ z } _h ^{\mathrm{tan}} - z _h ^{\mathrm{tan}} &= - \beta ( \widehat{ z } _h ^{\mathrm{nor}} - z _h ^{\mathrm{nor}} ) \quad \text{on } \partial \mathcal{T} _h \setminus \partial \Omega 
  \end{align*}
  We have strong conservativity, by \eqref{eq:xg_inclusion_nor} and
  \cref{lemma:strong.conservativity}, so
  $ \widehat{ z } _h ^{\mathrm{nor}} $ and
  $ \widehat{ z } _h ^{\mathrm{tan}} $ are both
  single-valued. Therefore, so are
  \begin{equation*}
    \check{z} _h ^{\mathrm{nor}} \coloneqq \widehat{ z } _h ^{\mathrm{nor}} - \av{ z _h ^{\mathrm{nor}} } , \qquad \check{z} _h ^{\mathrm{tan}} \coloneqq \widehat{ z } _h ^{\mathrm{tan}} - \av{ z _h ^{\mathrm{tan}} } ,
  \end{equation*}
  and \eqref{eq:xg_inclusion_tan} implies
  $ \check{z} _h ^{\mathrm{tan}} \in \check{V} _h ^{\mathrm{tan}}
  $. Since normal average coincides with tangential jump, observe that
  \begin{equation*}
    \check{z} _h ^{\mathrm{nor}} = \av{ \widehat{ z } _h ^{\mathrm{nor}} - z _h ^{\mathrm{nor}} } = - \alpha \llbracket \widehat{ z } _h ^{\mathrm{tan}} - z _h ^{\mathrm{tan}} \rrbracket = - \alpha \llbracket \check{z} _h ^{\mathrm{tan}} - z _h ^{\mathrm{tan}} \rrbracket \in \check{V} _h ^{\mathrm{nor}} ,
  \end{equation*}
  where the inclusion is by \eqref{eq:xg_inclusion_nor}. Similarly,
  since tangential average coincides with normal jump,
  \begin{equation*}
    \av{ \check{z} _h ^{\mathrm{tan}} } = \av{ \widehat{ z } _h ^{\mathrm{tan}} - z _h ^{\mathrm{tan}} } = - \beta \llbracket \widehat{ z } _h ^{\mathrm{nor}} - z _h ^{\mathrm{nor}} \rrbracket = \beta \llbracket z _h ^{\mathrm{nor}} \rrbracket .
  \end{equation*}
  Therefore,
  $ ( z _h , \check{z} _h ^{\mathrm{nor}} , \check{z} _h
  ^{\mathrm{tan}} ) $ is an XG solution.
\end{proof}

\begin{corollary}
  With the hypotheses of \cref{thm:xg_ldg-h}, the XG method is
  strongly multisymplectic.
\end{corollary}

\begin{proof}
  Follows directly from \cref{thm:ldg-h_ms} and \cref{thm:xg_ldg-h}.
\end{proof}

\subsection{The IP-H method}

We now discuss a close relative of the LDG-H method, which generalizes
the hybrid interior-penalty (IP-H) method in \citet{CoGoLa2009} to the
semilinear Hodge--Laplace problem. For the linear Hodge--Laplace
problem, this coincides with the IP-H method in \citet[Section
8.2.3]{AwFaGuSt2023}. Unlike the previous methods in this section,
which can be applied to more general canonical systems, the IP-H
method is specific to the Hodge--Laplace problem.

For the $k$-form semilinear Hodge--Laplace problem, define the IP-H
flux to be
\begin{equation*}
  \Phi ( z _h , \widehat{ z } _h ) = \Bigl[ ( \widehat{ u } _h
  ^{\mathrm{nor}} - u _h ^{\mathrm{nor}} ) + \alpha ^{ k -1 } \bigl(  \widehat{ \sigma } _h ^{\mathrm{tan}} - ( \delta u _h ) ^{\mathrm{tan}} \bigr) \Bigr] \oplus
  \Bigl[ \bigl(  \widehat{ \rho } _h ^{\mathrm{nor}} - (\mathrm{d} u _h ) ^{\mathrm{nor}} \bigr) + \alpha ^k ( \widehat{ u } _h ^{\mathrm{tan}} - u _h ^{\mathrm{tan}} ) \Bigr],
\end{equation*}
where $ \alpha ^{ k -1 } $ and $ \alpha ^k $ are symmetric penalty
operators on $ L ^2 \Lambda ^{ k -1 } ( \partial \mathcal{T} _h ) $
and $ L ^2 \Lambda ^k ( \partial \mathcal{T} _h ) $, respectively. We
write $ \alpha = \alpha ^{ k -1 } \oplus \alpha ^k $, which will allow
us to suppress the superscripts. As with LDG-H, we choose
$ \widehat{ W } _h ^{\mathrm{nor}} = L ^2 \Lambda ( \partial
\mathcal{T} _h ) $, which strongly enforces
\begin{equation*}
  \widehat{ u } _h ^{\mathrm{nor}} = u _h ^{\mathrm{nor}} - \alpha \bigl(  \widehat{ \sigma } _h ^{\mathrm{tan}} - ( \delta u _h ) ^{\mathrm{tan}} \bigr) , \qquad \widehat{ \rho } _h ^{\mathrm{nor}} = ( \mathrm{d} u _h ) ^{\mathrm{nor}} - \alpha ( \widehat{ u } _h ^{\mathrm{tan}} - u _h ^{\mathrm{tan}} ) .
\end{equation*}
For the weak form \eqref{eq:weakform_HL}, integrating the first and
last equations by parts gives
\begin{subequations}
  \label{eq:ip-h}
  \begin{alignat}{2}
    ( \sigma _h - \delta u _h , \tau _h )_{ \mathcal{T} _h } + \langle \widehat{ u } _h ^{\mathrm{nor}} - u _h ^{\mathrm{nor}} , \tau _h ^{\mathrm{tan}} \rangle _{ \partial \mathcal{T} _h } &= 0 , \quad & \forall \tau _h &\in W _h ^{ k -1 } , \label{eq:ip-h_tau}\\
    ( \sigma _h , \delta v _h ) _{ \mathcal{T} _h } + ( \rho _h ,
    \mathrm{d} v _h ) _{ \mathcal{T} _h } + \langle \widehat{ \sigma }
    _h ^{\mathrm{tan}} , v _h ^{\mathrm{nor}} \rangle _{ \partial
      \mathcal{T} _h } - \langle \widehat{ \rho } _h ^{\mathrm{nor}} ,
    v _h ^{\mathrm{tan}} \rangle _{ \partial \mathcal{T} _h } &=
    \biggl( \frac{ \partial F }{ \partial u _h } , v _h \biggr) _{ \mathcal{T} _h } , \quad &\forall v _h &\in W _h ^k , \label{eq:ip-h_v}\\
    ( \rho _h - \mathrm{d} u _h , \eta _h ) _{ \mathcal{T} _h } - \langle \widehat{ u } _h ^{\mathrm{tan}} - u _h ^{\mathrm{tan}} , \eta _h ^{\mathrm{nor}} \rangle _{ \partial \mathcal{T} _h } &= 0 , \quad &\forall \eta _h &\in W _h ^{ k + 1 } \label{eq:ip-h_eta},
  \end{alignat}
\end{subequations}
and doing similarly with \eqref{eq:weakvar_HL} shows that variations satisfy
\begin{subequations}
  \label{eq:ip-h_var}
  \begin{alignat}{2}
    ( \tau _i - \delta v _i , \tau _h )_{ \mathcal{T} _h } + \langle \widehat{ v } _i ^{\mathrm{nor}} - v_i ^{\mathrm{nor}} , \tau _h ^{\mathrm{tan}} \rangle _{ \partial \mathcal{T} _h } &= 0 , \quad & \forall \tau _h &\in W _h ^{ k -1 } , \label{eq:ip-h_var_tau}\\
    ( \tau _i , \delta v _h ) _{ \mathcal{T} _h } + ( \eta _i ,
    \mathrm{d} v _h ) _{ \mathcal{T} _h } + \langle \widehat{ \tau }
    _i ^{\mathrm{tan}} , v _h ^{\mathrm{nor}} \rangle _{ \partial
      \mathcal{T} _h } - \langle \widehat{ \eta } _i ^{\mathrm{nor}} ,
    v _h ^{\mathrm{tan}} \rangle _{ \partial \mathcal{T} _h } &=
    \biggl( \frac{ \partial ^2 F }{ \partial u _h ^2 } v _i , v _h \biggr) _{ \mathcal{T} _h } , \quad &\forall v _h &\in W _h ^k , \label{eq:ip-h_var_v}\\
    ( \eta _i - \mathrm{d} v _i , \eta _h ) _{ \mathcal{T} _h } - \langle \widehat{ v } _i ^{\mathrm{tan}} - v _i ^{\mathrm{tan}} , \eta _h ^{\mathrm{nor}} \rangle _{ \partial \mathcal{T} _h } &= 0 , \quad &\forall \eta _h &\in W _h ^{ k + 1 } \label{eq:ip-h_var_eta}.
  \end{alignat}
\end{subequations}

\begin{theorem}
  \label{thm:ip-h_ms}
  If $ \delta W _h ^k \subset W _h ^{ k -1 } $ and
  $ \mathrm{d} W _h ^k \subset W _h ^{ k + 1 } $, then the IP-H method
  for the semilinear Hodge--Laplace problem is
  multisymplectic. Furthermore, if additionally
  \begin{subequations}
    \label{eq:ip-h_strong}
    \begin{align} 
      \bigl\llbracket (W _h ^k) ^\mathrm{nor} \bigr\rrbracket + \Bigav{ \alpha \bigl( ( \delta W _h ^k ) ^{\mathrm{tan}} + \widehat{ V } _h ^{k-1,\mathrm{tan}} \bigr) } &\subset \ringhat{V} _h ^{k -1 , \mathrm{tan}} , \label{eq:ip-h_strong_k-1} \\
      \bigl\llbracket (\mathrm{d} W _h ^k) ^{\mathrm{nor}} \bigr\rrbracket + \bigav{ \alpha ( W _h ^{k,\mathrm{tan}} + \widehat{ V } _h ^{k, \mathrm{tan}} ) } &\subset \ringhat{V} _h ^{k , \mathrm{tan}} , \label{eq:ip-h_strong_k}
    \end{align} 
  \end{subequations}
  then it is strongly multisymplectic. For the piecewise-constant
  penalty \eqref{eq:piecewise_constant_penalty}, the conditions reduce
  to
  \begin{subequations}
    \label{eq:ip-h_strong_piecewise}
    \begin{align} 
      \bigl\llbracket (W _h ^k) ^\mathrm{nor} \bigr\rrbracket + \bigav{ ( \delta W _h ^k ) ^{\mathrm{tan}} } &\subset \ringhat{V} _h ^{k -1 , \mathrm{tan}} , \label{eq:ip-h_strong_piecewise_k-1} \\
      \bigl\llbracket (\mathrm{d} W _h ^k) ^{\mathrm{nor}} \bigr\rrbracket + \av{ W _h ^{k,\mathrm{tan}} } &\subset \ringhat{V} _h ^{k , \mathrm{tan}} , \label{eq:ip-h_strong_piecewise_k}
    \end{align} 
  \end{subequations}
  and for the reduced-stabilization penalty \eqref{eq:reduced_penalty}
  they reduce even further to
  \begin{subequations}
    \label{eq:ip-h_strong_reduced}
    \begin{align} 
      \bigl\llbracket (W _h ^k) ^\mathrm{nor} \bigr\rrbracket &\subset \ringhat{V} _h ^{k -1 , \mathrm{tan}} , \label{eq:ip-h_strong_reduced_k-1} \\
      \bigl\llbracket (\mathrm{d} W _h ^k) ^{\mathrm{nor}} \bigr\rrbracket &\subset \ringhat{V} _h ^{k , \mathrm{tan}} . \label{eq:ip-h_strong_reduced_k}
    \end{align} 
  \end{subequations}
\end{theorem}

\begin{proof}
  Let $ ( w _1 , \widehat{ w } _1 ) $ and
  $ ( w _2 , \widehat{ w } _2 ) $ be a pair of first variations, where
  \begin{equation*}
    w _i = \tau _i \oplus v _i \oplus \eta _i , \qquad \widehat{ w } _i ^{\mathrm{nor}} = \widehat{ v } _i ^{\mathrm{nor}} \oplus \widehat{ \eta } _i ^{\mathrm{nor}} , \qquad \widehat{ w } _i ^{\mathrm{tan}} = \widehat{ \tau } _i ^{\mathrm{tan}} \oplus \widehat{ v } _i ^{\mathrm{tan}} .
  \end{equation*}
  For (local) multisymplecticity, recall from \cref{ex:mscl_jump_HL}
  that it suffices to show that
  \begin{equation*}
    \langle \widehat{ \tau } _1 ^{\mathrm{tan}} - \tau _1 ^{\mathrm{tan}} , \widehat{ v } _2 ^{\mathrm{nor}} - v _2 ^{\mathrm{nor}} \rangle _{ \partial \mathcal{T} _h } + \langle \widehat{ v } _1 ^{\mathrm{tan}} - v _1 ^{\mathrm{tan}} , \widehat{ \eta } _2 ^{\mathrm{nor}} - \eta _2 ^{\mathrm{nor}} \rangle _{ \partial \mathcal{T} _h } 
  \end{equation*}
  is symmetric in the indices $ i = 1, 2 $. To see this, we begin by writing
  \begin{equation*}
    \langle \widehat{ \tau } _1 ^{\mathrm{tan}} - \tau _1 ^{\mathrm{tan}} , \widehat{ v } _2 ^{\mathrm{nor}} - v _2 ^{\mathrm{nor}} \rangle _{ \partial \mathcal{T} _h } = \bigl\langle \widehat{ \tau } _1 ^{\mathrm{tan}} - ( \delta v _1 ) ^{\mathrm{tan}} , \widehat{ v } _2 ^{\mathrm{nor}} - v _2 ^{\mathrm{nor}} \bigr\rangle _{ \partial \mathcal{T} _h } + \bigl\langle ( \delta v _1 ) ^{\mathrm{tan}} - \tau _1 ^{\mathrm{tan}} , \widehat{ v } _2 ^{\mathrm{nor}} - v _2 ^{\mathrm{nor}} \bigr\rangle _{ \partial \mathcal{T} _h }.
  \end{equation*}
  For the first term, the numerical flux relation for
  $ \widehat{ v } _2 ^{\mathrm{nor}} $ (which is a variation of
  $ \widehat{ u } _h ^{\mathrm{nor}} $) gives
  \begin{equation*}
    \bigl\langle \widehat{ \tau } _1 ^{\mathrm{tan}} - ( \delta v _1 ) ^{\mathrm{tan}} , \widehat{ v } _2 ^{\mathrm{nor}} - v _2 ^{\mathrm{nor}} \bigr\rangle _{ \partial \mathcal{T} _h } = - \Bigl\langle \widehat{ \tau } _1 ^{\mathrm{tan}} - ( \delta v _1 ) ^{\mathrm{tan}} , \alpha \bigl( \widehat{ \tau } _2 ^{\mathrm{tan}} - ( \delta v _2 ) ^{\mathrm{tan}} \bigr) \Bigr\rangle _{ \partial \mathcal{T} _h } .
  \end{equation*}
  For the second term, $ \delta W _h ^k \subset W _h ^{ k -1 } $
  allows us to take $ \tau _h = \delta v _1 - \tau _1 $ in
  \eqref{eq:ip-h_var_tau} with $ i = 2 $, giving
  \begin{equation*}
    \bigl\langle ( \delta v _1 ) ^{\mathrm{tan}} - \tau _1 ^{\mathrm{tan}} , \widehat{ v } _2 ^{\mathrm{nor}} - v _2 ^{\mathrm{nor}} \bigr\rangle _{ \partial \mathcal{T} _h } = ( \delta v _1 - \tau _1, \delta v _2 - \tau _2 ) _{ \mathcal{T} _h } .
  \end{equation*}
  Both terms are symmetric in $ i = 1, 2 $, and hence so is their
  sum. Similarly, we write
  \begin{equation*}
    \langle \widehat{ v } _1 ^{\mathrm{tan}} - v _1 ^{\mathrm{tan}} , \widehat{ \eta } _2 ^{\mathrm{nor}} - \eta _2 ^{\mathrm{nor}} \rangle _{ \partial \mathcal{T} _h } = \bigl\langle \widehat{ v } _1 ^{\mathrm{tan}} - v _1 ^{\mathrm{tan}} , \widehat{ \eta } _2 ^{\mathrm{nor}} - ( \mathrm{d} v _2 ) ^{\mathrm{nor}} \bigr\rangle _{ \partial \mathcal{T} _h } + \bigl\langle \widehat{ v } _1 ^{\mathrm{tan}} - v _1 ^{\mathrm{tan}} , ( \mathrm{d} v _2 ) ^{\mathrm{nor}} - \eta _2 ^{\mathrm{nor}} \bigr\rangle _{ \partial \mathcal{T} _h }.
  \end{equation*}
  For the first term, the numerical flux relation for
  $ \widehat{ \eta } _2 ^{\mathrm{nor}} $ (which is a variation of
  $ \widehat{ \rho } _h ^{\mathrm{nor}} $) gives
  \begin{equation*}
    \bigl\langle \widehat{ v } _1 ^{\mathrm{tan}} - v _1 ^{\mathrm{tan}} , \widehat{ \eta } _2 ^{\mathrm{nor}} - ( \mathrm{d} v _2 ) ^{\mathrm{nor}} \bigr\rangle _{ \partial \mathcal{T} _h } = - \bigl\langle \widehat{ v } _1 ^{\mathrm{tan}} - v _1 ^{\mathrm{tan}} , \alpha ( \widehat{ v } _2 ^{\mathrm{tan}} - v _2 ^{\mathrm{tan}} )\bigr\rangle _{ \partial \mathcal{T} _h }.
  \end{equation*}
  For the second term, $ \mathrm{d} W _h ^k \subset W _h ^{ k + 1 } $
  allows us to take $ \eta _h = \mathrm{d} v _2 - \eta _2 $ in
  \eqref{eq:ip-h_var_eta} with $ i = 1 $, giving
  \begin{equation*}
    \bigl\langle \widehat{ v } _1 ^{\mathrm{tan}} - v _1 ^{\mathrm{tan}} , ( \mathrm{d} v _2 ) ^{\mathrm{nor}} - \eta _2 ^{\mathrm{nor}} \bigr\rangle _{ \partial \mathcal{T} _h } = - ( \mathrm{d} v _1 - \eta _1 , \mathrm{d} v _2 - \eta _2 ) _{ \mathcal{T} _h } .
  \end{equation*}
  Again, both terms are symmetric in $ i = 1, 2 $, which completes the
  proof of local multisymplecticity.

  Finally, proceeding as in \cref{thm:ldg-h_ms}, we see that
  \eqref{eq:ip-h_strong_k-1} implies
  $ \llbracket \widehat{ u } _h ^{\mathrm{nor}} \rrbracket \in
  \ringhat{V} _h ^{k-1, \mathrm{tan}} $ and \eqref{eq:ip-h_strong_k}
  implies
  $ \llbracket \widehat{ \rho } _h ^{\mathrm{nor}} \rrbracket \in
  \ringhat{V} _h ^{k, \mathrm{tan}} $. Therefore, if both conditions
  hold, we conclude that the IP-H method is strongly conservative by
  \cref{lemma:strong.conservativity} and thus strongly multisymplectic
  by \cref{thm:strong.MS}. For the piecewise-constant and
  reduced-stabilization forms of the penalty, the simplified
  conservativity conditions are obtained exactly as in the proof of
  \cref{thm:ldg-h_ms}.
\end{proof}

\begin{remark}
  \label{rmk:ip-h_ldg-h}
  Note that, if the hypotheses
  $ \delta W _h ^k \subset W _h ^{ k -1 } $ and
  $ \mathrm{d} W _h ^k \subset W _h ^{ k + 1 } $ hold, then the
  conservativity conditions for LDG-H automatically imply those for
  IP-H: \eqref{eq:ldg-h_strong} implies \eqref{eq:ip-h_strong},
  \eqref{eq:ldg-h_strong_piecewise} implies
  \eqref{eq:ip-h_strong_piecewise}, and
  \eqref{eq:ldg-h_strong_reduced} implies
  \eqref{eq:ip-h_strong_reduced}.
\end{remark}

\begin{corollary}
  On a simplicial mesh with piecewise-constant penalty coefficients
  \eqref{eq:piecewise_constant_penalty} and equal-order finite element
  spaces \eqref{eq:equal_order}, the IP-H method is strongly
  multisymplectic.
\end{corollary}

\subsubsection{IP-H as a primal interior-penalty DG method}

We now illustrate how IP-H corresponds to a primal interior-penalty DG
method in cases where the hybrid variables may be eliminated. Suppose
for simplicity that $\alpha$ is a piecewise-constant penalty
\eqref{eq:piecewise_constant_penalty} whose coefficients are nonzero
and single-valued. As in \cref{sec:xg}, let
$ \beta \coloneqq \alpha ^{-1} $ be the piecewise-constant penalty
whose coefficients are the reciprocals of those for
$\alpha$. Furthermore, since we need to impose boundary conditions in
order to eliminate $ \widehat{ \sigma } _h ^{\mathrm{tan}} $ and
$ \widehat{ u } _h ^{\mathrm{tan}} $ on $ \partial \Omega $, suppose
that we have homogeneous essential boundary conditions
$ \widehat{ \sigma } _h ^{\mathrm{tan}} , \widehat{ u } _h
^{\mathrm{tan}} \in \ringhat{V} _h ^{\mathrm{tan}} $. A similar
approach can be taken, e.g., with natural boundary conditions, and the
resulting expressions differ only on $ \partial \Omega $. Finally,
assume that the hypotheses of \cref{thm:ip-h_ms} hold.

Our goal is to express each of the terms on the left-hand side of
\eqref{eq:ip-h_v} in terms of the primal variable $ u _h $. First,
taking $ \tau _h = \delta v _h $ in \eqref{eq:ip-h_tau}, we get
\begin{align*} 
  ( \sigma _h , \delta v _h ) _{ \mathcal{T} _h }
  &= ( \delta u _h , \delta v _h ) _{ \mathcal{T} _h } - \Bigl\langle \llbracket  \widehat{ u } _h ^{\mathrm{nor}} - u _h ^{\mathrm{nor}} \rrbracket , \bigav{ (\delta v _h ) ^{\mathrm{tan}} } \Bigr\rangle _{ \partial \mathcal{T} _h } + \Bigl\langle \alpha \bigl\llbracket \widehat{ \sigma } _h ^{\mathrm{tan}} - ( \delta u _h ) ^{\mathrm{tan}} \bigr\rrbracket, \bigl\llbracket (\delta v _h ) ^{\mathrm{tan}} \bigr\rrbracket \Bigr\rangle _{ \partial \mathcal{T} _h }  \\
  &= ( \delta u _h , \delta v _h ) _{ \mathcal{T} _h } + \Bigl\langle \llbracket u _h ^{\mathrm{nor}} \rrbracket , \bigav{ (\delta v _h ) ^{\mathrm{tan}} } \Bigr\rangle _{ \partial \mathcal{T} _h } - \Bigl\langle \alpha \bigl\llbracket ( \delta u _h ) ^{\mathrm{tan}} \bigr\rrbracket, \bigl\llbracket (\delta v _h ) ^{\mathrm{tan}} \bigr\rrbracket \Bigr\rangle _{ \partial \mathcal{T} _h } ,
\end{align*} 
where the first line uses \cref{prop:average-jump} and the flux
relation for $ \widehat{ u } _h ^{\mathrm{nor}} $, and the second line
uses $ \llbracket \widehat{ u } _h ^{\mathrm{nor}} \rrbracket = 0 $
and
$ \llbracket \widehat{ \sigma } _h ^{\mathrm{tan}} \rrbracket = 0 $.
Similarly, taking $ \eta _h = \mathrm{d} v _h $ in \eqref{eq:ip-h_eta}
and using $ \beta = \alpha ^{-1} $, we get
\begin{equation*} 
  ( \rho _h , \mathrm{d} v _h ) _{ \mathcal{T} _h }
  = ( \mathrm{d} u _h , \mathrm{d} v _h ) _{ \mathcal{T} _h } - \Bigl\langle \llbracket u _h ^{\mathrm{tan}} \rrbracket , \bigav{ ( \mathrm{d} v _h ) ^{\mathrm{nor}} } \Bigr\rangle _{ \partial \mathcal{T} _h } + \Bigl\langle \beta \bigl\llbracket ( \mathrm{d} u _h ) ^{\mathrm{nor}} \bigr\rrbracket  , \bigl\llbracket ( \mathrm{d} v _h ) ^{\mathrm{nor}} \bigr\rrbracket \Bigr\rangle _{ \partial \mathcal{T} _h }.
\end{equation*} 
Next, using
$ \llbracket \widehat{ \sigma } _h ^{\mathrm{tan}} \rrbracket = 0 $
and $ \llbracket \widehat{ u } _h ^{\mathrm{nor}} \rrbracket = 0 $
again, we have
\begin{align*} 
  \langle \widehat{ \sigma } _h ^{\mathrm{tan}} , v _h ^{\mathrm{nor}} \rangle _{ \partial \mathcal{T} _h }
  &= \bigl\langle \av{ \widehat{ \sigma } _h ^{\mathrm{tan}} } , \llbracket v _h ^{\mathrm{nor}} \rrbracket \bigr\rangle _{ \partial \mathcal{T} _h } \\
  &= \Bigl\langle \bigav{ ( \delta u _h ) ^{\mathrm{tan}} }, \llbracket v _h ^{\mathrm{nor}} \rrbracket \Bigr\rangle _{ \partial \mathcal{T} _h } + \bigl\langle \beta \llbracket u _h ^{\mathrm{nor}} \rrbracket , \llbracket v _h ^{\mathrm{nor}} \rrbracket \bigr\rangle _{ \partial \mathcal{T} _h } ,
\end{align*} 
and similarly,
\begin{equation*}
  \langle \widehat{ \rho } _h ^{\mathrm{nor}} , v _h ^{\mathrm{tan}} \rangle _{ \partial \mathcal{T} _h } = \Bigl\langle \bigav{ ( \mathrm{d} u _h ) ^{\mathrm{nor}} }, \llbracket v _h ^{\mathrm{tan}} \rrbracket \Bigr\rangle _{ \partial \mathcal{T} _h } + \bigl\langle \alpha \llbracket u _h ^{\mathrm{tan}} \rrbracket , \llbracket v _h ^{\mathrm{tan}} \rrbracket \bigr\rangle _{ \partial \mathcal{T} _h } .
\end{equation*}
Altogether, the left-hand side of \eqref{eq:ip-h_v} may therefore be
replaced by the primal bilinear form
$ a _h ( u _h , v _h ) \coloneqq a _h ^\delta ( u _h , v _h ) + a _h
^{ \mathrm{d} } ( u _h , v _h ) $, where
\begin{align*}
  a _h ^\delta ( u _h , v _h ) &\coloneqq ( \delta u _h , \delta v _h ) _{ \mathcal{T} _h } + \bigl\langle \beta \llbracket u _h ^{\mathrm{nor}} \rrbracket , \llbracket v _h ^{\mathrm{nor}} \rrbracket \bigr\rangle _{ \partial \mathcal{T} _h } - \Bigl\langle \alpha \bigl\llbracket ( \delta u _h ) ^{\mathrm{tan}} \bigr\rrbracket, \bigl\llbracket (\delta v _h ) ^{\mathrm{tan}} \bigr\rrbracket \Bigr\rangle _{ \partial \mathcal{T} _h } \\
  &\qquad + \Bigl\langle \bigav{ ( \delta u _h ) ^{\mathrm{tan}} }, \llbracket v _h ^{\mathrm{nor}} \rrbracket \Bigr\rangle _{ \partial \mathcal{T} _h } + \Bigl\langle \llbracket u _h ^{\mathrm{nor}} \rrbracket , \bigav{ (\delta v _h ) ^{\mathrm{tan}} } \Bigr\rangle _{ \partial \mathcal{T} _h } \\
  a _h ^{\mathrm{d}} ( u _h , v _h ) &\coloneqq ( \mathrm{d} u _h , \mathrm{d} v _h ) _{ \mathcal{T} _h } - \bigl\langle \alpha \llbracket u _h ^{\mathrm{tan}} \rrbracket , \llbracket v _h ^{\mathrm{tan}} \rrbracket \bigr\rangle _{ \partial \mathcal{T} _h } + \Bigl\langle \beta \bigl\llbracket ( \mathrm{d} u _h ) ^{\mathrm{nor}} \bigr\rrbracket , \bigl\llbracket ( \mathrm{d} v _h ) ^{\mathrm{nor}} \bigr\rrbracket \Bigr\rangle _{ \partial \mathcal{T} _h }\\
  &\qquad - \Bigl\langle \bigav{ ( \mathrm{d} u _h )
    ^{\mathrm{nor}} }, \llbracket v _h ^{\mathrm{tan}} \rrbracket
  \Bigr\rangle _{ \partial \mathcal{T} _h } - \Bigl\langle \llbracket u _h ^{\mathrm{tan}} \rrbracket , \bigav{
    ( \mathrm{d} v _h ) ^{\mathrm{nor}} } \Bigr\rangle _{ \partial
    \mathcal{T} _h } .
\end{align*}
For stabilization, examining the two leading terms of $ a _h ^\delta $
shows that we will want $ \alpha ^{ k -1 } , \beta ^{ k -1 } > 0 $,
and considering $ a _h ^{ \mathrm{d} } $ similarly shows that we will
want $ \alpha ^k , \beta ^k < 0 $.

\begin{remark}
  The bilinear form $ a _h $ may also be written more briefly as
  \begin{alignat*}{2}
    a _h ( u _h , v _h ) = ( \mathrm{D} u _h , \mathrm{D} v _h ) _{ \mathcal{T} _h } &+ \bigl\langle \gamma \llbracket u _h \rrbracket , \llbracket v _h \rrbracket \bigr\rangle _{ \partial \mathcal{T} _h } &&- \bigl\langle \gamma ^{-1} \llbracket \mathrm{D} u _h \rrbracket , \llbracket \mathrm{D} v _h \rrbracket \bigr\rangle _{ \partial \mathcal{T} _h } \\
    &+ \bigl[ \av{\mathrm{D} u _h }, \llbracket v _h \rrbracket \bigr] _{ \partial \mathcal{T} _h } &&- \bigl[ \llbracket u _h \rrbracket , \av{\mathrm{D} v _h } \bigr] _{ \partial \mathcal{T} _h } ,
  \end{alignat*}
  where $ \gamma \coloneqq \beta ^{ k -1 } \oplus (- \alpha ^k ) $ and
  $ \gamma ^{-1} = \alpha ^{ k -1 } \oplus ( - \beta ^k ) $. If
  $ \alpha ^{ k -1 } , \beta ^{ k -1 } > 0 $ and
  $ \alpha ^k , \beta ^k < 0 $, as in the comments above, then all
  coefficients of $\gamma, \gamma ^{-1} $ will be positive.
\end{remark}

\subsubsection{IP-H methods for scalar and vector problems}

For scalar problems with $ k = 0 $ or $ k = n $, \cref{rmk:ip-h_ldg-h}
ensures that the spaces and penalties discussed in
\cref{sec:ldg-h_scalar} for LDG-H also yield strongly multisymplectic
IP-H methods. For the piecewise-constant penalty, we recover the
equal-order IP-H and unequal-order ``IP-H-like'' methods in
\citet{CoGoLa2009}, whose strong multisymplecticity was established in
\citet{McSt2020}. For the reduced-stabilization penalty, we recover
the hybridized interior-penalty methods analyzed by
\citet{Lehrenfeld2010,LeSc2016} and \citet{Oikawa2015,Oikawa2016}. For
either $ k = 0 $ or $ k = n $ with piecewise-constant penalty, the
primal bilinear form $ a _h $ may be written as
\begin{align*}
  a _h (u _h , v _h ) = ( \operatorname{grad} u _h , \operatorname{grad} v _h ) _{ \mathcal{T} _h } &+ \bigl\langle \gamma \llbracket u _h \widehat{ n } \rrbracket , \llbracket v _h \widehat{ n } \rrbracket \bigr\rangle _{ \partial \mathcal{T} _h } - \bigl\langle \gamma ^{-1} \llbracket \operatorname{grad} u _h \cdot \widehat{ n } \rrbracket , \llbracket \operatorname{grad} v _h \cdot \widehat{ n } \rrbracket \bigr\rangle _{ \partial \mathcal{T} _h } \\
  &- \bigl\langle \av{ \operatorname{grad} u _h }, \llbracket v _h \widehat{ n } \rrbracket \bigr\rangle _{ \partial \mathcal{T} _h } - \bigl\langle \llbracket u _h \widehat{ n } \rrbracket , \av{ \operatorname{grad} v _h } \bigr\rangle _{ \partial \mathcal{T} _h } ,
\end{align*}
where $ \widehat{ n } $ is the outer unit normal vector field on
$ \partial \mathcal{T} _h $. Here, the jumps denoted
$ \llbracket {} \cdot \widehat{ n } \rrbracket $ are taken in an
orientation-independent sense, while $ \av{ \cdot } $ is an ordinary
vector average, i.e.,
\begin{align*}
  \llbracket v _h \widehat{ n } \rrbracket _{e ^\pm } &= \frac{1}{2} \bigl(  v _h \widehat{ n } \rvert _{ e^+ } + v _h \widehat{ n } \rvert _{ e^- } \bigr) ,\\
  \llbracket \operatorname{grad} v _h \cdot \widehat{ n } \rrbracket _{e ^\pm } &= \frac{1}{2} \bigl( \operatorname{grad} v _h \cdot \widehat{ n } \rvert _{ e^+ } + \operatorname{grad} v _h \cdot \widehat{ n } \rvert _{ e^- } \bigr) ,\\
  \av{ \operatorname{grad} v _h } _{ e^\pm } &= \frac{1}{2} \bigl( \operatorname{grad} v _h \rvert _{ e ^+} + \operatorname{grad} v _h \rvert _{ e ^- } \bigr) . 
\end{align*}
Note that $ a _h $ differs subtly from the primal form for the classic
interior-penalty method of \citet{DoDu1976}, which omits the
$ \bigl\langle \gamma ^{-1} \llbracket \operatorname{grad} u _h \cdot
\widehat{ n } \rrbracket , \llbracket \operatorname{grad} v _h \cdot
\widehat{ n } \rrbracket \bigr\rangle _{ \partial \mathcal{T} _h } $
term, cf.~\citep[Corollary 3.4]{CoGoLa2009}.

For the vector Laplacian in dimension $ n = 3 $ with $ k = 1 $ (or
$ k = 2 $), $ a _h ^\delta $ and $ a _h ^{ \mathrm{d} } $ give primal
bilinear forms for the minus-grad-div and curl-curl operators (or vice
versa). In particular, we can write
\begin{alignat*}{2}
  a _h ^{\operatorname{curl}} ( u _h , v _h ) = ( \operatorname{curl} u _h , \operatorname{curl} v _h ) _{ \mathcal{T} _h } &+ \bigl\langle \gamma \llbracket u _h \times \widehat{ n } \rrbracket , \llbracket v _h \times \widehat{ n } \rrbracket \bigr\rangle _{ \partial \mathcal{T} _h } &&- \bigl\langle \gamma ^{-1} \llbracket \operatorname{curl} u _h \times \widehat{ n } \rrbracket , \llbracket \operatorname{curl} v _h \times \widehat{ n } \rrbracket \bigr\rangle _{ \partial \mathcal{T} _h } \\
  &+ \bigl\langle \av{ \operatorname{curl} u _h }, \llbracket v _h \times \widehat{ n } \rrbracket \bigr\rangle _{ \partial \mathcal{T} _h } &&+ \bigl\langle \llbracket u _h \times \widehat{ n } \rrbracket , \av{ \operatorname{curl} v _h } \bigr\rangle _{ \partial \mathcal{T} _h } .
\end{alignat*}
Similarly to the scalar case above, the jumps and averages are taken
as
\begin{align*}
  \llbracket v _h \times \widehat{ n } \rrbracket _{e ^\pm } &= \frac{1}{2} \bigl(  v _h \times \widehat{ n } \rvert _{ e^+ } + v _h \times \widehat{ n } \rvert _{ e^- } \bigr) , \\
  \llbracket \operatorname{curl} v _h \times \widehat{ n } \rrbracket _{e ^\pm } &= \frac{1}{2} \bigl( \operatorname{curl} v _h \times \widehat{ n } \rvert _{ e^+ } + \operatorname{curl} v _h \times \widehat{ n } \rvert _{ e^- } \bigr) ,\\
  \av{ \operatorname{curl} v _h } _{ e ^\pm } &= \frac{1}{2} \bigl( \operatorname{curl} v _h \rvert _{ e ^+} + \operatorname{curl} v _h \rvert _{ e ^- } \bigr) . 
\end{align*}
This is closely related to the interior-penalty form proposed by
\citet*{PeScMo2002} for the curl-curl operator, which differs only in
its omission of the
$ \bigl\langle \gamma ^{-1} \llbracket \operatorname{curl} u _h \times
\widehat{ n } \rrbracket , \llbracket \operatorname{curl} v _h \times
\widehat{ n } \rrbracket \bigr\rangle _{ \partial \mathcal{T} _h } $
term---just as the primal form for IP-H differs from that for the
classic IP method in the scalar case. Similarly, for the
minus-grad-div operator, we can write
\begin{alignat*}{2}
  a _h ^{ \operatorname{div} } = ( \operatorname{div} u _h , \operatorname{div} v _h ) _{ \mathcal{T} _h } &+ \bigl\langle \gamma \llbracket u _h \cdot \widehat{ n } \rrbracket , \llbracket v _h \cdot \widehat{ n } \rrbracket \bigr\rangle _{ \partial \mathcal{T} _h } &&- \bigl\langle \gamma ^{-1} \llbracket \operatorname{div} u _h \ \widehat{ n } \rrbracket , \llbracket \operatorname{div} v _h \ \widehat{ n } \rrbracket \bigr\rangle _{ \partial \mathcal{T} _h } \\
  &- \bigl\langle \av{ \operatorname{div} u _h }, \llbracket v _h \cdot n \rrbracket \bigr\rangle _{ \partial \mathcal{T} _h } &&- \bigl\langle \llbracket u _h \cdot n \rrbracket , \av{ \operatorname{div} v _h } \bigr\rangle _{ \partial \mathcal{T} _h } ,
\end{alignat*}
where
\begin{align*}
  \llbracket v _h \cdot \widehat{ n } \rrbracket _{e ^\pm } &= \frac{1}{2} \bigl(  v _h \cdot \widehat{ n } \rvert _{ e^+ } + v _h \cdot \widehat{ n } \rvert _{ e^- } \bigr) , \\
  \llbracket \operatorname{div} v _h \ \widehat{ n } \rrbracket _{e ^\pm } &= \frac{1}{2} \bigl( \operatorname{div} v _h \ \widehat{ n } \rvert _{ e^+ } + \operatorname{div} v _h \ \widehat{ n } \rvert _{ e^- } \bigr) ,\\
  \av{ \operatorname{div} v _h } _{ e ^\pm } &= \frac{1}{2} \bigl( \operatorname{div} v _h \rvert _{ e ^+} + \operatorname{div} v _h \rvert _{ e ^- } \bigr) .
\end{align*}
(Recall that $\gamma$ is a graded operator, so different penalty
coefficients may be taken for $ a _h ^{ \operatorname{curl} } $ and
$ a _h ^{ \operatorname{div} } $.)

\subsection{HDG methods with reduced spaces}
\label{sec:reduced}

We next consider HDG methods with
$ \widehat{ W } _h ^{\mathrm{nor}} = \widehat{ W } _h ^{\mathrm{tan}}
$ rather than
$ \widehat{ W } _h ^{\mathrm{nor}} = L ^2 \Lambda ( \partial
\mathcal{T} _h ) $. As in the preceding two sections, we take broken
tangential traces
\begin{equation*}
  \widehat{ W } _h ^{\mathrm{tan}} ( \partial K ) = \prod _{ e \subset \partial K } \widehat{ W } _h ^{\mathrm{tan}} (e) ,
\end{equation*}
where
$ \widehat{ W } _h ^{\mathrm{tan}} ( e ^+ ) = \widehat{ W } _h
^{\mathrm{tan}} ( e ^- ) $ at interior facets $e$. In particular,
taking
$ \widehat{ W } _h ^{\mathrm{nor}} = \widehat{ W } _h ^{\mathrm{tan}}
$ implies that the resulting method is always strongly conservative,
since
$ \llbracket \widehat{ W } _h ^{\mathrm{nor}} \rrbracket = \av{
  \widehat{ W } _h ^{\mathrm{tan}} } = \ringhat{V} _h ^{\mathrm{tan}}
$.

Let us take the LDG-H flux function
\begin{equation*}
  \Phi ( z _h , \widehat{ z } _h ) = ( \widehat{ z } _h
  ^{\mathrm{nor}} - z _h ^{\mathrm{nor}} ) + \alpha ( \widehat{ z } _h
  ^{\mathrm{tan}} - z _h ^{\mathrm{tan}} ).
\end{equation*}
(For the semilinear Hodge--Laplace problem, a similar approach can be
taken with the IP-H flux.) Since
$ \widehat{ W } _h ^{\mathrm{nor}} = \widehat{ W } _h ^{\mathrm{tan}}
$, it follows from \eqref{eq:weakform_wnor} that
\begin{equation}
  \label{eq:ldg-h_reduced_znor}
  \widehat{ z } _h ^{\mathrm{nor}} = \mathbb{P}  \bigl[ z _h ^{\mathrm{nor}} - \alpha ( \widehat{ z } _h ^{\mathrm{tan}} - z _h ^{\mathrm{tan}} ) \bigr] ,
\end{equation}
where $ \mathbb{P} $ again denotes $ L ^2 $-orthogonal projection onto
$ \widehat{ W } _h ^{\mathrm{tan}} $. If
$ W _h ^{\mathrm{nor}} \subset \widehat{ W } _h ^{\mathrm{tan}} $,
then the projection has no effect on $ z _h ^{\mathrm{nor}} $, and we
recover LDG-H methods with reduced-stabilization penalties already
considered in \cref{sec:ldg-h}.  (Note that, in this case, the
projection is implemented through the choice of
$ \widehat{ W } _h ^{\mathrm{nor}} $ and need not be incorporated into
the operator $\alpha$.) These cases are immediately seen to satisfy
the strong multisymplecticity condition
\eqref{eq:ldg-h_strong_reduced}, since
$ W _h ^{\mathrm{nor}} \subset \widehat{ W } _h ^{\mathrm{tan}} $
implies
$ \llbracket W _h ^{\mathrm{nor}} \rrbracket \subset \av{ \widehat{ W
  } _h ^{\mathrm{tan}} } = \ringhat{V} _h ^{\mathrm{tan}} $.

To restrict our attention to methods not already considered in
\cref{sec:ldg-h}, suppose now that
$ W _h ^{\mathrm{nor}} \not\subset \widehat{ W } _h ^{\mathrm{tan}} $.
Assume that
$ W _h ^{\mathrm{tan}} \subset \widehat{ W } _h ^{\mathrm{tan}} $, so
that the hypothesis of \cref{lem:ldg-h_flux} for multisymplecticity is
satisfied. In these cases, the projection has no effect when
$ \widehat{ z } _h ^{\mathrm{nor}} $ is substituted into
\eqref{eq:weakform_w} and \eqref{eq:weakform_wtan}, and we obtain
precisely the same symmetric system \eqref{eq:ldg-h_weakform} for
$ ( z _h , \widehat{ z } _h ^{\mathrm{tan}} ) $ as we did for the
standard LDG-H method. Hence, the \emph{only} effect of the projection
in these new cases is on $ \widehat{ z } _h ^{\mathrm{nor}} $ itself,
not on the remaining variables.

\begin{theorem}
  Consider the class of methods described above with
  $ \widehat{ W } _h ^{\mathrm{nor}} = \widehat{ W } _h
  ^{\mathrm{tan}} $ and LDG-H flux. If
  $ W _h ^{\mathrm{nor}} \subset \widehat{ W } _h ^{\mathrm{tan}} $ or
  $ W _h ^{\mathrm{tan}} \subset \widehat{ W } _h ^{\mathrm{tan}} $,
  then the method is strongly multisymplectic.
\end{theorem}

\begin{proof}
  As discussed above, the condition
  $ \widehat{ W } _h ^{\mathrm{nor}} = \widehat{ W } _h
  ^{\mathrm{tan}} $ immediately gives strong conservativity. If
  $ W _h ^{\mathrm{nor}} \subset \widehat{ W } _h ^{\mathrm{tan}} $,
  then the resulting method coincides with a standard LDG-H method,
  which is multisymplectic by \cref{thm:ldg-h_ms}. If
  $ W _h ^{\mathrm{tan}} \subset \widehat{ W } _h ^{\mathrm{tan}} $,
  then the resulting method is multisymplectic by
  \cref{lem:ldg-h_flux}. In either case, we therefore have strong
  multisymplecticity by \cref{thm:strong.MS}.
\end{proof}

For scalar problems, in addition to the LDG-H methods previously
discussed in \cref{sec:ldg-h_scalar}, we also have the following
methods satisfying
$ W _h ^{\mathrm{nor}} \not\subset \widehat{ W } _h ^{\mathrm{tan}} $
and $ W _h ^{\mathrm{tan}} \subset \widehat{ W } _h ^{\mathrm{tan}}
$. For $ k = 0 $, take
\begin{gather*}
  W _h ^0 (K) = \mathcal{P} _{ r -1 } \Lambda ^0 (K) , \qquad W _h ^1 (K) = \mathcal{P} _r \Lambda ^1 (K) ,\\
  \widehat{ W } _h ^{0, \mathrm{nor}} (e) = \widehat{ W } _h ^{0, \mathrm{tan}} (e) = \mathcal{P} _{ r -1 } \Lambda ^0 (e) .
\end{gather*}
The resulting method can be seen as the Hodge dual of the $ k = n $
LDG-H method with \eqref{eq:k=n_r_r-1} and reduced-stabilization
penalty. For $ k = n $, taking
\begin{gather*}
  W _h ^{n-1} (K) = \mathcal{P} _{ r -1 } \Lambda ^{n-1} (K) , \qquad W _h ^n (K) = \mathcal{P} _r \Lambda ^n (K) ,\\
  \widehat{ W } _h ^{n-1, \mathrm{nor}} (e) = \widehat{ W } _h ^{n-1, \mathrm{tan}} (e) = \mathcal{P} _{ r -1 } \Lambda ^{n-1} (e) ,
\end{gather*}
yields the Hodge dual of the $ k = 0 $ LDG-H method with
\eqref{eq:k=0_r_r-1} and reduced-stabilization penalty. This
corresponds to the alternative implementation of reduced stabilization
discussed in \citet[\S1.2.2.3]{Lehrenfeld2010}.

\begin{remark}
  For linear problems, \citeauthor{Lehrenfeld2010} lists the
  singularity of the local Neumann solvers as a ``computational
  disadvantage[]'' of the $ k = n $ implementation. However, this
  issue can be remedied, as in \citet[Section 5]{Cockburn2016} and
  \citet{AwFaGuSt2023}, by introducing an additional
  piecewise-constant global variable $ \overline{ u } _h $,
  corresponding to the element-wise average of $ u _h $, so that the
  local solvers depend on
  $ ( \widehat{ \sigma } _h ^{\mathrm{tan}} , \overline{ u } _h ) $
  rather than $ \widehat{ \sigma } _h ^{\mathrm{tan}} $ alone.
\end{remark}

In the case $ \alpha = 0 $, the $ k = n $ method above coincides with
the NC-H method, whose multisymplecticity for de~Donder--Weyl systems
was shown in \citet{McSt2020}. In particular, for the semilinear
Poisson equation, \eqref{eq:ldg-h_weakform} becomes 
\begin{alignat*}{2}
  ( \delta u _h , \tau _h ) _{ \mathcal{T} _h } &= ( \sigma _h , \tau _h ) _{ \mathcal{T} _h } , \quad &\forall \tau _h &\in W _h ^{ n -1 } ,\\
  ( \sigma _h , \delta v _h ) _{ \mathcal{T} _h } + \langle \widehat{ \sigma } _h ^{\mathrm{tan}} , v _h ^{\mathrm{nor}} \rangle _{ \partial \mathcal{T} _h } &= \biggl( \frac{ \partial F }{ \partial u _h }  , v _h \biggr) _{ \mathcal{T} _h }, \quad &\forall v _h &\in W _h ^n ,\\
  \langle u _h ^{\mathrm{nor}} , \widehat{ \tau } _h ^{\mathrm{tan}} \rangle _{ \partial \mathcal{T} _h } &= 0, \quad &\forall \widehat{ \tau } _h ^{\mathrm{tan}} &\in \ringhat{V} _h ^{\mathrm{tan}} .
\end{alignat*}
The first equation allows us to eliminate $ \sigma _h = \delta u _h $,
substituting into the second equation to obtain
\begin{alignat*}{2}
  ( \delta u _h , \delta v _h ) _{ \mathcal{T} _h } + \langle \widehat{ \sigma } _h ^{\mathrm{tan}} , v _h ^{\mathrm{nor}} \rangle _{ \partial \mathcal{T} _h } &= \biggl( \frac{ \partial F }{ \partial u _h }  , v _h \biggr) _{ \mathcal{T} _h }, \quad &\forall v _h &\in W _h ^n ,\\
  \langle u _h ^{\mathrm{nor}} , \widehat{ \tau } _h ^{\mathrm{tan}} \rangle _{ \partial \mathcal{T} _h } &= 0, \quad &\forall \widehat{ \tau } _h ^{\mathrm{tan}} &\in \ringhat{V} _h ^{\mathrm{tan}} .
\end{alignat*}
In the linear case, this is the primal-hybrid nonconforming method of
\citet{RaTh1977_hybrid}.

\begin{remark}
  \label{rmk:nc-h_dirichlet}
  The Hodge dual of this nonconforming method, using local Dirichlet
  rather than Neumann solvers, recovers the NC-H method described in
  \citet{CoGoLa2009}. Specifically, for $ k = 0 $, we take the spaces
  \begin{gather*}
    W _h ^0 (K) = \mathcal{P} _r \Lambda ^0 (K) , \qquad W _h ^1 (K) = \mathcal{P} _{r-1} \Lambda ^1 (K) ,\\
    \widehat{ W } _h ^{0, \mathrm{nor}} (e) = \widehat{ W } _h ^{0, \mathrm{tan}} (e) = \mathcal{P} _{ r -1 } \Lambda ^0 (e) ,
  \end{gather*}
  along with the local flux function
  \begin{equation*}
    \Phi ( z _h , \widehat{ z } _h ) = \widehat{ z } _h
    ^{\mathrm{tan}} - z _h ^{\mathrm{tan}} ,
  \end{equation*}
  which is multisymplectic by \cref{lem:afw-h_flux}. For the
  semilinear Poisson equation, we obtain
  \begin{alignat*}{2}
    ( \mathrm{d} u _h , \mathrm{d} v _h ) _{ \mathcal{T} _h } - \langle \widehat{ \rho } _h ^{\mathrm{nor}} , v _h ^{\mathrm{tan}} \rangle _{ \partial \mathcal{T} _h } &= \biggl( \frac{ \partial F }{ \partial u _h } , v _h \biggr) _{ \mathcal{T} _h } , \quad &\forall v _h &\in W _h ^0 , \\
    \langle \widehat{ u } _h ^{\mathrm{tan}} - u _h ^{\mathrm{tan}} , \widehat{ \eta } _h ^{\mathrm{nor}} \rangle _{ \partial \mathcal{T} _h } &= 0, \quad &\forall \widehat{ \eta } _h ^{\mathrm{nor}} &\in \widehat{ W } _h ^{0, \mathrm{nor}} , \\
    \langle \widehat{ \rho } _h ^{\mathrm{nor}} , \widehat{ v } _h ^{\mathrm{tan}} \rangle _{ \partial \mathcal{T} _h } &= 0 , \quad &\forall \widehat{ v } _h ^{\mathrm{tan}} &\in \ringhat{V} _h ^{0, \mathrm{tan}} ,
  \end{alignat*}
  after having eliminated $ \rho _h = \mathrm{d} u _h $.
\end{remark}

\subsection{Primal-hybrid nonconforming methods}

We now discuss a different generalization of the NC-H method for the
semilinear $k$-form Hodge--Laplace problem. In the linear case with
$ k = 1 $ and $ n = 2 $, this includes the nonconforming primal-hybrid
method of \citet*{BaCaSt2024} for the vector Poisson equation, whose
lowest-order case is a hybridization of the
$ \mathcal{P} _1 $-nonconforming method of \citet{BrCuLiSu2008}. While
this family of methods does \emph{not} fit into the framework of
\cref{sec:multisymplectic_hybrid}, except in some special cases,
multisymplecticity can nevertheless be proved directly.

Given $k$, suppose we have broken spaces
\begin{equation*}
  W _h ^k \coloneqq \prod _{ K \in \mathcal{T} _h } W _h ^k (K) , \qquad \widehat{ W } _h ^{k-1, \mathrm{nor}} \coloneqq \prod _{ e \in \partial \mathcal{T} _h } \widehat{ W } _h ^{k-1, \mathrm{nor}} (e) , \qquad \widehat{ W } _h ^{k, \mathrm{tan}} \coloneqq \prod _{ e \in \partial \mathcal{T} _h } \widehat{ W } _h ^{k, \mathrm{tan}} (e) ,
\end{equation*}
where
$ \widehat{ W } _h ^{k-1, \mathrm{nor}} (e^+) = \widehat{ W } _h
^{k-1, \mathrm{nor}} (e^-) $ and
$ \widehat{ W } _h ^{k, \mathrm{tan}} (e^+) = \widehat{ W } _h ^{k,
  \mathrm{tan}} (e^-) $ at interior facets. Furthermore, suppose that
we have broken trace subspaces
\begin{equation*}
  \overline{ W } _h ^{k-1, \mathrm{tan}} \coloneqq \prod _{ e \in \partial \mathcal{T} _h } \overline{ W } _h ^{k-1, \mathrm{tan}} (e) , \qquad \overline{ W } _h ^{k, \mathrm{nor}} \coloneqq \prod _{ e \in \partial \mathcal{T} _h } \overline{ W } _h ^{k, \mathrm{nor}} (e) ,
\end{equation*}
where
$ \overline{ W } _h ^{k-1, \mathrm{tan}} (e) \subset \widehat{ W } _h
^{k-1, \mathrm{nor}} (e) $ and
$ \overline{ W } _h ^{k, \mathrm{nor}} (e) \subset \widehat{ W } _h
^{k, \mathrm{tan}} (e) $ for each $ e \in \partial \mathcal{T} _h
$. Finally, let
\begin{equation*}
  \ringhat{V} _h ^{k-1, \mathrm{nor}} \subset \widehat{ V } _h ^{k-1, \mathrm{nor}} \subset \widehat{ W } _h ^{k-1, \mathrm{nor}} , \qquad   \ringhat{V} _h ^{k, \mathrm{tan}} \subset \widehat{ V } _h ^{k, \mathrm{tan}} \subset \widehat{ W } _h ^{k, \mathrm{tan}} ,
\end{equation*}
be the single-valued trace subspaces, where the ring denotes traces
vanishing on $ \partial \Omega $.

Following \citep{BaCaSt2024}, we seek $ u _h \in W _h ^k $,
$ \overline{ \sigma } _h ^{\mathrm{tan}} \in \overline{ W } _h ^{k-1,
  \mathrm{tan}} $,
$ \overline{ \rho } _h ^{\mathrm{nor}} \in \overline{ W } _h ^{k,
  \mathrm{nor}} $,
$ \widehat{ u } _h ^{\mathrm{nor}} \in \widehat{ V } _h
^{k-1,\mathrm{nor}} $, and
$ \widehat{ u } _h ^{\mathrm{tan}} \in \widehat{ V } _h ^{k,
  \mathrm{tan}} $ satisfying
\begin{subequations}
  \label{eq:bcs}
  \begin{alignat}{2}
    ( \delta u _h , \delta v _h ) _{ \mathcal{T} _h } + ( \mathrm{d} u _h , \mathrm{d} v _h ) _{ \mathcal{T} _h } + \langle \widehat{ \sigma } _h ^{\mathrm{tan}} , v _h ^{\mathrm{nor}} \rangle _{ \partial \mathcal{T} _h } - \langle \widehat{ \rho } _h ^{\mathrm{nor}} , v _h ^{\mathrm{tan}} \rangle _{ \partial \mathcal{T} _h } &=  \biggl( \frac{ \partial F }{ \partial u _h } , v _h \biggr) _{ \mathcal{T} _h } , \quad &\forall v _h &\in W _h ^k , \label{eq:bcs_v} \\
    \langle \widehat{ u } _h ^{\mathrm{nor}} - u _h ^{\mathrm{nor}} , \overline{ \tau } _h ^{\mathrm{tan}} \rangle _{ \partial \mathcal{T} _h } &= 0 , \quad &\forall \overline{ \tau } _h ^{\mathrm{tan}} &\in \overline{ W } _h ^{ k -1 , \mathrm{tan}}, \label{eq:bcs_taubar} \\
    \langle \widehat{ u } _h ^{\mathrm{tan}} - u _h ^{\mathrm{tan}} , \overline{ \eta } _h ^{\mathrm{nor}} \rangle _{ \partial \mathcal{T} _h } &= 0 , \quad &\forall \overline{ \eta } _h ^{\mathrm{nor}} &\in \overline{ W } _h ^{ k , \mathrm{nor}}, \label{eq:bcs_etabar} \\
    \langle \widehat{ \sigma } _h ^{\mathrm{tan}} , \widehat{ v } _h ^{\mathrm{nor}} \rangle _{ \partial \mathcal{T} _h } &= 0 , \quad &\forall \widehat{ v } _h ^{\mathrm{nor}} &\in \ringhat{V} _h ^{k-1, \mathrm{nor}} , \label{eq:bcs_vnor}\\
    \langle \widehat{ \rho } _h ^{\mathrm{nor}} , \widehat{ v } _h ^{\mathrm{tan}} \rangle _{ \partial \mathcal{T} _h } &= 0 , \quad &\forall \widehat{ v } _h ^{\mathrm{tan}} &\in \ringhat{V} _h ^{k, \mathrm{tan}} \label{eq:bcs_vtan},
  \end{alignat}
\end{subequations}
where $ \widehat{ \sigma } _h ^{\mathrm{tan}} $ and
$ \widehat{ \rho } _h ^{\mathrm{nor}} $ are numerical fluxes defined
by
\begin{equation*}
  \widehat{ \sigma } _h ^{\mathrm{tan}} \coloneqq \overline{ \sigma } _h ^{\mathrm{tan}} - \beta ( \widehat{ u } _h ^{\mathrm{nor}} - u _h ^{\mathrm{nor}} ) , \qquad \widehat{ \rho } _h ^{\mathrm{nor}} \coloneqq \overline{ \rho } _h ^{\mathrm{nor}} - \alpha ( \widehat{ u } _h ^{\mathrm{tan}} - u _h ^{\mathrm{tan}} ) .
\end{equation*}

\begin{remark}
  \label{rmk:bcs_tan}
  In contrast with the preceding sections, the global variables here
  are $ \widehat{ u } _h ^{\mathrm{nor}} $ and
  $ \widehat{ u } _h ^{\mathrm{tan}} $. However, if
  $ \bigl\llbracket \beta ( W _h ^{k-1,\mathrm{nor}} + \widehat{ W }
  _h ^{k-1, \mathrm{nor}} ) \bigr\rrbracket \subset \ringhat{ V } _h
  ^{k-1, \mathrm{nor}} $, so that
  $ \widehat{ \sigma } _h ^{\mathrm{tan}} $ is single-valued, then we
  can obtain an equivalent hybridization on
  $ \widehat{ \sigma } _h ^{\mathrm{tan}} \in \widehat{ V } _h ^{k-1,
    \mathrm{tan}} \coloneqq \widehat{ V } _h ^{k-1, \mathrm{nor}} $
  and
  $ \widehat{ u } _h ^{\mathrm{tan}} \in \widehat{ V } _h ^{k,
    \mathrm{tan}} $, replacing \eqref{eq:bcs_vnor} by
  \begin{equation}
    \langle \widehat{ u } _h ^{\mathrm{nor}} , \widehat{ \tau } _h ^{\mathrm{tan}} \rangle _{ \partial \mathcal{T} _h } = 0 , \quad \forall \widehat{ \tau } _h ^{\mathrm{tan}} \in \ringhat{V} _h ^{k-1, \mathrm{tan}} . \tag{\ref*{eq:bcs_vnor}$^\prime$}
  \end{equation} 
\end{remark}
  
\begin{example}[the NC-H method]
  Consider the scalar case $ k = 0 $ with vanishing penalty
  $ \alpha = 0 $, and take the spaces
  \begin{equation*}
    W _h ^0 (K) = \mathcal{P} _r \Lambda ^0 (K) , \qquad \overline{ W } _h ^{0, \mathrm{tan}} (e) = \widehat{ W } _h ^{0, \mathrm{tan}} (e) = \mathcal{P} _{ r -1 } \Lambda ^0 (e) .
  \end{equation*}
  Then \eqref{eq:bcs} coincides with the NC-H method with local
  Dirichlet solvers in \cref{rmk:nc-h_dirichlet}. Since \eqref{eq:bcs}
  is symmetric with respect to the Hodge star, we get essentially the
  same method for $ k = n $, $ \beta = 0 $.
\end{example}

\begin{example}[the primal-hybrid nonconforming method of \citep{BaCaSt2024}]
  \label{ex:bcs}
  When $ k = 1 $ and $ n = 2 $, take
  \begin{equation*}
    W _h ^k (K) = \mathcal{P} _{ 2 r -1 } \Lambda ^1 (K) , \qquad \widehat{ W } _h ^{0, \mathrm{nor}} (e) = \mathcal{P} _{ 2 r -1 } \Lambda ^0 (e) , \qquad \widehat{ W } _h ^{1, \mathrm{tan}} (e) = \mathcal{P} _{ 2 r -1 } \Lambda ^1 (e) ,
  \end{equation*}
  and trace subspaces
  \begin{equation*}
    \overline{ W } _h ^{0, \mathrm{tan}} (e) = \mathcal{P} _{ r -1 } \Lambda ^0 (e) , \qquad \overline{ W } _h ^{1, \mathrm{nor}} (e) = \mathcal{P} _{ r -1 } \Lambda ^1 (e) .
  \end{equation*}
  The linear case recovers the method in \citet*{BaCaSt2024}, whose
  lowest-order version hybridizes the $ \mathcal{P} _1 $-nonconforming
  method of \citet{BrCuLiSu2008}. We note that degree $ 2 r -1 $ is
  required for the higher-degree spaces, rather than merely degree
  $r$, because this admits a projection (due to \citet{BrSu2009} and
  \citet{Mirebeau2012}) that commutes with both $ \delta $ and
  $ \mathrm{d} $. A special choice of piecewise-constant $\alpha$ and
  $\beta$, depending on the distance to corners of $\partial \Omega$
  at which singularities may form, is required to ensure consistency;
  see \citep{BaCaSt2024,BrCuLiSu2008} and references therein.
\end{example}

\begin{theorem}
  The primal-hybrid nonconforming method \eqref{eq:bcs} is
  multisymplectic. Furthermore, if
  \begin{equation*}
    \bigl\llbracket \beta ( W _h ^{k-1,\mathrm{nor}} + \widehat{ W } _h ^{k-1, \mathrm{nor}} ) \bigr\rrbracket \subset \ringhat{ V } _h ^{k-1, \mathrm{nor}}, \qquad     \bigav{ \alpha ( W _h ^{k,\mathrm{tan}} + \widehat{ W } _h ^{k, \mathrm{tan}} ) } \subset \ringhat{ V } _h ^{k, \mathrm{tan}},
  \end{equation*}
  then it is strongly multisymplectic. For piecewise-constant
  penalties, these conditions reduce to
  \begin{equation*}
    \llbracket W _h ^{k-1,\mathrm{nor}} \rrbracket \subset \ringhat{ V } _h ^{k-1, \mathrm{nor}}, \qquad  \av{ W _h ^{k,\mathrm{tan}} } \subset \ringhat{ V } _h ^{k, \mathrm{tan}}.
  \end{equation*}
  For reduced-stabilization penalties, where $\beta$ and $\alpha$
  project onto $ \widehat{ W } _h ^{k-1, \mathrm{nor}} $ and
  $ \widehat{ W } _h ^{k, \mathrm{tan}} $, respectively, strong
  multisymplecticity holds unconditionally.
\end{theorem}

\begin{proof}
  The proof of multisymplecticity is similar to that of \citep[Lemma
  2.8]{BaCaSt2024}, which establishes symmetry of the condensed
  bilinear form for the method in \cref{ex:bcs}.

  Let
  $ ( v _i , \overline{ \tau } ^{\mathrm{tan}} _i , \overline{ \eta }
  _i ^{\mathrm{nor}} , \widehat{ v } _i ^{\mathrm{nor}} , \widehat{ v
  } _i ^{\mathrm{tan}} ) $ be variations of a solution to
  \eqref{eq:bcs} for $ i = 1, 2 $, meaning that they satisfy
  \eqref{eq:bcs} except with $ \partial F / \partial u _h $ replaced
  by the linearization $ ( \partial ^2 F / \partial u _h ^2 ) v _i $
  on the right-hand side of \eqref{eq:bcs_v}. We wish to show that
  \begin{equation*}
    \langle \widehat{ \tau } _1 ^{\mathrm{tan}} , \widehat{ v } _2 ^{\mathrm{nor}} \rangle _{ \partial \mathcal{T} _h } + \langle \widehat{ v } _1 ^{\mathrm{tan}} , \widehat{ \eta } _2 ^{\mathrm{nor}} \rangle _{ \partial \mathcal{T} _h } = \langle \widehat{ \tau } _2 ^{\mathrm{tan}} , \widehat{ v } _1 ^{\mathrm{nor}} \rangle _{ \partial \mathcal{T} _h } + \langle \widehat{ v } _2 ^{\mathrm{tan}} , \widehat{ \eta } _1 ^{\mathrm{nor}} \rangle _{ \partial \mathcal{T} _h } ,
  \end{equation*}
  or equivalently,
  \begin{equation*}
    \langle \widehat{ \tau } _1 ^{\mathrm{tan}} , \widehat{ v } _2 ^{\mathrm{nor}} \rangle _{ \partial \mathcal{T} _h } - \langle  \widehat{ \eta } _1 ^{\mathrm{nor}}, \widehat{ v } _2 ^{\mathrm{tan}} \rangle _{ \partial \mathcal{T} _h } = \langle \widehat{ \tau } _2 ^{\mathrm{tan}} , \widehat{ v } _1 ^{\mathrm{nor}} \rangle _{ \partial \mathcal{T} _h } - \langle \widehat{ \eta } _2 ^{\mathrm{nor}}, \widehat{ v } _1 ^{\mathrm{tan}} \rangle _{ \partial \mathcal{T} _h }  .
  \end{equation*}
  First, observe that
  \begin{align*}
    \langle \widehat{ \tau } _1 ^{\mathrm{tan}} , \widehat{ v } _2 ^{\mathrm{nor}} \rangle _{ \partial \mathcal{T} _h }
    &= \langle \widehat{ \tau } _1 ^{\mathrm{tan}} , \widehat{ v } _2 ^{\mathrm{nor}} - v _2 ^{\mathrm{nor}} \rangle _{ \partial \mathcal{T} _h } + \langle \widehat{ \tau } _1 ^{\mathrm{tan}} , v _2 ^{\mathrm{nor}} \rangle _{ \partial \mathcal{T} _h } \\
    &= \bigl\langle \overline{ \tau } _1 ^{\mathrm{tan}} - \beta ( \widehat{ v } _1 ^{\mathrm{nor}} - v _1 ^{\mathrm{nor}} ) , \widehat{ v } _2 ^{\mathrm{nor}} - v _2 ^{\mathrm{nor}} \rangle _{ \partial \mathcal{T} _h } + \langle \widehat{ \tau } _1 ^{\mathrm{tan}} , v _2 ^{\mathrm{nor}} \rangle _{ \partial \mathcal{T} _h },\\
    &= -\bigl\langle \beta ( \widehat{ v } _1 ^{\mathrm{nor}} - v _1 ^{\mathrm{nor}} ) , \widehat{ v } _2 ^{\mathrm{nor}} - v _2 ^{\mathrm{nor}} \bigr\rangle _{ \partial \mathcal{T} _h } + \langle \widehat{ \tau } _1 ^{\mathrm{tan}} , v _2 ^{\mathrm{nor}} \rangle _{ \partial \mathcal{T} _h },
  \end{align*}
  where the last line uses the fact that $ \overline{ \tau } _1 ^{\mathrm{tan}} $ is orthogonal to $ \widehat{ v } _2 ^{\mathrm{nor}} - v _2 ^{\mathrm{nor}} $ by \eqref{eq:bcs_taubar}. Similarly,
  \begin{equation*}
    \langle \widehat{ \eta } _1 ^{\mathrm{nor}} , \widehat{ v } _2 ^{\mathrm{tan}} \rangle _{ \partial \mathcal{T} _h } = -\bigl\langle \alpha ( \widehat{ v } _1 ^{\mathrm{tan}} - v _1 ^{\mathrm{tan}} ) , \widehat{ v } _2 ^{\mathrm{tan}} - v _2 ^{\mathrm{tan}} \bigr\rangle _{ \partial \mathcal{T} _h } + \langle \widehat{ \eta } _1 ^{\mathrm{nor}} , v _2 ^{\mathrm{tan}} \rangle _{ \partial \mathcal{T} _h } .
  \end{equation*}
  Subtracting these and applying the linearization of
  \eqref{eq:bcs_v}, we get
  \begin{multline*}
    \langle \widehat{ \tau } _1 ^{\mathrm{tan}} , \widehat{ v } _2 ^{\mathrm{nor}} \rangle _{ \partial \mathcal{T} _h } - \langle \widehat{ \eta } _1 ^{\mathrm{nor}} , \widehat{ v } _2 ^{\mathrm{tan}} \rangle _{ \partial \mathcal{T} _h } = - \bigl\langle \beta ( \widehat{ v } _1 ^{\mathrm{nor}} - v _1 ^{\mathrm{nor}} ) , \widehat{ v } _2 ^{\mathrm{nor}} - v _2 ^{\mathrm{nor}} \bigr\rangle _{ \partial \mathcal{T} _h } + \bigl\langle \alpha ( \widehat{ v } _1 ^{\mathrm{tan}} - v _1 ^{\mathrm{tan}} ) , \widehat{ v } _2 ^{\mathrm{tan}} - v _2 ^{\mathrm{tan}} \bigr\rangle _{ \partial \mathcal{T} _h } \\
    + \biggl( \frac{ \partial ^2 F }{ \partial u _h ^2 } v _1 , v _2 \biggr) _{ \mathcal{T} _h } - ( \delta v _1 , \delta v _2 ) _{ \mathcal{T} _h } - ( \mathrm{d} v _1 , \mathrm{d} v _2 ) _{ \mathcal{T} _h } .
  \end{multline*}
  This is symmetric in the indices $ i = 1 , 2 $, so the local
  multisymplectic conservation law holds.

  Furthermore, the additional conditions imply that the numerical
  fluxes are single-valued, by \eqref{eq:bcs_vnor} and
  \eqref{eq:bcs_vtan}, and strong multisymplecticity follows by the
  same argument as in \cref{thm:strong.MS}. Note that
  $ \av{ \overline{ W } _h ^{k-1, \mathrm{tan}} } \subset \ringhat{V}
  _h ^{k-1, \mathrm{nor}} $ and
  $ \llbracket \overline{ W } _h ^{k, \mathrm{nor}} \rrbracket \subset
  \ringhat{V} _h ^{k, \mathrm{tan}} $ hold automatically, by the
  assumptions that
  $ \overline{ W } _h ^{k-1, \mathrm{tan}} (e) \subset \widehat{ W }
  _h ^{k-1, \mathrm{nor}} (e) $ and
  $ \overline{ W } _h ^{k, \mathrm{nor}} (e) \subset \widehat{ W } _h
  ^{k, \mathrm{tan}} (e) $ for $ e \in \partial \mathcal{T} _h $,
  which is why the conditions have fewer terms than in
  \cref{thm:ldg-h_ms} or \cref{thm:ip-h_ms}.
\end{proof}

\begin{corollary}
  The primal-hybrid nonconforming method of \citep{BaCaSt2024} is
  strongly multisymplectic.
\end{corollary}

Finally, we discuss some special instances in which the primal
variational principle \eqref{eq:bcs} can be written in the first-order
form of \cref{sec:multisymplectic_hybrid}. We have already shown, in
\cref{rmk:bcs_tan} that the method can be hybridized in terms of
global variables $ \widehat{ \sigma } _h ^{\mathrm{tan}} $ and
$ \widehat{ u } _h ^{\mathrm{tan}} $, rather than
$ \widehat{ u } _h ^{\mathrm{nor}} $ and
$ \widehat{ u } _h ^{\mathrm{tan}} $, under a strong conservativity
assumption. It remains to discuss the conditions under which the
primal equation \eqref{eq:bcs_v} can be written in the mixed
first-order form \eqref{eq:weakform_HL}.

\begin{proposition}
  \label{prop:bcs_first-order}
  Suppose that the spaces $ W _h ^{ k \pm 1 } $ satisfy the inclusions
  \begin{equation*}
    \delta W _h ^k \subset W _h ^{ k -1 } , \qquad ( W _h ^{ k -1 } ) ^{\mathrm{tan}} \subset \overline{ W } _h ^{k-1, \mathrm{tan}} , \qquad  \mathrm{d} W _h ^k \subset W _h ^{ k +1 } , \qquad ( W _h ^{ k +1 } ) ^{\mathrm{nor}} \subset \overline{ W } _h ^{k, \mathrm{nor}} .
  \end{equation*}
  Then solutions to the primal-hybrid problem \eqref{eq:bcs} are
  equivalent to those for the first-order mixed-hybrid problem with
  \eqref{eq:bcs_v} replaced by \eqref{eq:weakform_HL}, where
  $ \sigma _h = \delta u _h $ and $ \rho _h = \mathrm{d} u _h $.
\end{proposition}

\begin{proof}
  Integrating by parts, \eqref{eq:weakform_tau} becomes
  \begin{equation*}
    ( \sigma _h , \tau _h ) _{ \mathcal{T} _h } = ( \delta u _h , \tau _h ) _{ \mathcal{T} _h } - \langle \widehat{ u } _h ^{\mathrm{nor}} - u _h ^{\mathrm{nor}} , \tau _h ^{\mathrm{tan}} \rangle _{ \partial \mathcal{T} _h } .
  \end{equation*}
  Since
  $ ( W _h ^{ k -1 } ) ^{\mathrm{tan}} \subset \overline{ W } _h
  ^{k-1, \mathrm{tan}} $, the boundary term vanishes by
  \eqref{eq:bcs_taubar}, and since
  $ \delta W _h ^k \subset W _h ^{ k -1 } $, we conclude that this
  holds if and only if $ \sigma _h = \delta u _h $. Similarly,
  \eqref{eq:weakform_eta} holds if and only if
  $ \rho _h = \mathrm{d} u _h $, using \eqref{eq:bcs_etabar} and the
  assumed inclusions for $ W _h ^{ k + 1 } $.
\end{proof}

For the method of \citep{BaCaSt2024} discussed in \cref{ex:bcs}, these
inclusions hold if and only if $ r = 1 $. Recall that we cannot simply
take $ W _h ^1 (K) = \mathcal{P} _r \Lambda ^1 (K) $ for $ r > 1 $,
since the analysis requires the existence of a projection commuting
with both $ \delta $ and $ \mathrm{d} $ (and numerical experiments
reveal inconsistency otherwise on domains admitting corner
singularities). However, we could take $ W _h ^1 (K) $ to be the
order-$r$ Brenner--Sung--Mirebeau space \citep{BrSu2009,Mirebeau2012},
which consists of $ \mathcal{P} _r \Lambda ^1 (K) $ plus
\emph{harmonic} polynomial $1$-forms up to degree $ 2 r -1 $; see also
\citep[Section 2.2]{BaCaSt2024}. This yields a variant of the method
in \citep{BaCaSt2024} that satisfies the hypotheses of
\cref{prop:bcs_first-order} for arbitrary $r \geq 1$, since the
higher-degree harmonic polynomials are in the kernel of both
$ \delta $ and $ \mathrm{d} $.

This suggests a possible generalization where
$ W _h ^{ k \pm 1 } (K) = \mathcal{P} _{ r -1 } \Lambda ^k (K) $, and
where $ W _h ^k (K) $ consists of $ \mathcal{P} _r \Lambda ^k (K) $
plus harmonic polynomial $k$-forms up to sufficiently high degree. To
our knowledge, though, the construction of such an element with a
commuting projection when $ n > 2 $ and $ r > 1 $ remains an open
problem, even in the case $ n = 3 $. (\citet{BrSu2009} conjectured a
straightforward extension of the $ n = 2 $ element, but
\citet{Mirebeau2012} found a counterexample.) Another challenge when
$ n > 2 $ would be the construction of suitable penalties to ensure
consistency with singular solutions; when $ n = 3 $, both edge and
vertex singularities are possible on $ \partial \Omega $.

\section{Conclusion}

In this paper, we have developed a unified framework within which
hybrid FEEC methods may be applied to canonical systems of PDEs
involving differential forms, characterizing those methods whose
numerical traces satisfy a multisymplectic conservation law when
applied to Hamiltonian systems. This substantially expands on the
previous work of \citet{McSt2020} for de~Donder--Weyl systems, which
only involve the operators $ \operatorname{grad} $ and
$ \operatorname{div} $, to include the more general Hamiltonian
systems of \citet{Bridges2006} involving the exterior differential and
codifferential. Such methods may therefore be applied to a larger
class of problems---including Hodge--Dirac and Hodge--Laplace
problems, Stokes flow, and Maxwell's equations---for which
multisymplecticity corresponds to important local symmetry/reciprocity
relations satisfied by boundary values.

In a forthcoming paper, we plan to extend the framework to hybrid FEEC
semidiscretization of time-dependent Hamiltonian systems, with
applications including the Hodge wave equation and time-domain Maxwell
equations, along similar lines by which \citep{McSt2024} applied the
techniques of \citep{McSt2020} to semidiscretize time-dependent
de~Donder--Weyl systems. Another interesting direction for future
research is to explore the numerical consequences of (weak/strong)
multisymplecticity---and whether this leads to numerical advantages
for, say, HDG methods over the conforming AFW method when applied to
Hamiltonian systems.

\end{document}